\documentclass[10pt,reqno]{amsart}
\usepackage{amsmath}
\usepackage{amssymb}
\usepackage{amsthm}
\usepackage[usenames]{color}
\usepackage{eepic,epic}

\newtheorem{thm}{Theorem}[section]

\newtheorem{lem}[thm]{Lemma}

\theoremstyle{definition}
\newtheorem{defn}[thm]{Definition}
\theoremstyle{remark}
\newtheorem{rem}[thm]{Remark}
\numberwithin{equation}{section}

\newcommand{\ep}{\epsilon}

\begin{document}
\title[]
{The Lane-Emden system near the critical hyperbola on nonconvex domains}
\author{Woocheol Choi}
\address{School of Mathematics, Korea Institute for Advanced Study, Seoul 130-722, Korea}
\email{wchoi@kias.re.kr}

\maketitle
\begin{abstract}
In this paper we study the asymptotic behavior of minimal energy solutions to the Lane-Emden system $-\Delta u = v^p$ and $-\Delta v = u^q$ on bounded domains as the index $(p,q)$ approaches to the critical hyperbola from below. Precisely, we remove the convexity assumption on the domain in the result of Guerra \cite{G}. The main task is to get the uniform boundedness of the solutions near the boundary because it is difficulty to adapt the moving plane method for the system on nonconvex domains if $\max \{p,q\} > \frac{n+2}{n-2}$. For the purpose, we shall derive a contradiction by exploiting carefully the Pohozaev type identity if the maximum point approaches to the boundary. 
\end{abstract}
\section{introduction}

In this paper we consider the Lane-Emden system
\begin{equation}\label{eq-main}
\left\{ \begin{array}{ll} -\Delta u = v^{p}&\textrm{in}~\Omega,
\\
-\Delta v = u^q&\textrm{in}~\Omega,
\\
u>0,~v>0&\textrm{in}~\Omega,
\\
u=v=0&\textrm{on}~\partial \Omega,
\end{array}
\right.
\end{equation}
where $\Omega \subset \mathbb{R}^n\,(n \geq 3)$ is a smooth bounded domain and $(p,q) \in (0,\infty)^2$. The nonlinear system \eqref{eq-main} is a fundamental form among strongly coupled nonlinear systems and it has received a lot of interest from many authors. 
\

The existence theory of \eqref{eq-main} is related to the critical hyperbola given by the graph of $(p,q) \in (0,\infty)^2$ of the form
\begin{equation}\label{eq-cr-hy}
\frac{1}{p+1} + \frac{1}{q+1} = \frac{n-2}{n}.
\end{equation}
This notion was introduced by Cl\'ement et al. \cite{CFM} and van der Vorst \cite{V}. Under the condition that $(p,q)$ satisfies $pq >1$ and
\begin{equation}\label{pq-critic-1}
\frac{1}{p+1} + \frac{1}{q+1} > \frac{n-2}{n},
\end{equation}
Hulshof-Vorst \cite{HV} and Figueiredo-Felmer \cite{FF} showed the existence of a nontrivial solution to \eqref{eq-main} by applying a min-max method of Benci-Rabinowitz \cite{BR} with a further assumption $\min\{p, q\}>1$. Recently, the latter condition was relaxed by Bonheure-Moreira-Ramos \cite{BMR} to $pq \neq 1$. On the other hand, Mitidieri \cite{M1} obtained a Pohozaev type identity for \eqref{eq-main} which yields that if the domain is star-shaped and $(p,q)$ satisfies
\begin{equation*}
\frac{1}{p+1} + \frac{1}{q+1} \leq \frac{n-2}{n},
\end{equation*}
then there exists no nontrivial solution to \eqref{eq-main}.
\

Given the existence result of solutions, a fundamental problem is to find the explicit shape of solutions. An answer for this issue can be provided the moving plane method. It enables us to yield that if the domain is symmetric with respect to a direction, then so are the solutions (refer to \cite{QS}). On the other hand, it is expected that the solutions of \eqref{eq-main} may become singular as the index $(p,q)$ approaches to the critical hyperbola. Actually, Guerra \cite{G} showed that a sequence of least energy solutions to \eqref{eq-main} blows up as $(p,q)$ approaches to the critical hyperbola from below, and studied their asymptotic behavior. The result was obtained under an additional assumption that the domain is convex and $\min\{p,q\} \geq 1$. Up to our knowledge, this result was the only contribution on this issue for the system \eqref{eq-main}. 
\

The aim of our paper is to remove the convexity assumption in the  result of Guerra \cite{G}. Before explaining the convexity issue, it is worthwhile to note that when $p=q$ and $u=v$, the problem \eqref{eq-main} is reduced to the Lane-Emden-Fowler equation
\begin{equation}\label{eq-LEF}
\left\{\begin{array}{ll} -\Delta u_{\epsilon} = u_{\epsilon}^{\frac{n+2}{n-2}-\epsilon}&\quad \textrm{in}~\Omega,
\\
u_{\epsilon}=0&\quad \textrm{on}~\partial \Omega.
\end{array}
\right.
\end{equation}
As for this problem, the asymptotic behavior as $\ep \searrow 0$ has been studied very well through a series of papers. First, Han \cite{H} and Rey \cite{R} studied asymptotic behavior of the least energy solutions, and this result was extended to finite energy solutions in Bahri-Li-Rey \cite{BLR} and Rey \cite{R3} ($N \ge 4$ and $N = 3$, respectively). In addition, applying the Lyapunov-Schmidt reduction method, Rey \cite{R} constructed one-peak solutions to \eqref{eq-LEF}. Also multi-peak solutions were constructed by Bahri-Li-Rey \cite{BLR}, Rey \cite{R3} and Musso-Pistoia \cite{MP} (for $N \geq 3$). We remark that many crucial  techniques used for studying \eqref{eq-LEF} are difficult to be generalized for the system \eqref{eq-main}. 
\

Now we turn to the convexity issue. The convexity is needed both for the problems \eqref{eq-main} and \eqref{eq-LEF} if one applies the moving plane method to show the uniform boundedness of the solution $u_{\ep}$ with respect to $\ep >0$ on a neighborhood of the boundary $\partial \Omega$. This yields that the blow up point converges to an interior point, and then a further analysis using the Green's expression and the Pohozaev type identity can be conducted to get further informations of the asymptotic behavior. However, if the domain is not convex, it is difficult to apply the moving plane method in a direct way. When it comes to the single problem \eqref{eq-LEF}, Han \cite{H} overcame this difficulty by applying the Kelvin transform to \eqref{eq-LEF} on balls which touch the domain $\Omega$ by the boundary $\partial \Omega$. Unfortunately, such an idea does not work for the system \eqref{eq-main} if one of $p$ and $q$ is larger than $\frac{n+2}{n-2}$ (see page \,73 and Section 31.1 in \cite{QS}). This kind of difficulty was also observed previously in \cite{FLN} where the authors obtained the Gidas-Spruck type a priori estimate for \eqref{eq-main} with the convexity assumption.
\

The contribution of this paper is to prove the uniform boundedness near the boundary without using the moving plane method. For this aim, we shall make use of the Pohozaev type identity and the boundary behavior of Green's function of the Dirichlet Laplacian $-\Delta$ on the domain $\Omega$. Our approach may work for any smooth bounded domains without the convexity assumption.
\

To begin with, we fix a value $p \in [1,\infty)$ and find $q_{\epsilon}>1$ such that
\begin{equation}\label{eq-pq-e}
\frac{1}{p+1} + \frac{1}{q_{\epsilon}+1} = \frac{n-2}{n} + \epsilon\quad \textrm{for}~\epsilon >0.
\end{equation}
Then $(p,q_{\epsilon})$ is subcritical and approaches to the critical hyperbola as $\epsilon \searrow 0$. By the symmetry of \eqref{eq-main}, we may assume that $p \leq q_{\ep}$ without loss of generality. We then have $p \in \left[1, \frac{n+2}{n-2}\right]$. 
\begin{defn}\label{def-ME}
We consider a sequence of solutions $\{ (u_{\epsilon}, v_{\epsilon})\}_{\epsilon >0}$ such that each $(u_{\epsilon}, v_{\epsilon})$ is a solutions to \eqref{eq-main} with $q=q_{\epsilon}$. Then we say that $\{(u_{\epsilon}, v_{\epsilon})\}_{\ep >0}$ is of type $(ME)$ if the following condition holds;
\begin{equation}\label{eq-energy}
S_{\epsilon}(\Omega) = \frac{ \int_{\Omega} |\Delta u_{\epsilon}|^{\frac{p+1}{p}} dx}{ \left\|u_{\epsilon}\right\|_{L^{q_{\epsilon}+1}(\Omega)}^{\frac{p+1}{p}}} = S + o(1)\quad \textrm{as}\quad \epsilon \searrow 0,
\end{equation}
where $S$ is the best constant of the Sobolev embedding 
\begin{equation}\label{eq-hl}
\| u\|_{L^{q+1}(\mathbb{R}^n)} \leq S^{-\frac{p}{p+1}} \| \Delta u\|_{L^{\frac{p+1}{p}}(\mathbb{R}^n)}.
\end{equation}
\end{defn}
 We recall the result of Guerra \cite{G}. 
\begin{thm}[Guerra \cite{G}]\label{gue}Assume that $\Omega$ is a convex bounded domain and that $p \in \left[1, \frac{n+2}{n-2}\right]$. Let $\{ (u_{\ep}, v_{\ep}) \}_{\epsilon>0}$ is a sequence of solutions to \eqref{eq-main} of type $(ME)$.
 Then $(u_{\epsilon}, v_{\epsilon})$ blows up at a point $x_0 \in \Omega$ as $\ep$ goes to zero, up to a subsequence if necessary. In addition, the following holds;
\begin{enumerate}
\item The point $x_0$ is a critical point of the Robin function $H(x,x)$ (see \eqref{eq-g-decom}) when $p \in \left[\frac{n}{n-2}, \frac{n+2}{n-2}\right]$ and of the function $\widetilde{H}(\cdot,x_0)$ defined in \eqref{eq-tilh} when $p \in \left[1, \frac{n}{n-2}\right)$.
\item We have
\begin{equation*}
\left\{ \begin{aligned}
\lim_{\ep \rightarrow 0^{+}} \ep \|u_{\ep}\|_{L^{\infty}(\Omega)}^{\frac{N}{p(N-2)-2}+1} &= S^{\frac{1-pq}{p(q+1)}} \|U\|_{L^q (\mathbb{R}^n)}^q \|V\|_{L^p (\mathbb{R}^n)}^p  |H(x_0, x_0)|&\quad\textrm{if}~p > \frac{n}{n-2},
\\
\lim_{\ep \rightarrow 0^{+}} \ep \frac{\|u_{\ep}\|_{L^{\infty}(\Omega)}^{\frac{n}{n-2}+1}}{\log \|u_{\ep}\|_{L^{\infty}(\Omega)}} &= \frac{p+1}{n-2} a^{\frac{n}{n-2}}  S^{\frac{1-pq}{p(q+1)}} \|U\|_{L^q (\mathbb{R}^n)}^q |H(x_0, x_0)|&\quad\textrm{if}~p = \frac{n}{n-2},
\\
\lim_{\ep \rightarrow 0^{+}} \ep \|u_{\ep}\|_{L^{\infty}(\Omega)}^{p+1} &= S^{\frac{1-pq}{p(q+1)}} \|U\|_{L^q (\mathbb{R}^n)}^{q (p+1)} |\widetilde{H}(x_0, x_0)|&\quad\textrm{if}~p < \frac{n}{n-2}.
\end{aligned}
\right.
\end{equation*}
\end{enumerate} 
\end{thm}
\begin{rem}
The assumption $p \geq 1$ is due to a technical reason. Especially, the proof of Lemma 2.3  in \cite{G} for the global upper bound of the rescaled solutions works for $p \geq 1$ (see Lemma \ref{lem-gu} below). Removing this assumption be also an interesting problem.
\end{rem}
Now, under the assumption that $\Omega$ is any smooth bounded domain, we state our first main result.
\begin{thm}\label{thm-1} Let $p \in \left[ \frac{n}{n-2}, \frac{n+2}{n-2}\right]$. Consider a sequence of solutions $\{( u_{\epsilon}, v_{\epsilon})\}_{\epsilon >0}$ to \eqref{eq-main} of type $(ME)$. Then $(u_{\epsilon}, v_{\epsilon})$ are uniformly bounded near the boundary and blows up in an interior point of $\Omega.$
\end{thm}
As a first step for Theorem \ref{thm-1}, we prove in Lemma \ref{lem-lam} that the sequence of solutions blows up using the strict inequality of Green's functions (see \eqref{eq-ineq-gr}). Then, we are left to show that the blow up point $x_{\ep}$ (see \eqref{eq-max}) are uniformly away from the boundary. This is the main part of the paper. To show it, we shall argue by the contradiction. Suppose that the blow up point approaches to the boundary $\partial \Omega$ as $\epsilon \searrow 0$. Under this assumption, our strategy is to derive a contradiction from the following Pohozaev type identity (see Lemma \ref{lem-poho}) on an annulus centered at the blow up point;
\begin{equation}\label{eq-hn-15}
\begin{split}
&-\int_{\partial B(x_{\ep}, 2d_{\ep})} \left( \frac{\partial u_{\ep}}{\partial \nu} \frac{\partial v_{\ep}}{\partial x_j} + \frac{\partial v_{\ep}}{\partial \nu} \frac{\partial u_{\ep}}{\partial x_j} \right) d S_x + \int_{\partial B(x_{\ep}, 2d_{\ep})}(\nabla u_{\ep} \cdot \nabla v_{\ep}) \nu_j dS_x
\\
&\qquad\qquad= \frac{1}{p+1} \int_{\partial B(x_{\ep}, 2d_{\ep})} v_{\ep}^{p+1} \nu_j dS_x + \frac{1}{q_{\ep}+1} \int_{\partial B(x_{\ep}, 2d_{\ep})} u_{\ep}^{q_{\ep}+1} \nu_j dS_x,
\end{split}
\end{equation}
where $x_{\ep} \in \Omega$ is the maximum point of $(u_{\ep}, v_{\ep})$ defined in \eqref{eq-max} and $d_{\ep} = \textrm{dist}(x_{\ep}, \partial \Omega_{\ep})/4$. 
\

The left hand side will be estimated as a derivative of the function $H$ multiplied by some value depending on $\epsilon$, where $H$ is the regular part of Green's function on domain $\Omega$ (see \eqref{eq-g-decom}). Then, applying the estimate \eqref{eq-hn} for $H$;
\begin{equation}\label{eq-hn-16}
n_x \cdot \nabla_1 H(x,x) \geq C \mathbf{d}(x)^{-(n-1)}\quad \textrm{for}\quad x \in \Omega ~\textrm{with}~ \mathbf{d}(x) \leq c,
\end{equation}
we can get a lower bound of the left hand side. In order to estimate the left hand side of \eqref{eq-hn-15}, we need to represent the solution $(u_{\ep}, v_{\ep})$ on $\partial B(x_{\ep}, 2d_{\ep})$ in terms of Green's function with a relatively small error. It requires a careful analysis as $G(x,x_{\ep})$ for $x\in \partial B(x_{\ep}, 2d_{\ep})$ goes to infinity as $\ep$ goes to zero since $\lim_{\ep \rightarrow 0} d_{\ep} = 0$. We shall obtain  the desired estimates in Lemma \ref{lem-g1} and Lemma \ref{lem-g2}.  
\

On the other hand, we shall find a sharp upper bound of the right hand side using the decaying property of solutions near the blow up point. Occasionally, these two lower and upper bounds will lead to a contradiction. Due to a technical reason, we will handle the cases $p > \frac{n}{n-2}$ and $p = \frac{n}{n-2}$ separately in Section 5 and \mbox{Section 6.}
\

It requires more work to handle the case $p < \frac{n}{n-2}$ due to some technical difficulty. In particular, we need to handle the function $\widetilde{G}:\Omega \times \Omega \rightarrow \mathbb{R}$ defined by \begin{equation}\label{eq-GG}
\left\{ \begin{array}{ll} -\Delta_x \widetilde{G}(x,y) = G^p (x,y)& x \in\Omega,
\\
\widetilde{G}(x,y) = 0& x\in\partial \Omega.
\end{array}
\right.
\end{equation}
The $C^1$-regular part $\widetilde{H}$ of function $\widetilde{G}$ will be defined in \eqref{eq-h1-2} and it will play a similar role as the regular part $H$ of Green's function $G$. For our aim, we will derive the following estimate of $\widetilde{H}$ near the boundary of the domain:
\begin{thm}\label{thm-3} Assume that $n \geq 5$. There exists a constant $\alpha_n >0$ such that for $p \in [1,1+\alpha_n )$, the following statement is true;
\

\noindent \textbf{(A1)} There exist constants $C>0$ and $\delta >0$ such that, for $x \in \Omega$ with $\mathbf{d}(x):= \textrm{dist}(x,\partial \Omega) < \delta$ satisfies the inequality
\begin{equation}\label{eq-h2}
\partial_{n_{x}} \widetilde{H}(x,x) \geq C \mathbf{d}(x)^{1-(n-2)p}.
\end{equation} 
\end{thm}
The proof for this result is much more involved than that of the estimate \eqref{eq-hn-16} for the function $H$ which can be obtained directly by applying the Maximum principle. To prove the result, we shall rescale the function $\widetilde{H}$ in a suitable way and investigate its limit. Then we shall see that \textbf{(A1)} holds true provided a value of certain integration is not zero, where the value depends only on the values $n$ and $p$ (see Lemma \ref{lem-as}). We may check that it is true if $p=1$ and then a continuity argument will prove Theorem \ref{thm-3}. When the values of $n$ and $p$ are given, one might test that the value is nonzero in a numerical way and we guess that the estimate \eqref{eq-h2} of $\widetilde{H}$ given in \textbf{(A1)} is always true. 
\

Now we are ready to state our result for the case $p < \frac{n}{n-2}$.\begin{thm}\label{thm-2} Let $p \in \left[1, \frac{n}{n-2}\right)$. Assume that the estimate \eqref{eq-h2} of \textbf{(A1)}  holds. Then, any sequence of solutions $\{(u_{\epsilon}, v_{\epsilon})\}_{\epsilon >0}$ to \eqref{eq-main} satisfying \eqref{eq-energy} is uniformly bounded near the boundary and blows up in an interior point of $\Omega.$ 
\end{thm}
In order to prove this result, we shall follow the strategy used for Theorem \ref{thm-1}. However, some more careful analysis is required. As there, we shall obtain a contradiction from the Pohozaev type identity when we assume that the blow up point converges to a boundary point. In this case, it is suitable to write the Pohozaev type identity \eqref{eq-hn-15} in the following way
\begin{equation}\label{eq-hn-16}
\begin{split}
&\left[-\int_{\partial B(x_{\ep}, 2d_{\ep})} \left( \frac{\partial u_{\ep}}{\partial \nu} \frac{\partial v_{\ep}}{\partial x_j} + \frac{\partial v_{\ep}}{\partial \nu} \frac{\partial u_{\ep}}{\partial x_j} \right) d S_x + \int_{\partial B(x_{\ep}, 2d_{\ep})}(\nabla u \cdot \nabla v) \nu_j dS_x \right.
\\
&\qquad\left.-\frac{1}{p+1} \int_{\partial B(x_{\ep}, 2d_{\ep})} v_{\ep}^{p+1} \nu_j dS_x\right] = \frac{1}{q_{\ep}+1} \int_{\partial B(x_{\ep}, 2d_{\ep})} u_{\ep}^{q_{\ep}+1} \nu_j dS_x,
\end{split}
\end{equation}
First, to get a sharp estimate of $u_{\ep}(y)$ for $y \in \partial B(x_{\ep}, 2d_{\ep})$ for computing the left hand side, we will also need to know a sharp estimate of the value $v_{\ep}(x)$ for all $x \in \Omega$. These estimates will be achieved through Lemma \ref{lem-v} and Lemma \ref{lem-u}. Injecting these estimates into \eqref{eq-hn-16}, we will estimate the left hand side as a multiple of the derivative $\partial_{x_j} \widetilde{H}(x,x_{\ep})|_{x=x_{\ep}}$. Then, by applying the lower bound \eqref{eq-h2} we will get a sharp lower bound of the value \eqref{eq-hn-16}. On the other hand, an upper bound of the value \eqref{eq-hn-16} will be obtained from the right hand side using the decay estimate of $u_{\ep}$. Then, those upper and lower bounds will lead a contradiction again. 
\

Once we know that the blow up point converges to an interior point, then the argument of Guerra \cite{G} can be applied to investigate the further detail on the blow up solutions. Hence the result of Theorem \ref{gue} holds without the convexity assumption on the domain $\Omega$ for $p \in \left[ \frac{n}{n-2}, \frac{n+2}{n-2}\right]$ and with an additional assumption for $p \in \left[1, \frac{n}{n-2}\right)$. 
\

We remark that when $p=1$, the problem \eqref{eq-main} is reduced to the biharmonic equation
\begin{equation*}
\left\{\begin{split}
\Delta^2 u = u^q &\quad \textrm{in}~\Omega,
\\
u >0&\quad \textrm{in}~\Omega,
\\
\Delta u = u =0 &\quad \textrm{on}~\partial \Omega.
\end{split}
\right.
\end{equation*}
As for this problem, the asymptotic behavior of the least energy solutions as $q \nearrow \frac{n+4}{n-4}$ was studied first by Chou-Geng \cite{CG} with the convexity assumption on the domain. Later, Ben Ayed and El Mehdi \cite{BE} removed the convexity assumption. Our argument is similar but different to that of \cite{BE} in the point that they used the Pohozaev type identity implicitly and the inequality \eqref{eq-h2} with $p=1$ while we use the Pohozev type identity in a more direct way. In addiction, up to author's best knowledge, the inequality \eqref{eq-h2} even for $p=1$ is first proved rigorously in this paper.  
\

 Before finishing this section, we mention the Brezis-Nirenberg type problem 
\begin{equation}\label{eq-BN}
\left\{ \begin{array}{ll} -\Delta u = v^{p}+\lambda v&\textrm{in}~\Omega,
\\
-\Delta v = u^q + \mu u&\textrm{in}~\Omega,
\\
u>0,~v>0&\textrm{in}~\Omega,
\\
u=v=0&\textrm{on}~\partial \Omega,
\end{array}
\right.
\end{equation}
where $(p,q)$ satisfies  the relation \eqref{eq-cr-hy} and $\lambda >0$ and $\mu >0$. Hulshof-Mitidieri-Vorst \cite{HMV} found nontrivial solutions to \eqref{eq-BN} for $0 < \lambda \mu < \lambda_1 (\Omega)^2$ where $\lambda_1 (\Omega)$ is the first eigenvalue of $-\Delta$ on $\Omega$ with the Dirichlet boundary condition. Guerra \cite{G} studied also the asymptotic behavior of energy minimizing solution as $(\lambda, \mu) \rightarrow (0,0)$ when $\Omega$ is a bounded convex domain. It would be not difficult to modify our arguments in this paper to remove the convexity assumption also for the problem \eqref{eq-BN}.
\ 

This paper is organized as follows. In Section 2, we are concerned about the properties of Green's function along with the related function $\widetilde{G}$ and its regular part $\widetilde{H}$. We shall obtain a sharp estimate of Green's function in Lemma \ref{lem-h-asym} which will be used to prove Theorem \ref{thm-3}. In Section 3, we show that a sequence of the minimal energy solutions should blow up and that the blow up point cannot approach to the boundary too fast (see Lemma \ref{lem-bb}). Also we shall recall the global $L^{\infty}$ upper estimate of Guerra \cite{G} for the blow up solutions and the Pohozaev type identity. In Section 4, we will obtain a sharp estimate of the function $v_{\ep}$ on an annulus centered at the blow up point for all $p \in \left[1, \frac{n+2}{n-2}\right)$. Section 5 is devoted to prove Theorem \ref{thm-1} for the case $p \in \left( \frac{n}{n-2}, \frac{n+2}{n-2}\right)$. First we shall obtain a sharp estimate of the solution $u_{\ep}$ near the blow up point, and then derive a contradiction from the Pohozaev type identity if the blow up point approaches to the boundary. These arguments will be modified in Section 6 to prove Theorem \ref{thm-1} for the case $p = \frac{n}{n-2}$. Section 7 is aimed to prove Theorem \ref{thm-2} concerning the case $p \in \left[1, \frac{n}{n-2}\right)$. To handle that case, a more careful analysis will be conducted to estimate the solution $u_{\ep}$ on an annulus centered at the blow-up point. Then we shall derive a contradiction again from the Pohozaev type identity if the blow up point approaches to the boundary. In Section 8 we shall prove Lemma \ref{lem-as} and Theorem \ref{thm-3} concerning the properties of the function $\widetilde{H}$. In Appendix A, we shall give the proof of Lemma \ref{lem-h-asym} on the sharp estimate of Green's function.

\bigskip
\noindent \textbf{Notations.}

\medskip
\noindent Here we list some notations which will be used throughout the paper.

\noindent - $\{(u_{\epsilon}, v_{\epsilon})\}_{\epsilon >0}$ always represent a sequence of solutions to \eqref{eq-main} with $(p,q_{\epsilon})$ satisfying \eqref{eq-pq-e} and the the minimal energy type condition \eqref{eq-energy}.

\noindent - $C > 0$ is a generic constant that may vary from line to line.

\noindent - For $k \in \mathbb{N}$ we denote by $B^k (x_0,r)$ the  ball $\{ x \in \mathbb{R}^{k}: |x-x_0| < r \}$ for each $x_0 \in \mathbb{R}^k$ and $r>0$.

\noindent - For $x \in \Omega$ we denote by $\textrm{dist}(x,\partial \Omega)$ the distance from $x$ to $\partial \Omega$ and we denote $\mathbf{d}(x):= \textrm{dist}(x, \partial \Omega)$.

\noindent - For a domain $D \subset \mathbb{R}^n$, the map $\nu = (\nu_1, \cdots, \nu_n): \partial D \to \mathbb{R}^n$ denotes the outward pointing unit normal vector on $\partial D$.

\noindent - $dS$ stands for the surface measure.

\noindent - $|S^{n-1}| = 2\pi^{n/2}/\Gamma(n/2)$ denotes the Lebesgue measure of $(n-1)$-dimensional unit sphere $S^{n-1}$.

\noindent - For $f : \mathbb{R}^n \times \mathbb{R}^n \rightarrow \mathbb{R}$ we denote $\nabla_1 f (x,y) :=\left.\nabla_a f (a,y)\right|_{a=x}$ .


\section{Preliminary results}
In this section we are concerned with Green's function and the related function $\widetilde{G}$ and its regular part $\widetilde{H}$. The property of Green's function $G$ and its regular part $H$ is important throughout the paper and the function $\widetilde{G}$ with its regular part $\widetilde{H}$ is essential in Section 6 where we treat the case $p \in \left[1, \frac{n}{n-2}\right)$. 
\subsection{Green's function and its regular part}
Let $G$ be Green's function of $-\Delta$ on $\Omega$ with  the Dirichlet boundary condition. It is divided into a singular part and a regular part as
\begin{equation}\label{eq-g-decom}
G(x,y) = \frac{c_n}{|x-y|^{n-2}} - H(x,y),
\end{equation}
where $c_n = {1}/{(n-2)|S_{n-1}|}$ and the regular part $H: \Omega \times \Omega \rightarrow \mathbb{R}$ is the function such that
\begin{equation*}
\left\{\begin{array}{ll}
-\Delta_x H(x,y) = 0&\quad x \in \Omega,
\\
H(x,y) = G(x,y)&\quad x \in \partial \Omega.
\end{array}\right.
\end{equation*}
Let $\mathbf{d}(x)= \textrm{dist}(x,\partial \Omega)$ for $x \in \Omega$. Take a small constant $\delta >0$. Then, for any $x$ with $d(x) < \delta$, we can find a unique direction $n_x \in S^{n-1}$ such that $x + \mathbf{d}(x) n_x \in \partial \Omega$ and we set $x^{*} = x + 2 \mathbf{d}(x) n_x$. We recall the following result. 
\begin{lem}[Rey \cite{R2}]\label{lem-Rey} The following estimates hold:
\begin{equation}\label{eq-h1-0}
H(x,y) = \frac{c_n}{|x-y^*|^{n-2}} + o (\mathbf{d}(y)^{-(n-2)})
\end{equation}
and
\begin{equation}\label{eq-h1-1}
\nabla_y H(y,x)=\nabla_y H(x,y) = -\frac{(n-2) c_n (x-y^*)}{|x-y^*|^{n}} + o (\mathbf{d}(y)^{-(n-1)}),
\end{equation}
where the notation $o(\cdot)$ means that
\[\lim_{\mathbf{d}(y)\rightarrow 0} \frac{o \left( \mathbf{d}(y)^{-k}\right)}{ \mathbf{d}(y)^{-k}}=0\quad \textrm{for}~ k= n-2~ \textrm{or}~ n-1.\]
By taking $y=x$ in \eqref{eq-h1-1}, we can find a small constant $c>0$ such that 
\begin{equation}\label{eq-hn}
n_x \cdot \nabla_1 H(x,x) \geq \left[ \frac{(n-2)c_n}{2}\right] \mathbf{d}(x)^{-(n-1)}\quad \textrm{for}\quad x \in \Omega ~\textrm{with}~ \mathbf{d}(x) \leq c,
\end{equation}
where we have denoted $\nabla_1 H(x,x):= \nabla_x H(x,y)|_{y=x}$.
\end{lem}
The estimate \eqref{eq-hn} will be essential in the proof of Theorem \ref{thm-1}.
In the next subsection, we will define the function $\widetilde{G}$ in terms of the function $G$ for the case $p \in \left[1, \frac{n}{n-2}\right)$ and define its $C^1$ regular part $\widetilde{H}$. In proving Theorem \ref{thm-3} in Section 8, we will need to  have sharp upper estimates of the values of $H(x,y)$ and $\nabla_x H(x,y)$ for all $(x,y) \in \Omega \times \Omega$. For this reason, we obtain an improved version of Lemma \ref{lem-Rey} in the following lemma. \begin{lem}\label{lem-h-asym} 
For $(x,y) \in \Omega \times \Omega$ we have
\begin{equation}\label{eq-h1-4}
H(x,y) = \frac{c_n}{|x-y^*|^{n-2}} + O \left( \frac{\mathbf{d}(y)}{|x-y^*|^{n-2}}\right),
\end{equation}
and
\begin{equation}\label{eq-h1-5}
\nabla_x H(x,y) = - \frac{(n-2)c_n (x-y^*)}{|x-y^*|^{n}} + O \left( \frac{\mathbf{d}(y)}{\mathbf{d}(x)|x-y^*|^{n-2}}\right),
\end{equation}
where the notation $O$ means that there is a constant $C>0$ such that
\[ \left| O( f(x,y))\right| \leq C |f(x,y)| \quad \textrm{for all}~ (x,y) \in \Omega \times \Omega.\]
\end{lem}
\begin{proof}The proof is deferred to Appendix A. 
\end{proof}

For later use we formulate the above result as follows:
\begin{equation}\label{eq-h1-3}
\begin{split}
G(x,y) &=\frac{c_n}{|x-y|^{n-2}} -H(x,y)
\\
& = \frac{c_n}{|x-y|^{n-2}} - \frac{c}{|x-y^*|^{n-2}} + O \left( \frac{\mathbf{d}(y)}{|x-y^*|^{n-2}}\right)\quad \forall ~(x,y) \in \Omega \times \Omega.
\end{split}
\end{equation}
\subsection{The function $\widetilde{G}$ and its regular part  $\widetilde{H}$} In order to study our problem for the case $p \in \left[1, \frac{n}{n-2}\right)$, we need to consider the function $\widetilde{G}: \Omega \times \Omega \rightarrow \mathbb{R}$ defined by
\begin{equation}\label{eq-h1-6}
\begin{cases} 
-\Delta_x \widetilde{G}(x,y)= G^p (x,y)&x \in \Omega,
\\
\quad\quad\widetilde{G}(x,y) =0& x \in \partial \Omega.
\end{cases}
\end{equation}
We define its $C^1$ regular part $\widetilde{H}: \Omega \times \Omega \rightarrow \mathbb{R}$ by
\begin{equation}\label{eq-h1-2}
\widetilde{H}(x,y) =\left\{\begin{array}{ll} \frac{\alpha_1}{|x-y|^{p(n-2)-2}}- \widetilde{G}(x,y),& \quad \textrm{for}~ p \in \left[1, \frac{n-1}{n-2}\right),
\\
 \frac{\alpha_1}{|x-y|^{p(n-2)-2}} -\frac{\alpha_2 H(x,y)}{|x-y|^{[(n-2)p-n]}} - \widetilde{G}(x,y), &\quad \textrm{for}~ p \in \left[ \frac{n-1}{n-2}, \frac{n}{n-2}\right),
\end{array}
\right.
\end{equation}
where
\begin{equation}\label{eq-alpha}
\alpha_1 = \frac{c_n^p}{[(n-2)p-2][(n-2)p-n]}\quad \textrm{and}\quad 
\alpha_2= \frac{pc_n^{p-1}}{[(n-2)p-n][(n-2)p-2n+2]}.
\end{equation}
These functions then actually have the $C^1$ regularity as the following lemma shows.
\begin{lem} For each $y \in \Omega$, the function $x \in \Omega \rightarrow \widetilde{H}(x,y)$ is contained in $C^1_{loc}(\Omega)$.
\end{lem}
\begin{proof}By a basic regularity theory, it is enough to check that the function $x \in \Omega \rightarrow (-\Delta_x) \widetilde{H}(x,y)$ is contained in $L^{n+\eta}_{loc} (\Omega)$ for some $\eta >0$. For this aim, we begin with the computation 
\[ (-\Delta)|x|^{-\alpha} = \alpha (n-\alpha -2) |x|^{-(\alpha +2)}\quad \textrm{for}\quad x \in \mathbb{R}^n \setminus \{0\}\]
for each $\alpha \neq n-2$ . Using this and \eqref{eq-h1-2} along with \eqref{eq-h1-6} for $x \in \Omega$ we find that
\begin{equation}\label{eq-tilh}
(-\Delta_x)\widetilde{H}(x,y) =\left\{ 
\begin{array}{ll}
\frac{c_n^p}{|x-y|^{(n-2)p}}- G^p (x,y) &\textrm{if}~p \in \left[1, \frac{n-1}{n-2}\right)
\\
\begin{split}
& \frac{c_n^p}{|x-y|^{(n-2)p}} -G^p (x,y) - p H(x,y) c_n^{p-1} \frac{1}{|x-y|^{(n-2)(p-1)}}
\\
&\quad - \frac{2p \nabla_x H(x,y) c_n^{p-1}}{[(n-2)p-2n+2]} (-1)(x-y) |x-y|^{-(n-2)(p-1)}.
\end{split}& \textrm{if}~p \in \left[ \frac{n-1}{n-2}, \frac{n}{n-2}\right).
\end{array}\right.
\end{equation}
In addition, the boundary value of $\widetilde{H}(x, y)$ for $x \in \partial \Omega$ is computed from \eqref{eq-h1-2} as
\begin{equation*}
-\widetilde{H} (x,y)  = \left\{ \begin{array}{ll}
\alpha_1 |x-y|^{-[(n-2)p -2]} &\textrm{if}~p \in \left[1, \frac{n-1}{n-2}\right)
\\
\alpha_1 |x-y|^{-[(n-2)p -2]} -\alpha_2  H(x,y)c_n^{p-1}|x-y|^{-[(n-2)(p-1)-2]} &\textrm{if}~p \in \left[\frac{n-1}{n-2}, \frac{n}{n-2}\right).
\end{array}
\right.
\end{equation*}
Given the above information, we shall now finish the proof for the two cases $p < \frac{n-1}{n-2}$ and $p \geq \frac{n-1}{n-2}$ separately. \

\noindent \textbf{Case 1}. Assume that $p \in \left[1, \frac{n-1}{n-2}\right)$.
\

 Inserting \eqref{eq-g-decom} into \eqref{eq-tilh} we have
\begin{equation}\label{eq-c1-1}
\begin{split}
-\Delta_x \widetilde{H}(x,y) &= \frac{c_n^p}{|x-y|^{(n-2)p}} - \left( \frac{c_n}{|x-y|^{n-2}} - H(x,y) \right)^p 
\\
&\leq \frac{C_y}{|x-y|^{(n-2)(p-1)}},
\end{split}
\end{equation}
where $C_y >0$ is a constant depending on $y$. We can check that $(n-2)(p-1) < 1$ since $p < \frac{n-1}{n-2}$. Thus we can deduce from \eqref{eq-c1-1} that $(-\Delta_x)\widetilde{H}(x,y) \in L_{loc}^{n+\eta} (\Omega)$ for some $\eta >0$.
\

\noindent \textbf{Case 2.} Assume that $p \in \left[ \frac{n-1}{n-2}, \frac{n}{n-2}\right).$

Plugging \eqref{eq-g-decom} into \eqref{eq-tilh} we find
\begin{equation}\label{eq-cl-2}
\begin{split}
\Delta_x \widetilde{H}(x,y)&= \left( \frac{c_n}{|x-y|^{n-2}}- H(x,y)\right)^p - \frac{c_n^p}{|x-y|^{(n-2)p}} + pH(x,y) c_{n}^{p-1} \frac{1}{|x-y|^{(n-2)(p-1)}}
\\
&\quad - \frac{2p \nabla_x H(x,y) c_n^{p-1}}{[(n-2)p-2n+2]} (-1)(x-y) |x-y|^{-(n-2)(p-1)}.
\end{split}
\end{equation}
By applying the Taylor formula of second order we have
\begin{equation}\label{eq-value}
\begin{split}
&\left( \frac{c_n}{|x-y|^{n-2}}- H(x,y)\right)^p - \frac{c_n^p}{|x-y|^{(n-2)p}} + pH(x,y) c_{n}^{p-1} \frac{1}{|x-y|^{(n-2)(p-1)}}
\\
&\qquad=H(x,y)^2 \int_0^1 \frac{1}{2 p (p-1)} \left( \frac{c_n}{|x-y|^{n-2}}- tH(x,y)\right)^{p-2} (1-t)^2 dt,
\end{split}
\end{equation}
For fixed $y \in \Omega$, we claim that the above value is uniformly finite for $x \in \Omega$. For this we remind that
$\sup_{x \in \Omega}H(x,y) < \infty$ and we take a constant $c>0$ small enough. Then one may see that \eqref{eq-value} is bounded by looking at the first formula for the case $|x-y| >c$ and  the second formula for the case $|x-y| \leq c$.
\

Keeping also in mind that $\sup_{x \in \Omega} |\nabla_x H(x,y)|< \infty$ for each fixed $y \in \Omega$, we can find a constant $C_1>0$ such that
\begin{equation*}
\frac{2p \nabla_x H(x,y)c_n^{p-1}}{[(n-2)(p-2)-2]} \frac{(x-y)}{|x-y|^{(n-2)(p-1)}} \leq C_1 |x-y|^{-(n-2)(p-1)+1}\quad \forall x \in \Omega.
\end{equation*}
Since $p< \frac{n}{n-2}$, we have
\[
(n-2)(p-1) -1 = (n-2)p -n -1 = 1-\delta\quad \textrm{for some $\delta >0$}.\]
Therefore we may estimate \eqref{eq-cl-2} as
\begin{equation*}
\left|-\Delta_x \widetilde{H}(x,y)\right| \leq C |x-y|^{-1 +\delta}\quad \forall x \in \Omega,
\end{equation*}
which implies that $\widetilde{H}(\cdot,y) \in L_{loc}^{n + \eta}(\Omega)$ for $\eta >0$ small enough. The lemma is proved.
\end{proof}
\begin{rem}In \cite{G}, the regular part of $\widetilde{G}$ is defined as 
$\widetilde{H}_0 (x,y) = \frac{\alpha_1}{|x-y|^{p(n-2)}} - \widetilde{G}(x,y)$ for any $p \in \left( 1, \frac{n}{n-2}\right)$. However it should be replaced by \eqref{eq-h1-2} for the case $p \geq \frac{n-1}{n-2}$. In fact, it was noted in  \cite{G} that $\widetilde{H}_0 (x,x) = \widetilde{H}(x,x)$, which leads to  $\frac{\partial}{\partial x} [\widetilde{H}_0 (x,x)] = \frac{\partial}{\partial x} [\widetilde{H}(x,x)]$. However, the function $\widetilde{H}_0 (x,y)$ is not symmetric in $x$ and $y$ variables. Hence we may not claim that $\left.\frac{\partial}{\partial x} \widetilde{H}_0 (x,y)\right|_{y=x} = \frac{1}{2} \frac{\partial}{\partial x} \widetilde{H}_0 (x,x)$ holds and the former one $\left.\frac{\partial}{\partial x} \widetilde{H}_0 (x,y)\right|_{y=x}$ is the required value in Theorem \ref{gue}.
\end{rem}
In the following lemma, we show that \eqref{eq-h2} holds true under some assumption not depending on the domains.
\begin{lem}\label{lem-as}\mbox{~}
\begin{enumerate}
\item
If $p \in \left[1, \frac{n-1}{n-2}\right)$, then \textbf{(A1)} is true if the following condition holds
\begin{equation}\label{eq-as-0}
\begin{split}
&\int_{\mathbb{R}^{n}_{+}} \left[ \frac{(1-z_n)}{|z-e_n|^{n}} - \frac{(1+z_n)}{|z+e_n|^n}\right]\left\{ \left( \frac{1}{|z-e_n|^{n-2}} - \frac{1}{|z+e_n|^{n-2}}\right)^p - \frac{1}{|z-e_n|^{(n-2)p}}\right\}dz
\\
& \neq (n-2)\alpha_1\left[\int_{\partial \mathbb{R}^n_{+}} \frac{2}{|(y-e_n)|^{(n-2)(p+1)}} dy - \int_{\mathbb{R}^{n}_{+}} \frac{2n}{|(y-e_n)|^{(n-2)p+n}} dy\right].
\end{split}
\end{equation}
\item If $p \in \left[ \frac{n-1}{n-2}, \frac{n}{n-2}\right]$, then \textbf{(A1)} is true if the following condition holds
\begin{equation}\label{eq-as-1}
\begin{split}
&\int_{\mathbb{R}^{n}_{+}} \left[ \frac{(1-z_n)}{|z-e_n|^{n}} - \frac{(1+z_n)}{|z+e_n|^n}\right]\left\{ \left( \frac{1}{|z-e_n|^{n-2}} - \frac{1}{|z+e_n|^{n-2}}\right)^p - \frac{1}{|z-e_n|^{(n-2)p}} \right.
\\
&\qquad  \left. + \frac{p}{|z-e_n|^{(n-2)(p-1)}|z+e_n|^{n-2}} +\frac{2p(n-2)}{[(n-2)(p-2)-2]} \frac{z+1}{|z+e_n|^{n}} \frac{(z-1)}{|z-e_n|^{(n-2)(p-1)}} \right\}dz
\\
& \neq (n-2)(\alpha_1 -\alpha_2) \left[\int_{\partial \mathbb{R}^n_{+}} \frac{2}{|(y-e_n)|^{(n-2)(p+1)}} dy - \int_{\mathbb{R}^{n}_{+}} \frac{2n}{|(y-e_n)|^{(n-2)p+n}} dy\right].
\end{split}
\end{equation}
\end{enumerate}
\end{lem}
The proof of this lemma will be deferred to Appendix A. There, we will also prove that \eqref{eq-as-0} is true for $p \in [1,1+\alpha_n]$ for a value $\alpha_n >0$ depending only on the dimension, which is exactly the content of Theorem \ref{thm-3}.
\section{Preliminary results on blow up}
In this section we obtain preliminary results for a sequence of the solutions $\{(u_{\ep}, v_{\ep})\}_{\ep >0}$ of type $(ME)$. We take a value $\lambda_{\epsilon}>0$ and a point $x_{\epsilon} \in \Omega$ such that
\begin{equation}\label{eq-max}
\lambda_{\epsilon}=\max_{x \in \Omega} \,\max \{ v_{\epsilon}^{\frac{p+1}{n}}(x), ~u_{\epsilon}^{\frac{q_{\epsilon}+1}{n}} (x) \} = \max \{ v_{\epsilon}^{\frac{p+1}{n}}(x_{\epsilon}),~u_{\epsilon}^{\frac{q_{\epsilon}+1}{n}}(x_{\epsilon})\}.
\end{equation}
We will prove that this value diverges to infinity as $\ep$ goes to zero in the following lemma. The proof will use the property of Green's function that we remind below.
\

For $Q \subset \mathbb{R}^n$ we use the notation $K_{(-\Delta_Q)^{-1}}: Q \times Q \rightarrow \mathbb{R}$ to denote Green's function of the Laplacian $-\Delta$ on $Q$ with the Dirichlet zero boundary condition. We also let $K_{(-\Delta)^{-1}}$ denote Green's function of the Laplacain on $\mathbb{R}^n$, i.e., 
\begin{equation*}
K_{(-\Delta)^{-1}}(x,y) = \frac{c_n}{|x-y|^{n-2}}.
\end{equation*}
Then, it is a classical fact that for any smooth subset $Q \subset \mathbb{R}^n$ with $Q \neq \mathbb{R}^n$, we have
\begin{equation}\label{eq-ineq-gr}
K_{(-\Delta_{Q})^{-1}}(x,y) < K_{(-\Delta)^{-1}}(x,y)\quad \textrm{for all} ~ (x,y) \in Q \times Q.
\end{equation}
\begin{lem}\label{lem-lam} We have $\lim_{\ep \rightarrow 0} \lambda_{\ep} = \infty$.
\end{lem}
\begin{proof}
In order to prove the lemma, we assume the contrary. Then there is a subsequence $\{\ep_{k}\}_{k \in \mathbb{N}}$ with $\lim_{k \rightarrow \infty} \ep_k =0$ we have $\sup_{k \in \mathbb{N}} \lambda_{\epsilon}< \infty$. This implies that the solutions $\{(u_{\epsilon_k}, v_{\epsilon_k})\}_{k \in \mathbb{N}}$ are uniformly bounded in $C^{2,\alpha}(\Omega) \times C^{2,\alpha}(\Omega)$ for some $\alpha \in (0,1)$ by the standard regularity theory. Hence $(u_{\epsilon_k}, v_{\epsilon_k})$ converges in $C^2 (\Omega) \times C^2 (\Omega)$ to a nontrivial solution $(u_0, v_0)$ of the equation
\begin{equation}\label{eq-lam-2}
\left\{ \begin{array}{ll} -\Delta u_0 = v_0^{p}&\textrm{in}~\Omega,
\\
-\Delta v_0 = u_0^{q}&\textrm{in}~\Omega,
\\
u_0 = v_0 = 0&\textrm{on}~\partial \Omega.
\end{array}
\right.
\end{equation}
On the other hand, by taking the limit $k \rightarrow \infty$ in \eqref{eq-energy} we get
\begin{equation*}
\| u_0 \|_{L^{q+1}(\Omega)} = S^{-\frac{p}{p+1}} \| \Delta u_0\|_{L^{\frac{p+1}{p}}(\Omega)}.
\end{equation*}
Let us set $w_0 : \Omega \rightarrow \overline{\mathbb{R}_{+}}$ by $w_0 (x) = (-\Delta_{\Omega}) u_0 (x)$ for $x \in \Omega$. Then $u_0 (x) = (-\Delta_{\Omega})^{-1} w_0 (x)$ for $x \in \Omega$ and so we have
\begin{equation}\label{eq-lam-1} \left\|(-\Delta_{\Omega})^{-1} w_0 \right\|_{L^{q+1}(\Omega)} = S^{-\frac{p}{p+1}} \left\|w_0\right\|_{L^{\frac{p+1}{p}}(\Omega)}.
\end{equation}
We extend the function $w_0$ to set $\widetilde{w}_0 : \mathbb{R}^n \rightarrow \overline{\mathbb{R}_{+}}$ by 
\begin{equation*}
\widetilde{w}_0 (x) = \left\{ \begin{array}{ll} w_0 (x)&\quad \textrm{for}~x \in \Omega,
\\
0&\quad \textrm{for}~ x \notin \Omega.
\end{array}
\right.
\end{equation*}
Then, using the inequality \eqref{eq-ineq-gr} and \eqref{eq-lam-1} we obtain the following estimate
\begin{equation*}
\begin{split}
S^{-\frac{p}{p+1}} \| \widetilde{w}_0\|_{L^{\frac{p+1}{p}}(\mathbb{R}^n)}&= S^{-\frac{p}{p+1}} \| w_0\|_{L^{\frac{p+1}{p}}(\Omega)} 
\\
&= \| (-\Delta_{\Omega})^{-1} w_0\|_{L^{q+1}(\Omega)} 
\\
& < \| (-\Delta_{\Omega})^{-1} \widetilde{w_0}\|_{L^{q+1}(\Omega)} <  \| (-\Delta)^{-1} \widetilde{w}_0\|_{L^{q+1}(\mathbb{R}^n)}.
\end{split}
\end{equation*}
However, this contradicts to the optimality of the constant $S^{-\frac{p}{p+1}}$ of the inequality \eqref{eq-hl}. Therefore it should hold that $\lim_{\epsilon \rightarrow 0} \lambda_{\epsilon} =\infty$. The lemma is proved.
\end{proof}
\begin{rem}
 In \cite{G} the author proved $\lambda_{\ep}$ as $\ep \rightarrow 0$ using the convexity assumption of $\Omega$ since a convex domain is a star-shaped domain for which the Pohozaev type identity of \cite{M1} can be applied to yield that \eqref{eq-lam-2} has no nontrivial solution. However, the blowing up of the sequence of solutions with the minimal energy type condition \eqref{eq-energy} can be deduced only using the condition \eqref{eq-energy} without the convexity assumption as the proof of the above lemma shows.
\end{rem}
For each $\ep >0$ we set $\Omega_{\epsilon}:= \lambda_{\epsilon}(\Omega-x_{\epsilon})$ and normalize the solutions as
\begin{equation*}
\widetilde{u}_{\epsilon}(x):= \lambda_{\epsilon}^{-\frac{n}{q_{\epsilon}+1}} u_{\epsilon}(\lambda_{\epsilon}^{-1} x + x_{\epsilon}),
\quad \textrm{and}\quad \widetilde{v}_{\epsilon}(x):= \lambda_{\epsilon}^{-\frac{n}{p+1}} v_{\epsilon}(\lambda_{\epsilon}^{-1} x + x_{\epsilon}),
\quad \textrm{for}~x \in \Omega_{\epsilon}.
\end{equation*}
Then it holds that
\begin{equation}\label{eq-ext-sol}
\left\{ \begin{array}{ll}
-\Delta \widetilde{u}_{\epsilon} = \widetilde{v}_{\epsilon}^{p}&\textrm{in}~\Omega_{\epsilon},
\\
-\Delta \widetilde{v}_{\epsilon} = \widetilde{u}_{\epsilon}^{q_{\epsilon}}&\textrm{in}~\Omega_{\epsilon},
\\
\widetilde{u}_{\epsilon} = \widetilde{v}_{\epsilon}=0&\textrm{on}~\partial\Omega_{\epsilon},
\end{array}
\right.
\end{equation}
and
\begin{equation*}
\max_{x \in \Omega_{\epsilon}} \{ \widetilde{u}_{\epsilon}(x),~\widetilde{v}_{\epsilon}(x)\} = 1 = \max \{ \widetilde{u}_{\epsilon}(0),~\widetilde{v}_{\epsilon}(0)\}.
\end{equation*}
In the next lemma, we obtain an estimate for the distance between the maximum point of the solutions and the boundary $\partial \Omega$.
\begin{lem}\label{lem-bb}
We have $\lim_{\epsilon \rightarrow 0} \lambda_{\epsilon} \textrm{dist}(x_{\epsilon}, \partial \Omega) = \infty$.
\end{lem}
\begin{proof}
The proof is similar to that of Lemma \ref{lem-lam}. As there, we assume the contrary. Then, up to a subsequence, we have $\lim_{\epsilon \rightarrow 0} \lambda_{\epsilon} \textrm{dist} (x_{\epsilon}, \partial \Omega) = l$ for some $l \in (0, \infty)$. This then implies that the extended domain $\Omega_{\epsilon}$ converges to a half space $\mathcal{H} = \left\{ x \in \mathbb{R}^n : \sum_{i=1}^{n} a_i x_i >0\right\}$ for some $(a_1,\cdots, a_n) \in \mathbb{R}^n \setminus \{0\}$ as $\ep \rightarrow 0$. Also, the normalized functions $(\widetilde{u}_{\epsilon}, \widetilde{v}_{\epsilon})$ converge to a nontrivial solution $(\overline{U}, \overline{V})$ of the problem
\begin{equation*}
\left\{ \begin{array}{ll} -\Delta \overline{U} = \overline{V}^{p}&\textrm{in}~\mathcal{H},
\\
-\Delta \overline{V} = \overline{U}^{q}&\textrm{in}~\mathcal{H},
\\
\overline{U}=\overline{V}=0&\textrm{on}~\partial \mathcal{H}.
\end{array}
\right.
\end{equation*}
and we know that $K_{(-\Delta_{\mathcal{H}})^{-1}} (x,y) < K_{(-\Delta)^{-1}}(x,y)$ from \eqref{eq-ineq-gr}. Then
we can obtain a contradiction as in the proof of Lemma \ref{lem-lam}. Thus the result of the lemma is true.
\end{proof}
We set $d_{\epsilon} := \frac{1}{4}\textrm{dist}(x_{\epsilon}, \partial \Omega)$ and $N_{\ep} = d_{\ep} \lambda_{\ep}$. Then we see from Lemma \ref{lem-bb} that
\begin{equation}\label{eq-de}
d_{\epsilon} = \frac{N_{\epsilon}}{\lambda_{\epsilon}}\quad \textrm{and}\quad \lim_{\ep \rightarrow 0} N_{\epsilon} = \infty.
\end{equation}
Remarkably, the fact that $N_{\ep} \rightarrow \infty$ as $\ep \rightarrow 0$ plays an important role in the proofs of Theorem \ref{thm-1} and Theorem \ref{thm-2}. 
\

Next we recall from \cite[Lemma 2.2]{G} the following result.
\begin{lem}[\cite{G}] There is a constant $C>0$ independent of $\epsilon>0$ such that $\lambda_{\epsilon}^{\epsilon} \leq C$ for all $\ep >0$.
\end{lem}
By this result we have 
\[ \lambda_{\ep}^{\frac{n}{q_{\ep}+1}} \leq C \lambda_{\ep}^{\frac{n}{q+1}}.\]
We shall use this inequality in many places of the proofs of our main results. 
\

By Lemma \ref{lem-bb} the domain $\Omega_{\epsilon}$ converges to $\mathbb{R}^n$ as $\epsilon$ goes to zero, and so the rescaled solution $(\widetilde{u}_{\epsilon}, \widetilde{v}_{\epsilon})$ converges in $C_{loc}^2 (\mathbb{R}^n)$ to a solution $(U, V)$ of the problem
\begin{equation}\label{eq-entire}
\left\{ \begin{array}{ll} -\Delta U = V^p&\quad \textrm{in}~\mathbb{R}^n,
\\
-\Delta V = U^q&\quad \textrm{in}~\mathbb{R}^n,
\\
U(y) >0,~ V(y) >0&\quad y \in \mathbb{R}^n,
\\
U(0)=1= \max_{x \in \mathbb{R}^n}U(x),~ U \rightarrow 0,~V \rightarrow 0&\quad\textrm{as}~|y| \rightarrow \infty.
\end{array}
\right.
\end{equation}
We recall the result of Chen-Li-Ou [CLO] that $U$ and $V$ are radially symmetric if $U \in L^{q+1}(\mathbb{R}^N)$, $V \in L^{p+1}(\mathbb{R}^N)$ and $p \geq 1$. In addition, Hulshof and Van der Vorst \cite{HV2} obtained the asymptotic behavior as follows. 
\begin{equation}\label{eq-decay}
\lim_{r \rightarrow \infty} r^{N-2} V(r) = a\quad \textrm{and}\quad \left\{ \begin{array}{ll} \lim_{r \rightarrow \infty}r^{N-2} U(r) = b&\quad \textrm{if}~ p > \frac{N}{N-2},
\\
\mbox{~}
\\
\lim_{r\rightarrow \infty} \frac{r^{N-2}}{\log r}U(r) =b&\quad \textrm{if}~p = \frac{N}{N-2},
\\
\mbox{~}
\\
\lim_{r \rightarrow \infty} r^{p(N-2)-2} U(r) =b&\quad \textrm{if}~ \frac{2}{N-2} < p < \frac{N}{N-2}.
\end{array}
\right.
\end{equation}
In the following sections, these sharp decaying rates will be used frequently in a combination with the following result.
\begin{lem}[\cite{G}]\label{lem-gu} There exists a constant $C>0$ such that
\begin{equation}\label{eq-bound}
\widetilde{u}_{\epsilon}(x) \leq C U(x) \quad \textrm{and}\quad \widetilde{v}_{\epsilon}(x) \leq CV(x),\quad \forall x \in \Omega_{\epsilon}\quad \forall \epsilon >0.
\end{equation}
\end{lem}
\begin{proof}
The proof is obtained through a combination of the Kelvin transform and a Moser iteration argument. We refer to \cite{G} for the detail.
\end{proof}
Let us define the following constants
\begin{equation}\label{eq-AUV}
A_V = \int_{\mathbb{R}^n} V^{p}(y) dy,\qquad A_U = \int_{\mathbb{R}^n} U^{q}(y) dy.
\end{equation}
We end this section with a local version of the Pohozaev type identity for the problem \eqref{eq-main}.
\begin{lem}\label{lem-poho}
Let $1 \leq j \leq n.$ Suppose that $(u,v) \in C^2 (\Omega) \times C^2 (\Omega)$ is a solution of \eqref{eq-main}. Then, for any open smooth subset $D \subset \Omega$, we have the following identity. 
\begin{equation}\label{eq-poho-1}
\begin{split}
&-\int_{\partial D} \left( \frac{\partial u}{\partial \nu} \frac{\partial v}{\partial x_j} + \frac{\partial v}{\partial \nu} \frac{\partial u}{\partial x_j} \right) d S_x + \int_{\partial D}(\nabla u \cdot \nabla v) \nu_j dS_x
\\
&\qquad\qquad= \frac{1}{p+1} \int_{\partial D} v^{p+1} \nu_j dS_x + \frac{1}{q+1} \int_{\partial D} u^{q+1} \nu_j dS_x,
\end{split}
\end{equation}
where $D$ is an open subset of $\Omega$.
\end{lem}
\begin{proof}
Multiplying \eqref{eq-main} by $\frac{\partial v}{\partial x_j}$  we get $-\Delta u \frac{\partial v}{\partial x_j} = v^p \frac{\partial v}{\partial x_j}$. Integrating this over the domain $D$ and using an integration by part, we get
\begin{equation}\label{eq-poho-2}
-\int_{\partial D} \frac{\partial u}{\partial \nu} \frac{\partial v}{\partial x_j} dS_x + \int_{D} \nabla u \cdot \frac{\partial \nabla v}{\partial x_j} dS_x = \frac{1}{p+1} \int_{\partial D} v^{p+1} \nu_j dS_x.
\end{equation}
Similarly we have $-\Delta v \frac{\partial u}{\partial x_j} = u^{q} \frac{\partial u}{\partial x_j}$ and
\begin{equation}\label{eq-poho-3}
-\int_{\partial D} \frac{\partial v}{\partial \nu} \frac{\partial u}{\partial x_j} dS_x + \int_{D} \nabla v \cdot \frac{\partial \nabla u}{\partial x_j} dS_x = \frac{1}{q+1} \int_{\partial D} v^{q+1}  \nu_j dS_x.
\end{equation}
Summing up \eqref{eq-poho-2} and \eqref{eq-poho-3} we get
\begin{equation*}
\begin{split}
&\frac{1}{p+1} \int_{\partial D}v^{p+1} \nu_j dS_x + \frac{1}{q+1} \int_{\partial D} u^{q+1} \nu_j dS_x
\\
&\qquad =-\int_{\partial D} \left(\frac{\partial u}{\partial \nu} \frac{\partial v}{\partial x_j} +\frac{\partial v}{\partial \nu} \frac{\partial u}{\partial x_j}\right) dS_x + \int_{D} \nabla v \cdot \frac{\partial \nabla u}{\partial x_j} dS_x+ \int_{D} \nabla u \cdot \frac{\partial \nabla v}{\partial x_j} dS_x
\\
&\qquad =-\int_{\partial D} \left( \frac{\partial u}{\partial \nu} \frac{\partial v}{\partial x_j} + \frac{\partial v}{\partial \nu} \frac{\partial u}{\partial x_j}\right) dS_x + \int_{\partial D}(\nabla u \cdot \nabla v) \nu_j dS_x,
\end{split}
\end{equation*}
where we applied an integration by parts in the second identity. This is the desired identity \eqref{eq-poho-1}. The lemma is proved.
\end{proof}
From Section 4 to Section 7, we shall always denote by $\{(u_{\ep}, v_{\ep})\}_{\ep>0}$ a sequence of the solutions of type $(ME)$, and also we shall keep using the notations $x_{\ep}$ and $\lambda_{\ep}$ defined in \eqref{eq-max} along with $d_{\epsilon} =\frac{N_{\ep}}{\lambda_{\ep}}$ defined in \eqref{eq-de}.

\section{Estimates for $v_{\ep}$ on the annulus}
In this section we prove a sharp estimate for $v_{\ep}$ and its derivatives on the annulus $\partial B(x_{\ep}, 2d_{\ep})$, which will be necessary for evaluating the left hand side of \eqref{eq-hn-15}. Although the \emph{a priori} assumption $d_{\ep} \rightarrow 0$ makes the analysis not easy, we shall get the desired estimate through a careful analysis. We state the following result.
\begin{lem}\label{lem-g10}Suppose $p \in\left(\frac{n}{n-2}, \frac{n+2}{n-2}\right]$. Assume that $\{(u_{\ep},v_{\ep})\}_{\ep >0}$ is a sequence of solutions to \eqref{eq-main} of type $(ME)$ and that $\lim_{\ep \rightarrow 0} d_{\ep} =0$. Then, for $x \in \partial B(x_{\ep}, 2d_{\ep})$ we have the estimates
\begin{equation}\label{eq-g1-0}
v_{\epsilon}(x) = A_U \lambda_{\epsilon}^{-\frac{n}{q_{\epsilon}+1}} G(x,x_{\epsilon}) + o (d_{\epsilon}^{-(n-2)} \lambda_{\epsilon}^{-\frac{n}{q+1}})
\end{equation}
and
\begin{equation}\label{eq-g1-8}
\nabla  v_{\epsilon}(x) = A_U \lambda_{\epsilon}^{-\frac{n}{q_{\epsilon}+1}} \nabla  G(x,x_{\epsilon}) + o (d_{\epsilon}^{-(n-1)} \lambda_{\epsilon}^{-\frac{n}{q+1}}).
\end{equation}
In addition, the $o$-notation is uniform with respect to $x \in \partial B(x_{\ep}, 2d_{\ep})$, i.e., it holds that
\begin{equation*}
\lim_{\ep \rightarrow 0} \sup_{x \in \partial B(x_{\ep}, 2d_{\ep})} \frac{|o (d_{\ep}^{-k}\lambda_{\ep}^{-\frac{n}{q+1}})|}{ (d_{\ep}^{-k}\lambda_{\ep}^{-\frac{n}{q+1}})} = 0\quad \textrm{for}~ k = n-1~\textrm{or}~n-2.
\end{equation*}
\end{lem}

\begin{proof}
As the function $(u_{\epsilon}, v_{\epsilon})$ is a solution to \eqref{eq-main}, we have
\begin{equation}\label{eq-g1-1}
\begin{split}
v_{\epsilon}(x) &= \int_{\Omega} G(x,y) u_{\epsilon}^{q}(y) dy 
\\
&= G(x,x_{\epsilon}) \left( \int_{\Omega}u_{\epsilon}^q (y) dy\right) + \int_{\Omega} [G(x,y) - G(x,x_{\epsilon})] u_{\epsilon}^{q} (y) dy.
\end{split}
\end{equation}
Given the estimates \eqref{eq-decay} and \eqref{eq-bound}, we may apply the dominated convergence theorem to yield
\begin{equation*} 
\lim_{\epsilon \rightarrow 0}\lambda_{\epsilon}^{\frac{n}{q_{\ep}+1}} \int_{\Omega} u_{\epsilon}^{q_{\ep}}(y) dy =\lim_{\epsilon \rightarrow 0} \int_{\Omega_{\epsilon}} \widetilde{u}_{\epsilon}^{q_{\ep}}(y) dy = \int_{\mathbb{R}^n} U^{q}(y) dy = A_U.
\end{equation*}
Using this and the fact that $G(x,x_{\ep}) = O(|x-x_{\ep}|^{-(n-2)}) = O (d_{\ep}^{-(n-2)})$ for $x \in \partial B(x_{\ep}, 2d_{\ep})$, we have
\begin{equation*}
G(x,x_{\ep}) \left( \int_{\Omega} u_{\ep}^{q_{\ep}}(y) dy\right) = \lambda_{\ep}^{-\frac{n}{q_{\ep}+1}} A_U G(x,x_{\ep}) + o (\lambda_{\ep}^{-\frac{n}{q+1}} d_{\ep}^{-(n-2)}).
\end{equation*}
Hence, to prove \eqref{eq-g1-20}, it is only left to estimate the last term of \eqref{eq-g1-1}. For this aim, we divide the term into three parts as follows:
\begin{equation}\label{eq-g1-11}
\int_{\Omega} [G(x,y) - G(x,x_{\epsilon})] u_{\epsilon}^{q_{\ep}} (y) dy = I_1 (x) + I_2 (x) + I_3 (x),
\end{equation}
where
\begin{equation*}
\begin{split}
I_1 (x)&:= \int_{B(x_{\epsilon}, d_{\epsilon})} [G(x,y) - G(x,x_{\epsilon})] u_{\epsilon}^{q_{\ep}} (y)dy,
\\
I_2 (x)&:= \int_{B(x_{\epsilon},4d_{\epsilon})\setminus B(x_{\epsilon},d_{\epsilon})}[G(x,y) - G(x,x_{\epsilon})] u_{\epsilon}^{q_{\ep}} (y) dy,
\\
I_3 (x)&:= \int_{\Omega \setminus B(x_{\epsilon}, 4d_{\epsilon})} [G(x,y) - G(x,x_{\epsilon})]u_{\epsilon}^{q_{\ep}} (y) dy.
\end{split}
\end{equation*}
We shall show that $I_1 (x)$, $I_2 (x)$, and $I_3 (x)$ are estimated as $o \left( d_{\ep}^{-(n-2)} \lambda_{\ep}^{-\frac{n}{q+1}}\right)$ and their derivatives $\nabla I_1 (x)$, $\nabla I_2 (x)$, and $\nabla I_3 (x)$ are estimated as $o \left( d_{\ep}^{-(n-1)} \lambda_{\ep}^{-\frac{n}{q+1}}\right).$
\

\noindent \emph{Estimate of $I_1$.} Since $|x-x_{\epsilon}|=2d_{\epsilon}$ we have $|x-y| \geq d_{\epsilon}$ for $y ~\in B(x_{\epsilon}, d_{\epsilon})$.  Thus
\begin{equation*}
|\nabla_y G(x,y)| \leq Cd_{\epsilon}^{-(n-1)}\quad \textrm{and}\quad |\nabla_x \nabla_y G(x,y)| \leq C d_{\ep}^{-n}\qquad \forall y \in B(x_{\epsilon}, d_{\epsilon}).
\end{equation*}
Combining this with the mean value formula yields 
\begin{equation}\label{eq-g1-22}
|G(x,y) - G(x,x_{\epsilon})| \leq C|y-x_{\epsilon}| d_{\epsilon}^{-(n-1)}\quad \textrm{and}\quad |\nabla_x G(x,y) -\nabla_x G(x,x_{\epsilon})| \leq C|y-x_{\epsilon}| d_{\epsilon}^{-n}
\end{equation}
for all $y \in B(x_{\ep}, d_{\ep})$. Using this and \eqref{eq-bound} we estimate $I_1$ as
\begin{equation}\label{eq-g1-2}
\begin{split}
I_1 (x) &\leq C d_{\epsilon}^{-(n-1)}\int_{B(x_{\epsilon}, d_{\epsilon}/2)} |y-x_{\epsilon}|\, \lambda_{\epsilon}^{\frac{q_{\ep}n}{q_{\ep}+1}} U^{q}(\lambda_{\epsilon} (y-x_{\epsilon})) dy
\\
&\leq C d_{\epsilon}^{-(n-1)} \lambda_{\epsilon}^{\frac{q_{\ep}n}{q_{\ep}+1}}\lambda_{\epsilon}^{-(n+1)} \int_{B(0, N_{\epsilon}/2)} |y|U^q (y) dy
\\
& = Cd_{\ep}^{-(n-2)}\lambda_{\epsilon}^{-\frac{n}{q_{\ep}+1}} N_{\ep}^{-1}   \int_{B(0, N_{\epsilon}/2)} |y|U^q (y) dy.
\end{split}
\end{equation}
Injecting this into \eqref{eq-g1-2} we get
\begin{equation*}
I_1 (x) = o (d_{\ep}^{-(n-2)} \lambda_{\ep}^{-\frac{n}{q+1}}).
\end{equation*}
By the same way along with the second inequality of \eqref{eq-g1-22}, we can obtain the estimate 
\[
\nabla I_1 (x) = o (d_{\ep}^{-(n-1)} \lambda_{\ep}^{-\frac{n}{q+1}}).
\]

\noindent \emph{Estimate of $I_2$.} For $y \in B(x_{\epsilon}, 4d_{\epsilon}) \setminus B(x_{\epsilon}, d_{\epsilon})$ we use the estimates \eqref{eq-bound} and \eqref{eq-decay} to find
\begin{equation*} 
u_{\epsilon}(y) \leq C \lambda_{\ep}^{\frac{n}{q_{\ep}+1}} U (\lambda_{\epsilon}(y-x_{\epsilon})) \leq \frac{C\lambda_{\ep}^{\frac{n}{q_{\ep}+1}}}{(\lambda_{\epsilon}d_{\epsilon})^{n-2}}.
\end{equation*}
Note that 
\begin{equation}\label{eq-g1-35} 
|x-y| \leq 8d_{\ep}\quad \textrm{for}~ y \in B(x_{\ep}, 4d_{\ep})\quad \textrm{and}\quad x \in \partial B(x_{\ep}, 2d_{\ep}).
\end{equation}
 Hence we have
\begin{equation}\label{eq-g1-31}
\left\{\begin{aligned}
|G(x,y)| + |G(x,x_{\epsilon})| &\leq \frac{c_n}{|x-y|^{n-2}} + \frac{c_n}{d_{\epsilon}^{(n-2)}} \leq \frac{C}{|x-y|^{n-2}},
\\
|\nabla_x G(x,y)| + |\nabla_x G (x,x_{\epsilon})| &\leq \frac{c_n}{|x-y|^{n-1}} + \frac{c_n}{d_{\epsilon}^{(n-1)}} \leq \frac{C}{|x-y|^{n-1}}.
\end{aligned}\right.
\end{equation}
Using the first estimate of \eqref{eq-g1-31} and \eqref{eq-g1-35}  we deduce 
\begin{equation*}
\begin{split}
I_2 (x) &\leq C \lambda_{\ep}^{\frac{qn}{q+1}} d_{\ep}^{-(n-2) q} \lambda_{\ep}^{-(n-2)q}\int_{B(x_{\epsilon},4d_{\epsilon})\setminus B(x_{\epsilon}, d_{\epsilon})} \frac{1}{|x-y|^{n-2}} dy
\\
&\leq C \lambda_{\ep}^{\frac{qn}{q+1}} d_{\ep}^{2-(n-2) q} \lambda_{\ep}^{-(n-2)q} 
\\
&= C\lambda_{\ep}^{-\frac{n}{q+1}}d_{\ep}^{-(n-2)} N_{\ep}^{n-(n-2)q}.
\end{split}
\end{equation*}
Since $q > n/(n-2)$ the above estimate leads to 
\begin{equation}\label{eq-g1-3}
I_2 (x) = o\left( \lambda_{\ep}^{-\frac{n}{q+1}}d_{\ep}^{-(n-2)}\right).
\end{equation}
Similarly, using the second estimate of \eqref{eq-g1-31}, we obtain
\[\nabla I_2 (x) = O \left(\lambda_{\ep}^{-\frac{n}{q+1}}d_{\ep}^{-(n-1)} N_{\ep}^{n-(n-2)q}\right)= o \left(\lambda_{\ep}^{-\frac{n}{q+1}}d_{\ep}^{-(n-1)}\right).\]
\

\noindent\emph{Estimate of $I_3$.} In view of that $|x-x_{\epsilon}|=2d_{\epsilon}$, we easily get the following estimates
\begin{equation}\label{eq-g1-32}
\left\{\begin{aligned}
|G(x,y) - G(x,x_{\epsilon})| &\leq C d_{\epsilon}^{-(n-2)}\quad \textrm{for}~ y \in \Omega \setminus B(x_{\epsilon}, 4d_{\epsilon}),
\\
|\nabla_x G(x,y) - \nabla_x G(x,x_{\epsilon})| &\leq C d_{\epsilon}^{-(n-1)}\quad \textrm{for}~ y \in \Omega \setminus B(x_{\epsilon}, 4d_{\epsilon}).
\end{aligned}\right.
\end{equation}
Using the first inequality of \eqref{eq-g1-32}, we get
\begin{equation*}
\begin{split}
I_3 (x) \leq Cd_{\epsilon}^{-(n-2)}\int_{\Omega \setminus B(x_{\epsilon}, 4d_{\epsilon})} u_{\epsilon}^{q_{\ep}}(y) dy. 
\end{split}
\end{equation*}
From \eqref{eq-decay} and \eqref{eq-bound} we deduce
\begin{equation*}
\begin{split}
\int_{\Omega \setminus B(x_{\epsilon}, 4d_{\epsilon})} u_{\epsilon}^{q_{\ep}}(y) dy  &= \lambda_{\epsilon}^{-\frac{n}{q_{\ep}+1}} \int_{\Omega_{\epsilon}\setminus B(0, 4N_{\epsilon})} \widetilde{u}_{\epsilon}^{q_{\ep}} (y) dy
\\
&\leq C\lambda_{\epsilon}^{-\frac{n}{q_{\ep}+1}}\int_{\mathbb{R}^n \setminus B(0,4N_{\epsilon})} U^{q}(y) dy
\\
&\leq C\lambda_{\epsilon}^{-\frac{n}{q+1}} N_{\epsilon}^{-(n-2)q +n}.
\end{split}
\end{equation*}
Thus, 
\begin{equation}\label{eq-g1-4}
I_3 (x)\leq C d_{\epsilon}^{-(n-2)} \lambda_{\epsilon}^{-\frac{n}{q+1}} N_{\epsilon}^{-(n-2)q +n} = o \left( d_{\epsilon}^{-(n-2)} \lambda_{\epsilon}^{-\frac{n}{q+1}}  \right),
\end{equation}
where we used the fact that $(n-2)p -2 >n-2$. Similarly, applying the second estimate of \eqref{eq-g1-32}, we may obtain 
\[\nabla I_3 (x) = O \left( d_{\epsilon}^{-(n-1)} \lambda_{\epsilon}^{-\frac{n}{q+1}} N_{\epsilon}^{-(n-2)q +n} \right)= o \left( d_{\epsilon}^{-(n-1)} \lambda_{\epsilon}^{-\frac{n}{q+1}}  \right).\] 
Collecting \eqref{eq-g1-2}, \eqref{eq-g1-3}, and \eqref{eq-g1-4} we get
\begin{equation*}
I_1 (x)+ I_2 (x)+ I_3 (x)= o\left(d_{\epsilon}^{-(n-2)} \lambda_{\epsilon}^{-\frac{n}{q+1}}\right)
\end{equation*}
and 
\begin{equation*}
|\nabla_x I_1 (x)| +|\nabla_x I_2 (x)|+ |\nabla_x I_3 (x)| = o\left(d_{\epsilon}^{-(n-1)} \lambda_{\epsilon}^{-\frac{n}{q+1}}\right).
\end{equation*}
We can deduce from the above estimates to get
\begin{equation*}
v_{\epsilon} (x) = A_U \lambda_{\epsilon}^{-\frac{n}{q+1}} G(x,x_{\epsilon}) + o \left(d_{\epsilon}^{-(n-2)} \lambda_{\epsilon}^{-\frac{n}{q+1}}\right)
\end{equation*}
and
\begin{equation*}
\nabla_x v_{\epsilon} (x) = A_U \lambda_{\epsilon}^{-\frac{n}{q+1}} \nabla_x G(x,x_{\epsilon}) + o \left(d_{\epsilon}^{-(n-1)} \lambda_{\epsilon}^{-\frac{n}{q+1}}\right).
\end{equation*}
The lemma is proved.
\end{proof}

\section{The case $p > \frac{n}{n-2}$}
This section is devoted to prove Theorem \ref{thm-1} for the case $p > \frac{n}{n-2}$. 
For the proof, as we explained in the introduction, we will assume that the maximum point $x_{\epsilon}$ converges to a boundary point, and derive a contradiction from the Pohozaev type identity \eqref{eq-hn-15} on the annulus $\partial B(x_{\ep}, 2d_{\ep})$. 
\begin{lem}\label{lem-g1}Suppose $p \in\left(\frac{n}{n-2}, \frac{n+2}{n-2}\right]$. Assume that $\{(u_{\ep},v_{\ep})\}_{\ep >0}$ is a sequence of solutions to \eqref{eq-main} of type $(ME)$ and that $\lim_{\ep \rightarrow 0} d_{\ep} =0$. Then the following estimates hold. For $x \in \partial B(x_{\ep}, 2d_{\ep})$ we have
\begin{equation}\label{eq-g1-20}
u_{\epsilon} (x)= A_V \lambda_{\epsilon}^{-\frac{n}{p+1}} G(x,x_{\epsilon}) + o (d_{\epsilon}^{-(n-2)}\lambda_{\epsilon}^{-\frac{n}{p+1}})
\end{equation}
and
\begin{equation}\label{eq-g1-7}
\nabla u_{\epsilon} (x)= A_V \lambda_{\epsilon}^{-\frac{n}{p+1}} \nabla G(x,x_{\epsilon}) + o (d_{\epsilon}^{-(n-1)}\lambda_{\epsilon}^{-\frac{n}{p+1}}).
\end{equation}
Here $A_{V}$ is the constant given in \eqref{eq-AUV} and the $o$-notation is uniform with respect to $x \in \partial B(x_{\ep}, 2d_{\ep})$.
\end{lem}
\begin{proof}
The proof follows the same lines of the proof of Lemma \ref{lem-g10}.
The only different part is \eqref{eq-g1-2}, which should be replaced by \begin{equation*}
\begin{split}
\int_{B(0,N_{\epsilon})} |y|V^p (y) dy & \leq C \int_{B(0,N_{\epsilon})} \frac{1}{(1+|y|)^{(N-2)p-1}} dy
\\
& \leq \left\{\begin{array}{ll} C &\quad \textrm{if}~ p \in \left( \frac{n+1}{n-2}, \frac{n+2}{n-2}\right),
\\
C\log N&\quad \textrm{if}~ p = \frac{n+1}{n-2},
\\
CN^{-(n-2)p +n+1}&\quad \textrm{if}~ p \in \left( \frac{n}{n-2}, \frac{n+1}{n-2}\right).
\end{array}\right.
\end{split}
\end{equation*}
As the power of $N$ is less than $1$, we can get the estimate $I_1 (x) = o (d_{\ep}^{-(n-2)}\lambda_{\ep}^{-\frac{n}{q+1}})$ as in the proof of Lemma \ref{lem-g10}. The other parts of the proof work in the exactly same way.
\end{proof}
Now we are ready to prove the main result for the case $p \in \left(\frac{n}{n-2}, \frac{n+2}{n-2}\right]$.

\begin{proof}[Proof of Theorem \ref{thm-1} for the case $p > \frac{n}{n-2}$ ] Let $d_{\epsilon} = \textrm{dist}(x_{\epsilon}, \partial \Omega)$. Then we need to show that $\inf_{\ep >0} d_{\ep} >0$. For this aim, with a view to a contradiction, we assume the contrary that $d_{\epsilon} \rightarrow 0$ as $\epsilon \rightarrow 0$ in a subsequence. 
\

By Lemma \ref{lem-bb} we have $d_{\epsilon} = \frac{N_{\epsilon}}{\lambda_{\epsilon}}$ with $N_{\epsilon} \rightarrow \infty$. Now we set $D_{\epsilon} = B(x_{\epsilon}, 2d_{\epsilon})$ for each $1 \leq j \leq n$ and we define the values $L_{\epsilon}^j$ and $R_{\epsilon}^j$ by
\begin{equation*}
\begin{split}
L_{\epsilon}^j:&=-\int_{\partial D_{\epsilon}} \left( \frac{\partial u_{\epsilon}}{\partial \nu} \frac{\partial v_{\epsilon}}{\partial x_j} + \frac{\partial v_{\epsilon}}{\partial \nu} \frac{\partial u_{\epsilon}}{\partial x_j} \right) d S_x + \int_{\partial D_{\epsilon}}(\nabla u_{\epsilon} \cdot \nabla v_{\epsilon}) \nu_j dS_x,
\\
 R_{\epsilon}^j:&= \frac{1}{p+1} \int_{\partial D_{\epsilon}} v_{\epsilon}^{p+1} \nu_j dS_x + \frac{1}{q_{\epsilon}+1} \int_{\partial D_{\epsilon}} u_{\epsilon}^{q_{\epsilon}+1} \nu_j dS_x.
\end{split}
\end{equation*}
Applying Lemma \ref{lem-poho} to $(u_{\ep}, v_{\ep})$ with $D = D_{\epsilon}$, we obtain the following identity
\[L_{\epsilon}^j = R_{\epsilon}^j.\]
In what follows, we shall estimate both the values of $L_j^{\epsilon}$ and $R_j^{\epsilon}$ precisely, which will give us a contradiction.

 Using \eqref{eq-g1-7} and \eqref{eq-g1-8} we compute $L_{\epsilon}^j$ as 
\begin{equation}\label{eq-g1-14}
\begin{split}
L_j^{\epsilon}&=-\lambda_{\epsilon}^{-(n-2+\epsilon)} A_U A_V \int_{\partial D_{\epsilon}} \left( \frac{\partial}{\partial \nu} G(x,x_{\epsilon}) \frac{\partial}{\partial x_j} G(x,x_{\epsilon}) + \frac{\partial}{\partial \nu} G(x,x_{\epsilon}) \frac{\partial}{\partial x_j} G(x,x_{\epsilon}) \right) dS_x 
\\
&\quad + \lambda_{\epsilon}^{-(n-2+\epsilon)} A_U A_V\int_{\partial D_{\epsilon}} |\nabla G(x,x_{\epsilon})|^2 \nu_j dS_x + o \left(|\partial D_{\epsilon}| \lambda_{\epsilon}^{-\left( \frac{n}{p+1}+ \frac{n}{q+1}\right)} d_{\epsilon}^{-2(n-1)}\right)
\\
&=-\lambda_{\epsilon}^{-(n-2+\epsilon)} A_U A_V I(2d_{\epsilon}) + o (d_{\epsilon}^{-(n-1)} \lambda_{\epsilon}^{-(n-2)}),
\end{split}
\end{equation}
where we have set
\begin{equation*}
I(r):= \left[ \int_{\partial B(x_{\epsilon},r)} 2 \frac{\partial G}{\partial \nu} (x,x_{\epsilon}) \frac{\partial}{\partial x_j} G(x,x_{\epsilon}) - |\nabla G(x,x_{\epsilon})|^2 \nu_j dS_x\right]\quad \textrm{for}\quad r>0. 
\end{equation*}
To compute the value of $I(2d_{\epsilon})$, we first observe that $I(r)$ is independent of $r>0$. To show this, we remind that $-\Delta_x G(x,x_{\epsilon})=0$ for $x \in A_{r}:= B(x_{\ep}, 2d_{\ep})\setminus B(x_{\ep},r)$ for each $r \in (0, 2d_{\ep})$. Using this and integration by parts we obtain
\begin{equation}\label{eq-g1-55}
\begin{split}
0 & = \int_{A_r} (-\Delta_x G)(x,x_{\epsilon}) \frac{\partial G}{\partial x_j}(x,x_{\epsilon}) dx
\\
& = -\int_{\partial A_r}\frac{\partial G}{\partial \nu}(x,x_{\epsilon}) \frac{\partial G}{\partial x_j}(x,x_{\epsilon}) dS_x + \int_{A_r} \nabla_x G(x,x_{\epsilon}) \frac{\partial \nabla_x G}{\partial x_j} (x,x_{\epsilon}) dx
\\
& = -\int_{\partial A_r}\frac{\partial G}{\partial \nu}(x,x_{\epsilon}) \frac{\partial G}{\partial x_j}(x,x_{\epsilon}) dS_x + \frac{1}{2} \int_{\partial A_r} |\nabla_x G(x,x_{\epsilon})|^2 \nu_j dS_x,
\end{split}
\end{equation}
which implies that $I(r)$ is constant on $(0,2d_{\epsilon}]$. Using this fact, we compute $I(2d_{\ep})$ by the following limit; 
\begin{equation*}
\begin{split}
I(2d_{\epsilon}) &=\lim_{r \rightarrow 0}I(r)
\\
 & = \lim_{r \rightarrow 0} \int_{\partial B(x_{\epsilon},r)} 2 \left( - \frac{c_n(n-2)}{|x-x_{\epsilon}|^{n}} - \frac{\partial H}{\partial \nu}(x,x_{\epsilon})\right) \left( - \frac{c_n (n-2)(x-x_{\epsilon})_j}{|x-x_{\epsilon}|^n}- \frac{\partial H}{\partial x_j}(x,x_{\epsilon})\right)
\\
&\quad - \left( - \frac{c_n(n-2)(x-x_{\epsilon})}{|x-x_{\epsilon}|^{n}} - \nabla H(x,x_{\epsilon})\right)^2 \nu_j dS_x.
\end{split}
\end{equation*}
Thanks to the oddness of the integrand, we have
\begin{equation*}\int_{\partial B(x_{\epsilon},r)} 2 \left( \frac{c_n(n-2)}{|x-x_{\epsilon}|^{n}} \right) \left( \frac{c_n (n-2)(x-x_{\epsilon})_j}{|x-x_{\epsilon}|^n}\right) - \left[\left(  \frac{c_n(n-2)(x-x_{\epsilon})}{|x-x_{\epsilon}|^{n}}\right)^2 \nu_j \right]dS_x =0.
\end{equation*}
Also, since $-\Delta_x H(x,x_{\ep}) =0$ holds for $x \in B(x_{\ep}, 2d_{\ep})$, we may proceed as in \eqref{eq-g1-55} to get
\begin{equation*}
\begin{split}
\int_{\partial B(x_{\epsilon},r)} 2 \left( \frac{\partial H}{\partial \nu}(x,x_{\epsilon})\right) \left( \frac{\partial H}{\partial x_j}(x,x_{\epsilon})\right) - \left[\left( \nabla H(x,x_{\epsilon})\right)^2 \nu_j \right]dS_x =0.
\end{split}
\end{equation*}
Having the above equalities, we can compute the limit as follows.
\begin{equation*}
\begin{split}
I(2d_{\epsilon}) &=\lim_{r \rightarrow 0} \int_{\partial B(x,r)} 2c_n^2 (n-2) \frac{\partial H}{\partial \nu}(x,x_{\epsilon}) \frac{(x-x_{\epsilon})_j}{|x-x_{\epsilon}|^n} + 2 \frac{c_n (n-2)}{|x-x_{\epsilon}|^{n-1}} \frac{\partial H}{\partial x_j}(x,x_{\epsilon}) dS_x
\\
&\qquad - \frac{2c_n (n-2)(x-x_{\epsilon})}{|x-x_{\epsilon}|^n}\nabla H (x,x_{\epsilon})\nu_j dS_x
\\
& = \left[\frac{2c_n (n-2)}{n} \frac{\partial H}{\partial x_j}(x_{\epsilon}, x_{\epsilon}) + 2c_n (n-2) \frac{\partial H}{\partial x_j}(x_{\epsilon}, x_{\epsilon}) - \frac{2c_n (n-2)}{n} \frac{\partial H}{\partial x_j}(x_{\epsilon}, x_{\epsilon}) \right]|S_{n-1}|
\\
& = 2c_n (n-2)|S_{n-1}| \frac{\partial H}{\partial x_j}(x_{\epsilon}, x_{\epsilon}).
\end{split}
\end{equation*}
Injecting this into \eqref{eq-g1-14} we have
\begin{equation}\label{eq-g1-45}
L_j^{\epsilon} = - \lambda_{\epsilon}^{-(n-2+\epsilon)} c_n A_U A_V |S_{n-1}| 2(n-2) \frac{\partial H}{\partial x_j} (x_{\epsilon},~x_{\epsilon}) + o (d_{\epsilon}^{-(n-1)} \lambda_{\epsilon}^{-(n-2)}).
\end{equation}
We find $(a_1, \cdots, a_n) \in S^{n-1}$ such that $n_{x_n} = - (a_1,\cdots, a_n)$. Then, using \eqref{eq-h2} we obtain
\begin{equation}\label{eq-g1-5}
\begin{split}
\sum_{j=1}^{n} a_j L_j^{\epsilon}& = c_n\lambda_{\epsilon}^{-(n-2+\epsilon)} A_U A_V |S_{n-1}| 2(n-2) \sum_{j=1}^{n} a_j \frac{\partial H}{\partial x_j} (x_{\epsilon}, x_{\epsilon}) + o (d_{\epsilon}^{-(n-1)} \lambda_{\epsilon}^{-(n-2)})
\\
& =c_n \lambda_{\epsilon}^{-(n-2+\epsilon)} A_U A_V |S_{n-1}| 2(n-2) \frac{\partial H}{\partial n_{x_\epsilon}} (x_{\epsilon}, x_{\epsilon}) + o (\lambda_{\epsilon}^{-(n-2)} d_{\epsilon}^{-(n-1)})
\\
& \geq C \lambda_{\epsilon}^{-(n-2)} d_{\epsilon}^{-(n-1)} = C\lambda_{\epsilon} N_{\epsilon}^{-(n-1)}.
\end{split}
\end{equation}
Next we shall find an upper bound of $R_j^{\epsilon}$. Applying \eqref{eq-bound} and \eqref{eq-decay} we have
\begin{equation*}
v_{\epsilon} (x) \leq C\lambda_{\epsilon}^{\frac{n}{p+1}} V(\lambda_{\epsilon}(x-x_{\epsilon})) \leq C \lambda_{\epsilon}^{\frac{n}{p+1}} N_{\epsilon}^{-(n-2)}\quad \forall~ x \in \partial B(x_{\ep}, 2d_{\ep}).
\end{equation*}
Using this we estimate
\begin{equation}\label{eq-g1-11}
\begin{split}
\left| \int_{\partial D_{\epsilon}} v_{\epsilon}^{p+1} \nu_j dS_x \right|& \leq C|\partial D_{\epsilon}| \lambda_{\epsilon}^n N_{\epsilon}^{-(n-2)(p+1)}
\\
& \leq C d_{\epsilon}^{(n-1)} \lambda_{\epsilon}^{n} N_{\epsilon}^{-(n-2)(p+1)}
\\
& = C \left( \frac{N_{\epsilon}}{\lambda_{\epsilon}}\right)^{(n-1)} \lambda_{\epsilon}^n N_{\epsilon}^{-(n-2)(p+1)} = C\lambda_{\epsilon} N_{\epsilon}^{(n-1)-(n-2)(p+1)}.
\end{split}
\end{equation}
Similarly we have
\begin{equation*}
u_{\epsilon}(x) \leq C\lambda_{\epsilon}^{\frac{n}{q_{\epsilon}+1}} N_{\epsilon}^{-(n-2)}\quad \forall~x \in \partial B(x_{\ep}, 2d_{\ep}),
\end{equation*} 
and consequently
\begin{equation}\label{eq-g1-12}
\left| \int_{\partial D_{\epsilon}} u_{\epsilon}^{q_{\epsilon}+1}(x) \nu_j dS_x \right| \leq C \lambda_{\epsilon} N_{\epsilon}^{(n-1) - (n-2) (q_{\epsilon}+1)}.
\end{equation}
Collecting \eqref{eq-g1-11}, \eqref{eq-g1-12} with the fact that $p <q$, we get
\begin{equation}\label{eq-g1-13}
|R_j^{\epsilon}| \leq C \lambda_{\epsilon} N_{\epsilon}^{(n-1) - (n-2)(p+1)}.
\end{equation}
Now we combine \eqref{eq-g1-5} and \eqref{eq-g1-13} to get
\begin{equation*}
\begin{split}
\lambda_{\epsilon} N_{\epsilon}^{-(n-1)} &\leq \sum_{j=1}^n a_j L_j^{\ep}
\\
&= \sum_{j=1}^{n} a_j R_j^{\ep}
\\
& \leq C \lambda_{\epsilon} N_{\epsilon}^{(n-1) - (n-2)(p+1)}.
\end{split}
\end{equation*}
Since $N_{\epsilon}$ goes to infinity as $\epsilon \rightarrow 0$, the above inequality yields that
\begin{equation*}
-(n-1) \leq (n-1) - (n-2) (p+1),
\end{equation*}
which is equivalent to $p \leq \frac{n}{n-2}$. However this contradicts to the condition $p > \frac{n}{n-2}$. Thus the assumption $d_{\epsilon} \rightarrow 0$ cannot hold, and so the maximum point $x_{\epsilon}$ converges to an interior point of $\Omega$. The proof is completed.
\end{proof}

\section{The case $p = \frac{n}{n-2}$}
In this section we prove Theorem \ref{thm-1} for the case $p=\frac{n}{n-2}$. The strategy is same with the proof for the case $p > \frac{n}{n-2}$. However we should modify the estimates of the solution $u_{\ep}$ on the annulus $\partial B(x_{\ep}, 2d_{\ep})$. This is due to the fact that the function $V^p$ is integrable for $p > \frac{n}{n-2}$ but not integrable for $p = \frac{n}{n-2}$ in view of the estimate \eqref{eq-decay}. We obtain the desired estimate in the following lemma.
\begin{lem}\label{lem-g2}
Suppose $p=\frac{n}{n-2}$. Assume that $\{(u_{\ep},v_{\ep})\}_{\ep >0}$ is a sequence of solutions to \eqref{eq-main} of type $(ME)$ and $\lim_{\ep \rightarrow 0} d_{\ep} =0$. Then the following estimates hold.  For $x \in \partial B(x_{\ep}, 2d_{\ep})$ we have
\begin{equation}\label{eq-g2-20}
u_{\epsilon}(x)= K_{\epsilon} \lambda_{\epsilon}^{-\frac{n}{p+1}} G(x,x_{\epsilon}) + O (d_{\epsilon}^{-(n-2)} \lambda_{\epsilon}^{-\frac{n}{p+1}})
\end{equation}
and
\begin{equation}\label{eq-g2-21}
\nabla u_{\epsilon}(x)= K_{\epsilon} \lambda_{\epsilon}^{-\frac{n}{p+1}}\nabla G(x,x_{\epsilon}) + O (d_{\epsilon}^{-(n-1)} \lambda_{\epsilon}^{-\frac{n}{p+1}}),
\end{equation}
where $K_{\ep}$ is a positive constant satisfying
\begin{equation*}
c_1 \log N_{\epsilon} \leq K_{\epsilon} \leq c_2 \log N_{\epsilon}\quad \forall \epsilon >0
\end{equation*}
for some $c_2 >c_1 >0$ independent of $\ep >0$.
\end{lem}
\begin{proof}
From \eqref{eq-main} we have
\begin{equation}\label{eq-g2-1}
\begin{split}
u_{\epsilon}(x) &= \int_{B(x_{\epsilon}, 4d_{\epsilon})} G(x,y) v_{\epsilon}^{p} (y) dy + \int_{B^c (x_{\epsilon}, 4d_{\epsilon})} G(x,y) v_{\epsilon}^p (y) dy.
\\
& = G(x,x_{\epsilon}) \int_{B(x_{\epsilon}, 4d_{\epsilon})} v_{\epsilon}^{p} (y) dy + \int_{B(x_{\epsilon}, 4d_{\epsilon})} (G(x,y)-G(x,x_{\epsilon})) v_{\epsilon}^p (y) dy
\\
&\qquad + \int_{B^c (x_{\epsilon}, 4d_{\epsilon})} G(x,y) v_{\epsilon}^p (y) dy.
\end{split}
\end{equation}
We first estimate the integration $\int_{B(x_{\ep}, 4d_{\ep})} v_{\ep}^{p}(y)dy$. An upper estimate will follows by using \eqref{eq-bound} and \eqref{eq-decay} as before. To find a sharp lower estimate, we are going to find a lower bound of $v_{\ep}$ on $B(x_{\ep}, c_0 d_{\ep})$ for a small constant $c_0 >0$ independent of $\ep >0$. For this purpose, we recall from Section 3 the notations for the scaled domain $\Omega_{\epsilon}:= \lambda_{\epsilon}(\Omega-x_{\epsilon})$ and normalized solutions
\begin{equation*}
\widetilde{u}_{\epsilon}(x):= \lambda_{\epsilon}^{-\frac{n}{q_{\epsilon}+1}} u_{\epsilon}(\lambda_{\epsilon}^{-1} x + x_{\epsilon}),
\quad \textrm{and}\quad \widetilde{v}_{\epsilon}(x):= \lambda_{\epsilon}^{-\frac{n}{p+1}} v_{\epsilon}(\lambda_{\epsilon}^{-1} x + x_{\epsilon}),
\quad \textrm{for}~x \in \Omega_{\epsilon}.
\end{equation*}
For our aim, it suffices to obtain a lower bound of $\widetilde{v}_{\ep}$ on $B(0, c_0 N_{\ep})$. To get it, we begin with the following formula of $\widetilde{v}_{\ep}$ from \eqref{eq-ext-sol};
\begin{equation}\label{eq-g2-12}
\widetilde{v}_{\ep}(x) = \int_{\Omega_{\ep}} G_{\Omega_{\ep}}(x,y) \widetilde{u}_{\ep}^{q}(y) dy,
\end{equation}
where $G_{\Omega_{\ep}}$ is Green's function of the problem
\begin{equation*}
\left\{ \begin{aligned} -\Delta u &= f \quad \textrm{in}~\Omega_{\ep},
\\
u&=0\quad \textrm{on}~\partial \Omega_{\ep}.
\end{aligned}
\right.
\end{equation*}
By scaling, we have
\begin{equation}\label{eq-g2-11}
\begin{split}
G_{\Omega_\ep} (x,y) &= \lambda_{\ep}^{-(n-2)}G(\lambda_{\ep}^{-1} x + x_{\ep}, \lambda_{\ep}^{-1} y + x_{\ep})
\\
&= \frac{c_n}{|x-y|^{n-2}} + \lambda_n^{-(n-2)}H(\lambda_{\ep}^{-1} x + x_{\ep}, \lambda_{\ep}^{-1} y + x_{\ep}).
\end{split}
\end{equation}
Since $\frac{N_\ep}{\lambda_{\ep}}= d_{\ep}$ the estimate \eqref{eq-h1-4} implies that
\begin{equation*}
\sup_{x,y \in B(0, N_{\ep})} H(\lambda_{\ep}^{-1} x + x_{\ep}, \lambda_{\ep}^{-1} y + x_{\ep}) \leq C d_{\ep}^{-(n-2)}.
\end{equation*}
Using this and $N_{\ep} = \lambda_{\ep} d_{\ep}$ we obtain from \eqref{eq-g2-11} the following estimate
\begin{equation*}
G_{\Omega_\ep}(x,y) =\frac{c_n}{|x-y|^{n-2}} + O( N_{\ep}^{-(n-2)})
\geq \frac{C}{|x|^{n-2}},\quad y \in B^n (0,1)\quad \textrm{and}\quad x \in B^n (0, c_0 N_{\ep}) \setminus B^n (0,2), 
\end{equation*}
where $c_0 >0$ is a small constant independent of $\ep >0$. Injecting this estimate into \eqref{eq-g2-12} we get
\begin{equation*}
\begin{split}
\widetilde{v}_{\ep} (x) &\geq \frac{C}{|x|^{n-2}} \int_{B^n (0,1)} \widetilde{u}_{\ep}^{q}(y) dy 
\\
&\geq \frac{C}{|x|^{n-2}} \quad \textrm{for}~ x \in B^n (0, c_0 N_{\ep}) \setminus B^n (0,2),
\end{split}
\end{equation*}
where we used the fact that $\widetilde{u}_{\ep}$ converges to $U$ in $C^0 (B^n (0,1))$ as $\ep \rightarrow 0$. Using this we have
\begin{equation*}
\begin{split}
\int_{B^n (x_{\ep}, 4d_{\ep})} \lambda_{\ep}^{\frac{n}{p+1}} v_{\ep}^{p}(y) dy & = \int_{B^n (0, 4N_{\ep})} \widetilde{v}_{\ep}^{p}(y) dy
\\
& \geq \int_{B^n (0, c_0 N_{\ep}) \setminus B^n (0,2)} \frac{C}{|y|^n} dy
\\
& \geq C \log (N_{\ep}).
\end{split}
\end{equation*}
 Given this lower bound, using the upper estimate \eqref{eq-bound} of $v_{\ep}$ given by \eqref{eq-decay} and \eqref{eq-bound}, we can find a large value $C>1$ such that
\begin{equation*}
\frac{\log (N_{\epsilon})}{C} \leq\int_{B(x_{\epsilon}, 4d_{\epsilon})} \lambda_{\epsilon}^{\frac{n}{p+1}}v_{\epsilon}^p (y) dy  \leq C \log (N_{\epsilon}).
\end{equation*}
Now, it remains to estimate the last two integrations in \eqref{eq-g2-1}. First, we decompose
\begin{equation*}
\begin{split}
 \int_{B(x_{\epsilon}, 4d_{\epsilon})} (G(x,y)-G(x,x_{\epsilon})) v_{\epsilon}^p (y) dy&= \int_{B(x_{\epsilon}, d_{\epsilon})} dy + \int_{B(x_{\epsilon}, 4d_{\epsilon})\setminus B(x_{\epsilon}, d_{\epsilon})} dy 
\\
&:= I_1 (x)+ I_2 (x).
\end{split}
\end{equation*}
In the same way for \eqref{eq-g1-2} and \eqref{eq-g1-3}, one can estimate $I_1 (x)$ and $I_2 (x)$ as 
\begin{equation}\label{eq-g2-2}
I_1 (x)+ I_2 (x)\leq C  \lambda_{\epsilon}^{\frac{n}{q+1}} N_{\epsilon}^{-(n-2)}.
\end{equation}
Lastly we evaluate the last term of \eqref{eq-g2-1} as
\begin{equation*}
\begin{split}
\int_{B^c (x_{\epsilon}, 4d_{\epsilon})} G(x,y) v_{\epsilon}^p (y) dy &=\int_{B^{c}(0, 4d_{\ep})} G(x,x_{\ep} +y) v_{\ep}^{p}(x_{\ep}+y) dy
\\
&\leq C \int_{B^c (0,4d_{\epsilon})} \frac{1}{|y|^{n-2}} \frac{\lambda_{\epsilon}^{\frac{pn}{p+1}}}{|\lambda_{\epsilon}y|^n} dy \leq C \lambda_{\epsilon}^{\frac{pn}{p+1}} \lambda_{\epsilon}^{-n} d_{\epsilon}^{-(n-2)}
\end{split}
\end{equation*}
where we used that $|x-(x_{\ep}+y)| \geq |y|-|x-x_{\ep}| \geq |y|-2d_{\ep} \geq \frac{|y|}{2}$ for $y \in B^{c}(0, 4d_{\ep})$. Inserting the estimate and \eqref{eq-g2-2} into \eqref{eq-g2-1}, we get the desired estimate of $u_{\epsilon}$. A similar argument can be applied to find the estimate of $\nabla u_{\epsilon}$. The proof is finished.
\end{proof}
\begin{proof}[Proof of Theorem \ref{thm-1} for the case $p = \frac{n}{n-2}$] Let $d_{\ep} = d (x_{\ep}, \partial \Omega)$. Then the task is to show that  $\inf_{\ep >0} d_{\ep} >0$. For this aim, we argue by contradiction as in the case $p > \frac{n}{n-2}$. Namely, we assume the contrary that the maximum point $x_{\epsilon}$ approaches to the boundary $\partial \Omega$ as $\epsilon \rightarrow 0$, i.e., $d_{\epsilon}\rightarrow 0$ as $\ep \rightarrow 0$. 
\

We begin with reminding from Lemma \ref{eq-de} that $d_{\epsilon} = \frac{N_{\epsilon}}{\lambda_{\epsilon}}$ with $N_{\epsilon} \rightarrow \infty$. Now, for each $1 \leq j \leq n$, we define the following values 
\begin{equation*}
\begin{split}
L_{\epsilon}^j:&=-\int_{\partial D_{\epsilon}} \left( \frac{\partial u_{\epsilon}}{\partial \nu} \frac{\partial v_{\epsilon}}{\partial x_j} + \frac{\partial v_{\epsilon}}{\partial \nu} \frac{\partial u_{\epsilon}}{\partial x_j} \right) d S_x + \int_{\partial D_{\epsilon}}(\nabla u_{\epsilon} \cdot \nabla v_{\epsilon}) \nu_j dS_x,
\\
 R_{\epsilon}^j:&= \frac{1}{p+1} \int_{\partial D_{\epsilon}} v_{\epsilon}^{p+1} \nu_j dS_x + \frac{1}{q_{\epsilon}+1} \int_{\partial D_{\epsilon}} u_{\epsilon}^{q_{\epsilon}+1} \nu_j dS_x.
\end{split}
\end{equation*}
Applying Lemma \ref{lem-poho} to $(u_{\ep}, v_{\ep})$ with $D_{\epsilon} = B(x_{\epsilon}, 2 d_{\epsilon})$ we get the following identity
\[L_{\epsilon}^j = R_{\epsilon}^j.\]
To compute the value $L_{\ep}^j$,  as in \eqref{eq-g1-14}, we first put the estimates of $u_{\ep}$ and $v_{\ep}$ given in \eqref{eq-g2-21} and \eqref{eq-g1-8} to get 
\begin{equation*}
\begin{split}
L_j^{\epsilon}&=-\lambda_{\epsilon}^{-(n-2+\epsilon)} A_U A_V I(2d_{\epsilon}) + o (d_{\epsilon}^{-(n-1)} \lambda_{\epsilon}^{-(n-2)}),
\end{split}
\end{equation*}
where we have set
\begin{equation*}
I(r):= \int_{\partial B(x_{\epsilon},r)} 2 \frac{\partial G}{\partial \nu} (x,x_{\epsilon}) \frac{\partial}{\partial x_j} G(x,x_{\epsilon}) - (|\nabla G(x,x_{\epsilon})|^2 \nu_j) dS_x\quad \textrm{for}\quad r>0.  
\end{equation*}
Similarly to \eqref{eq-g1-45} we compute the value of $I(r)$ to reach the following identity
\begin{equation*}
L_{\epsilon}^j = - \lambda_{\epsilon}^{-(n-2+\epsilon)} c_n A_U K_{\ep} |S_{n-1}| 2(n-2) \frac{\partial H}{\partial x_j} (x_{\epsilon},~x_{\epsilon}) + O (d_{\epsilon}^{-(n-1)} \lambda_{\epsilon}^{-(n-2)}).
\end{equation*}
Find $(a_1, \cdots, a_n) \in S^{n-1}$ such that $n_{x_{\epsilon}}= (a_1,\cdots, a_n)$. Then, similarly to \eqref{eq-g1-5}, we apply the inequality \eqref{eq-hn} to obtain
\begin{equation}\label{eq-g1-15}
\begin{split}
\sum_{j=1}^{n} a_j L_j^{\epsilon}& \geq C\log (N_{\epsilon})\lambda_{\epsilon}^{-(n-2)}  \frac{\partial H}{\partial n_{x_\epsilon}} (x_{\epsilon}, x_{\epsilon}) + O (\lambda_{\epsilon}^{-(n-2)} d_{\epsilon}^{-(n-1)})
\\
& \geq C \log (N_{\epsilon})\lambda_{\epsilon}^{-(n-2)} d_{\epsilon}^{-(n-1)} = C\log (N_{\epsilon})\lambda_{\epsilon}^1 N_{\epsilon}^{-(n-1)}.
\end{split}
\end{equation}
Next we estimate $R_j^{\epsilon}$. From \eqref{eq-bound} and \eqref{eq-decay} we have the following estimate
\begin{equation*}
v_{\epsilon} (x) \leq C\lambda_{\epsilon}^{\frac{n}{p+1}} V(\lambda_{\epsilon}(x-x_{\epsilon})) \leq C \lambda_{\epsilon}^{\frac{n}{p+1}} N_{\epsilon}^{-(n-2)}\quad\forall~ x \in \partial B(x_{\ep}, 2d_{\epsilon}). 
\end{equation*}
Using this we get
\begin{equation}\label{eq-g1-9}
\begin{split}
\left| \int_{\partial D_{\epsilon}} v_{\epsilon}^{p+1} \nu_j dS_x \right|& \leq C|\partial D_{\epsilon}| \lambda_{\epsilon}^n N_{\epsilon}^{-(n-2)(p+1)}
\\
& \leq C d_{\epsilon}^{(n-1)} \lambda_{\epsilon}^{n} N_{\epsilon}^{-(n-2)(p+1)}
\\
& = C \left( \frac{N_{\epsilon}}{\lambda_{\epsilon}}\right)^{(n-1)} \lambda_{\epsilon}^n N_{\epsilon}^{-(n-2)(p+1)} = \lambda_{\epsilon} N_{\epsilon}^{(n-1)-(n-2)(p+1)}.
\end{split}
\end{equation}
Similarly we have $u_{\epsilon}(x) \leq C\lambda_{\epsilon}^{\frac{n}{q_{\epsilon}+1}} N_{\epsilon}^{-(n-2)}$, and consequently,
\begin{equation}\label{eq-g1-10}
\left| \int_{\partial D_{\epsilon}} u_{\epsilon}^{q_{\epsilon}+1}(x) \nu_j dS_x \right| \leq C \lambda_{\epsilon} N_{\epsilon}^{(n-1) - (n-2) (q_{\epsilon}+1)}.
\end{equation}
Collecting the estimates \eqref{eq-g1-9} and \eqref{eq-g1-10} with the fact that $q_{\epsilon}>p=\frac{n}{n-2}$, we get
\begin{equation}\label{eq-g1-6}
|R_j^{\epsilon}| \leq C \lambda_{\epsilon} N_{\epsilon}^{(n-1) - (n-2)(p+1)} = C_{\lambda_{\epsilon}} N_{\epsilon}^{-(n-1)}.
\end{equation}
Finally, combining \eqref{eq-g1-15} and \eqref{eq-g1-6}, we find the inequality
\begin{equation*}
\begin{split}
C\log (N_{\epsilon}) \lambda_{\epsilon} N_{\epsilon}^{-(n-1)} &\leq \sum_{j=1}^{n} a_j L_j^{\ep}
\\
&= \sum_{j=1}^{n} a_j R_j^{\ep}
\\
&\leq C \lambda_{\epsilon} N_{\epsilon}^{-(n-1)},
\end{split}
\end{equation*}
which is a contradiction because $N_{\epsilon} \rightarrow \infty$. Therefore, the blow up point $x_{\epsilon}$ is away from the boundary $\partial \Omega$ uniformly in $\ep>0$, and hence the maximum point $x_{\ep}$ converges to an interior point of $\Omega$ up to a subsequence. The proof is finished.
\end{proof}

\section{The case $p < \frac{n}{n-2}$}
This section is devoted to prove Theorem \ref{thm-2}. Recall that $\{(u_{\ep}, v_{\ep})\}_{\ep>0}$ is a sequence of solutions of type $(ME)$ to the problem  
\begin{equation}\label{eq-main-2}
\left\{ \begin{array}{ll} -\Delta u_{\epsilon} = v_{\epsilon}^{p}&\textrm{in}~\Omega,
\\
-\Delta v_{\epsilon} = u_{\epsilon}^{q_{\epsilon}}&\textrm{in}~\Omega,
\\
u_{\epsilon}=v_{\epsilon}=0&\textrm{in}~\partial \Omega.
\end{array}
\right.
\end{equation}
As in the previous sections, we take  the value $\lambda_{\epsilon}>0$ and the point $x_{\epsilon} \in \Omega$ such that
\begin{equation*}
\lambda_{\epsilon}=\max_{x \in \Omega} \,\max \{ v_{\epsilon}^{\frac{p+1}{n}}(x), ~u_{\epsilon}^{\frac{q_{\epsilon}+1}{n}} (x) \} = \max \{ v_{\epsilon}^{\frac{p+1}{n}}(x_{\epsilon}),~u_{\epsilon}^{\frac{q_{\epsilon}+1}{n}}(x_{\epsilon})\}.
\end{equation*}
Also we denote $d_{\epsilon} := \frac{1}{4}\textrm{dist}(x_{\epsilon}, \partial \Omega)$. Then we see from Lemma \ref{lem-bb} that
\begin{equation*}
d_{\epsilon} = \frac{N_{\epsilon}}{\lambda_{\epsilon}}\quad \textrm{with}\quad \lim_{\ep \rightarrow 0} N_{\epsilon} = \infty.
\end{equation*}
In addition, we set $\Omega_{\epsilon}:= \lambda_{\epsilon}(\Omega-x_{\epsilon})$ and normalize the solutions as
\begin{equation*}
\widetilde{u}_{\epsilon}(x):= \lambda_{\epsilon}^{-\frac{n}{q_{\epsilon}+1}} u_{\epsilon}(\lambda_{\epsilon}^{-1} x + x_{\epsilon}),
\quad \textrm{and}\quad \widetilde{v}_{\epsilon}(x):= \lambda_{\epsilon}^{-\frac{n}{p+1}} v_{\epsilon}(\lambda_{\epsilon}^{-1} x + x_{\epsilon}),
\quad \textrm{for}~x \in \Omega_{\epsilon}.
\end{equation*}
To prove Theorem \ref{thm-2} we will apply the same strategy used for Theorem \ref{thm-1}. Namely, we shall find a contradiction by exploiting the Pohozaev type identity if we assume the maximum point approaches to the boundary. In this case, it is more difficult to get a precise estimate of $u_{\ep}(x)$ for $x \in \partial B(x_{\ep}, 2d_{\ep})$. For this estimate we first need to obtain sharp estimates of the solution $v_{\ep}(y)$ for all $y \in \Omega$. This will be achieved in Lemma \ref{lem-v} below. 
\

The assumption $d_{\ep} \rightarrow 0$ makes the analysis much more delicate. Hence, for readers' understanding, we first look at the  case which assumes that the maximum point converge to an interior point. 
\begin{lem}
Suppose $p \in \left[1 , \frac{n}{n-2}\right)$. Assume that the maximum point $x_{\ep}$ converges to an interior point $x_0 \in \Omega$. Then the following holds;
\begin{equation*}
\left\{\begin{aligned}
\lim_{\epsilon \rightarrow 0} \lambda_{\epsilon}^{\frac{n}{q+1}} v_{\epsilon}(y) &= A_U G(y,x_0)\quad \textrm{in}~C^0 (\Omega \setminus \{x_0\}),
\\
\lim_{\epsilon \rightarrow 0} \lambda_{\epsilon}^{\frac{np}{q+1}} u_{\epsilon}(x) &= A_U^p \widetilde{G}(x,x_0)\quad \textrm{in}~C^0 (\Omega \setminus \{x_0\}),\end{aligned}\right.
\end{equation*}
where the constant $A_U$ is defined by $A_U = \int_{\mathbb{R}^n} U^q (y) dy$.
\end{lem}
\begin{proof}
We recall from \eqref{eq-decay} that $U(x) \leq C|x|^{-[p(n-2)-2]}$. Also we note that $q (p(n-2)-2) = pq(n-2) -2q =(n+2) + 2p >n$. Thus we have $U^q \in L^1 (\mathbb{R}^N)$. By Lemma \ref{lem-gu} we have $\widetilde{u}_{\ep} (x)\leq CU(x)$ for all $x \in \Omega_{\ep}$. Thus we may apply the dominated convergence theorem to yield
\begin{equation}\label{eq-v-0}
\lim_{\epsilon \rightarrow 0} \int_{\Omega} \lambda_{\epsilon}^{\frac{n}{q_{\ep}+1}} u_{\epsilon}^{q_{\epsilon}} (z) dz =\lim_{\ep \rightarrow 0} \int_{\Omega_{\ep}} U_{\ep}^{q_{\ep}}(x) dx= \int_{\mathbb{R}^n} U^q (x) dx = A_U.
\end{equation}
Then, as the blow up point $x_{\epsilon}$ converges to an interior point $x_0 \in \Omega$, we may deduce from \eqref{eq-main-2} that
\begin{equation*}
\begin{split}
\lim_{\epsilon \rightarrow 0} \lambda_{\epsilon}^{\frac{N}{q+1}} v_{\epsilon}(y) &= \lim_{\ep \rightarrow 0} \int_{\Omega} G(y,z) \lambda_{\ep}^{\frac{n}{q+1}} u_{\ep}^{q_{\ep}}(z) dz
\\
&= A_U G(y,x_0)\quad \textrm{in}~C^0 (\Omega \setminus \{x_0\}).
\end{split}
\end{equation*}
Using this and also \eqref{eq-main-2} again, we get
\begin{equation*}
\lim_{\epsilon \rightarrow 0} (-\Delta) (\lambda_{\epsilon}^{\frac{np}{q+1}} u_{\epsilon})(z) = \lim_{\epsilon \rightarrow 0} \lambda_{\epsilon}^{\frac{np}{q+1}} v_{\epsilon}^{p} (z) = A_U^p G^p (z,x_0)\quad \textrm{in}~C^0 (\Omega \setminus \{x_0\}).
\end{equation*}
From this we get
\begin{equation}\label{eq-v-5}
\begin{split}
\lim_{\epsilon \rightarrow 0} \lambda_{\epsilon}^{\frac{np}{q+1}} u_{\epsilon}(x) &= A_U^p \int_{\Omega} G(x,z) G^p (z,x_0) dz,\quad x \in \Omega \setminus \{x_0\},
\\
&= A_U^p \widetilde{G}(x,x_0),
\end{split}
\end{equation}
where the value of integration in \eqref{eq-v-5} is finite since $|G(x,y)| \leq C |x-y|^{-(n-2)}$ and $p < \frac{n}{n-2}$. The proof is finished. 
\end{proof}
The above proof gives the idea how a sharp estimate of $\lambda_{\ep}^{\frac{np}{q+1}} u_{\ep}(x)$ for a fixed point $x \in \Omega \setminus \{x_0\}$ is obtain from the sharp estimate of $\lambda_{\ep}^{\frac{n}{q_{\ep}+1}} v_{\ep}(y)$ is required for all $y \in \Omega$. Given this idea, we now handle the intricate case that assumes $d_{\ep} \rightarrow 0$. In the following lemma, we obtain the global estimate of $\lambda_{\ep}^{\frac{n}{q_{\ep}+1}}v_{\ep}$.
\begin{lem}\label{lem-v}
Let $p \in \left[ 1, \frac{n}{n-2}\right)$. Assume that $d_{\ep} \rightarrow 0$. Then we have the following estimate
\begin{equation}\label{eq-v-6}
\lambda_{\epsilon}^{\frac{n}{q+1}} v_{\epsilon}(y) =
 \left\{\begin{array}{ll}  A_{U,\ep}G(y,x_{\epsilon}) + O \left(  |y-x_{\epsilon}|^{-(n-2)}N_{\ep}^{-\delta}\right), &\quad \textrm{if}~|y-x_{\epsilon}| \geq \frac{2d_{\ep}}{\sqrt{N_{\ep}}},
\\
O( V(\lambda_{\ep} (y-x_{\ep}))&\quad \textrm{if}~|y-x_{\epsilon}| < \frac{2d_{\ep}}{\sqrt{N_{\ep}}}.
\end{array}\right.
\end{equation}
where $\delta = \min \{ (n-2)q -n, 1\}$ and $A_{U,\epsilon}>0$ is a number converging to $A_{U} = \int_{\mathbb{R}^n} U^q (y) dy$ as $\epsilon \rightarrow 0$. Here the implicit constant of $O(\cdot)$ is uniform with respect to $y \in \Omega$ and $\epsilon >0$, i.e., there is a constant $C>0$ such that
\begin{equation*}
\left|O \left(   |y-x_{\epsilon}|^{-(n-2)}N_{\ep}^{-\delta}\right)\right|\leq C   |y-x_{\epsilon}|^{-(n-2)}N_{\ep}^{-\delta}\quad \textrm{and}\quad
\left|O(V(\lambda_{\ep} (y-x_{\ep}))\right| \leq C V(\lambda_{\ep} (y-x_{\ep}),
\end{equation*}
for all $y \in \Omega~ \textrm{and}~ \ep >0.$
\end{lem}
\begin{proof}
We begin with writing Green's expression from \eqref{eq-main} as follows.
\begin{equation}\label{eq-v-1}
\begin{split}
\lambda_{\epsilon}^{\frac{n}{q+1}}v_{\epsilon}(y)& = \int_{\Omega} G(y,z) \lambda_{\epsilon}^{\frac{n}{q_{\epsilon}+1}} u_{\epsilon}^{q}(z) dz
\\
& = \int_{B(x_{\ep}, d_{\ep}/\sqrt{N_{\ep}})} G(y,z) \lambda_{\epsilon}^{\frac{n}{q_{\epsilon}+1}} u_{\epsilon}^{q}(z) dz
+ \int_{\Omega\setminus B(x_{\ep}, d_{\ep}/\sqrt{N_{\ep}})} G(y,z) \lambda_{\epsilon}^{\frac{n}{q_{\epsilon}+1}} u_{\epsilon}^{q}(z) dz.
\end{split}
\end{equation}
We first estimate the last integration as an error term.
\begin{equation*}
\begin{split}
\left|\int_{\Omega\setminus B(x_{\ep}, d_{\ep}/\sqrt{N_{\ep}})} G(y,z) \lambda_{\epsilon}^{\frac{n}{q_{\epsilon}+1}} u_{\epsilon}^{q}(z) dz\right|
&\leq \int_{\Omega\setminus B(x_{\ep}, d_{\ep}/\sqrt{N_{\ep}})} \frac{1}{|y-z|^{n-2}} \frac{\lambda_{\ep}^{n}}{(1+\lambda_{\ep} |z-x_{\ep}|)^{(n-2)q}} dz
\\
&\leq \lambda_{\ep}^{n-(n-2)q}\int_{\Omega\setminus B(x_{\ep}, d_{\ep}/\sqrt{N_{\ep}})} \frac{1}{|y-z|^{n-2}} \frac{1}{|z-x_{\ep}|^{(n-2)q}} dz.
\end{split}
\end{equation*}
To analyze the integration, we split the domain as
\[
\Omega\setminus B(x_{\ep}, d_{\ep}/\sqrt{N_{\ep}}) = A_1 \cup A_2 \cup A_3,
\]
where we have set
\begin{equation*}
\left\{\begin{aligned}
A_1 &= \{ z \in \Omega: |z-y| \leq \frac{1}{2} |y-x_{\ep}|\},
\\
A_2 & = \{ z \in \Omega: |z-x_{\ep}| \leq \frac{1}{2} |y-x_{\ep}| \quad \textrm{and}\quad |z-x_{\ep}| > \frac{d_{\ep}}{\sqrt{N_{\ep}}}\},
\\
A_3 & = \{ z \in \Omega: |z-y| > \frac{1}{2}|y-x_{\ep}|\quad \textrm{and}\quad |z-x_{\ep}| > \frac{1}{2} |y-x_{\ep}|\}.
\end{aligned}\right.
\end{equation*}
First we estimate the integration on $A_1$. Note that $|z-x_{\ep}| \geq |y-x_{\ep}| - |z-y| \geq \frac{1}{2} |y-x_{\ep}|$ for $z \in A_1$. Hence we have
\begin{equation*}
\begin{split}
\lambda_{\ep}^{n-(n-2)q}\int_{A_1} \frac{1}{|y-z|^{n-2}} \frac{1}{|z-x_{\ep}|^{(n-2)q}} dz &\leq \frac{C\lambda_{\ep}^{n-(n-2)q}}{|y-x_{\ep}|^{(n-2)q}} \int_{|z-y| \leq \frac{|y-x_{\ep}|}{2}} \frac{1}{|y-z|^{n-2}}dz
\\
&\leq \frac{C\lambda_{\ep}^{n-(n-2)q}}{|y-x_{\ep}|^{(n-2)q}}|y-x_{\ep}|^2
\\
&= \frac{C}{|y-x_{\ep}|^{n-2}} \cdot \frac{1}{(\lambda_{\ep}|y-x_{\ep}|)^{(n-2)q -n}}.
\end{split}
\end{equation*}
Similarly, we have  $|y-z|\geq |y-x_{\ep}| - |z-x_{\ep}| \geq\frac{|y-x_{\ep}|}{2}$ for $z \in A_2$. Using this we estimate
\begin{equation*}
\begin{split}
\lambda_{\ep}^{n-(n-2)q}\int_{A_2} \frac{1}{|y-z|^{n-2}} \frac{1}{|z-x_{\ep}|^{(n-2)q}} dz &\leq \frac{C\lambda_{\ep}^{n-(n-2)q}}{|y-x_{\ep}|^{(n-2)}} \int_{|z-x_{\ep}|> \frac{d_{\ep}}{\sqrt{N_{\ep}}}} \frac{1}{|z-x_{\ep}|^{(n-2)q}}dz
\\
&\leq \frac{C\lambda_{\ep}^{n-(n-2)q}}{|y-x_{\ep}|^{(n-2)}} \frac{1}{(d_{\ep}/ \sqrt{N_{\ep}})^{(n-2)q-n}}
\\
&= \frac{C}{|y-x_{\ep}|^{n-2}} \cdot \frac{1}{(\sqrt{N_{\ep}})^{(n-2)q -n}}.
\end{split}
\end{equation*}
Also we note that $|z-x_{\ep}| \geq |z-y| - |y-x_{\ep}|\geq \frac{1}{2} |z-y|$ for $z \in A_3$. Using this we get
\begin{equation*}
\begin{split}
\lambda_{\ep}^{n-(n-2)q}\int_{A_3} \frac{1}{|y-z|^{n-2}} \frac{1}{|z-x_{\ep}|^{(n-2)q}} dz &\leq {C\lambda_{\ep}^{n-(n-2)q}} \int_{|y-z| > \frac{|y-x_{\ep}|}{2}} \frac{1}{|y-z|^{n-2 +(n-2)q}}dz
\\
&\leq \frac{C\lambda_{\ep}^{n-(n-2)q}}{|y-x_{\ep}|^{(n-2)q-2}}
\\
&= \frac{C}{|y-x_{\ep}|^{n-2}} \cdot \frac{1}{(\lambda_{\ep}|y-x_{\ep}|)^{(n-2)q -n}}.
\end{split}
\end{equation*}
Next we look at the main term of \eqref{eq-v-1}. Let us write it as
\begin{equation}\label{eq-v-11}
\begin{split}
& \int_{B(x_{\ep}, d_{\ep}/\sqrt{N_{\ep}})} G(y,z) \lambda_{\epsilon}^{\frac{n}{q_{\epsilon}+1}} u_{\epsilon}^{q}(z) dz
\\
& = G(y, x_{\ep}) 
 \int_{B(x_{\ep}, d_{\ep}/\sqrt{N_{\ep}})}  \lambda_{\epsilon}^{\frac{n}{q_{\epsilon}+1}} u_{\epsilon}^{q}(z) dz + 
 \int_{B(x_{\ep}, d_{\ep}/\sqrt{N_{\ep}})} [G(y,z)-G(y,x_{\ep})] \lambda_{\epsilon}^{\frac{n}{q_{\epsilon}+1}} u_{\epsilon}^{q}(z) dz.
\end{split}
\end{equation}
Let $A_{U,\ep} =  \int_{B(x_{\ep}, d_{\ep}/\sqrt{N_{\ep}})}  \lambda_{\epsilon}^{\frac{n}{q_{\epsilon}+1}} u_{\epsilon}^{q}(z) dz$. Then, using the dominated convergence theorem we find that
\begin{equation*}
\lim_{\ep \rightarrow 0} A_{U,\ep} =\lim_{\ep \rightarrow 0} \int_{|y| \leq \sqrt{N_{\ep}}} \widetilde{u}^{q_{\ep}} (y) dy  = \int_{\mathbb{R}^n} U^q (y) dy = A_U.
\end{equation*}
Now, it is only left to estimate the last term of \eqref{eq-v-11}. For $z \in B(x_{\ep}, \frac{d_{\ep}}{\sqrt{N_{\ep}}})$ we have
\begin{equation*}
\begin{split}
\left|\nabla_z G(y,z) \right| & \leq \frac{1}{d_{\ep}}  \sup_{w \in B(x_{\ep}, d_{\ep})} |G(y,w)|
\\
&\leq \frac{1}{d_{\ep}} \sup_{w \in B(x_{\ep}, d_{\ep})} \frac{C}{|y-w|^{n-2}}
\\
&\leq\frac{C}{d_{\ep} |y-x_{\ep}|^{n-2}}.
\end{split}
\end{equation*}
Using this we get
\begin{equation*}
\begin{split}
& \int_{B(x_{\ep}, d_{\ep}/\sqrt{N_{\ep}})} [G(y,z)-G(y,x_{\ep})] \lambda_{\epsilon}^{\frac{n}{q_{\epsilon}+1}} u_{\epsilon}^{q}(z) dz 
\\
& \quad\quad \leq \frac{C}{d_{\ep} |y-x_{\ep}|^{n-2}}  \int_{B(x_{\ep}, d_{\ep}/\sqrt{N_{\ep}})}|z-x_{\ep}| U^{q} (\lambda_{\ep} (z-x_{\ep})) dz
\\
&\quad\quad \leq \frac{C\sqrt{N_{\ep}}}{\lambda_{\ep}d_{\ep} |y-x_{\ep}|^{n-2}} \int_{\mathbb{R}^n} |x|U^q (x) dx
\\
&\quad\quad \leq \frac{C}{\sqrt{N_{\ep}} |y-x_{\ep}|^{n-2}}.
\end{split}
\end{equation*}The proof is complete.
\end{proof}
Based on the result of the previous lemma, we shall now find the sharp estimate of $u_{\epsilon}$ on $\partial B(x_{\epsilon}, 2d_{\epsilon})$.
\begin{lem}\label{lem-u}
For $x \in \partial B(x_{\ep}, 2d_{\ep})$ we have
\begin{equation}\label{eq-u-70}
u_{\epsilon}(x) = \lambda_{\epsilon}^{-\frac{np}{q_{\epsilon}+1}} (A_{U,\ep})^p \widetilde{G}(x,x_{\epsilon}) + o (\lambda_{\epsilon}^{-\frac{np}{q+1}} d_{\epsilon}^{-(n-2)p+2})
\end{equation}
and
\begin{equation*}
\nabla u_{\epsilon}(x) = \lambda_{\epsilon}^{-\frac{np}{q_{\epsilon}+1}} (A_{U,\ep})^p \nabla \widetilde{G}(x,x_{\epsilon}) + o (\lambda_{\epsilon}^{-\frac{np}{q+1}} d_{\epsilon}^{-(n-2)p+1}).
\end{equation*}
\end{lem}
\begin{proof}
 Using \eqref{eq-main}, we have
\begin{equation}\label{eq-u-1}
\begin{split}
\lambda_{\epsilon}^{\frac{np}{q_{\epsilon}+1}}u_{\epsilon}(x) 
&= \int_{\Omega}G(x,y)\left( \lambda_{\epsilon}^{\frac{p}{q_{\epsilon}+1}} v_{\epsilon}(y)\right)^p dy
\\
&= \int_{|y-x_{\ep}| \leq \frac{2d_{\ep}}{\sqrt{N}_{\ep}}}G(x,y)\left( \lambda_{\epsilon}^{\frac{p}{q_{\epsilon}+1}} v_{\epsilon}(y)\right)^p dy + \int_{|y-x_{\ep}| >\frac{2d_{\ep}}{\sqrt{N}_{\ep}}}G(x,y)\left( \lambda_{\epsilon}^{\frac{p}{q_{\epsilon}+1}} v_{\epsilon}(y)\right)^p dy.
\end{split}
\end{equation}
First we estimate the last term. Using \eqref{eq-decay} and \eqref{eq-bound}, we have
\begin{equation}\label{eq-u-61}
\begin{split}
&\int_{|y-x_{\ep}| \leq \frac{2d_{\ep}}{\sqrt{N}_{\ep}}}G(x,y)\left( \lambda_{\epsilon}^{\frac{p}{q_{\epsilon}+1}} v_{\epsilon}(y)\right)^p dy 
\\
&\qquad \leq d_{\ep}^{-(n-2)}\int_{|y-x_{\ep}| \leq \frac{2d_{\ep}}{\sqrt{N}_{\ep}}}(\lambda_{\ep}^{(n-2)}V (\lambda_{\ep}(y -x_{\ep})))^p dy 
\\
&\qquad \leq d_{\ep}^{-(n-2)} \lambda_{\ep}^{(n-2)p -n} \int_{B(0, \sqrt{N_{\ep}})} V^p (y) dy.
\\
&\qquad \leq Cd_{\ep}^{-(n-2)} \lambda_{\ep}^{(n-2)p -n} (\sqrt{N_{\ep}})^{n-p (n-2)}
\\
&\qquad = C d_{\ep}^{-(n-2)p+2} [d_{\ep}\lambda_{\ep}]^{(n-2)p -n}(\sqrt{N_{\ep}})^{n-p (n-2)}
\\
&\qquad = C d_{\ep}^{-(n-2)p+2} (\sqrt{N_{\ep}})^{-n + p (n-2)} = o ( d_{\ep}^{-(n-2)p+2}),
\end{split}
\end{equation}
which can be absorbed in the error of the estimate \eqref{eq-u-70}. 
\

Next we note from  Lemma \ref{lem-v} that
\begin{equation}\label{eq-u-68}
\lambda_{\epsilon}^{\frac{p}{q_{\epsilon}+1}} v_{\epsilon}(y) = A_{U,\epsilon}G(y,x_{\epsilon}) + R_{\epsilon}(y),
\end{equation}
where $R_{\epsilon}$ satisfies the estimate $R_{\epsilon}(y) = O (N_{\ep}^{-\delta} |y-x_{\epsilon}|^{-(n-2)})$. Let us write
\begin{equation}\label{eq-u-62}
\begin{split}
&\int_{|y-x_{\ep}| > \frac{2d_{\ep}}{\sqrt{N}_{\ep}}}G(x,y)\left( \lambda_{\epsilon}^{\frac{p}{q_{\epsilon}+1}} v_{\epsilon}(y)\right)^p dy 
\\
&= \int_{|y-x_{\ep}| > \frac{2d_{\ep}}{\sqrt{N}_{\ep}}} G(x,y) \left[ A_{U,\epsilon}G(y,x_{\epsilon}) + R_{\epsilon}(y)\right]^{p}dy 
\\
&=  A_{U,\epsilon}^p \int_{|y-x_{\ep}| > \frac{2d_{\ep}}{\sqrt{N}_{\ep}}}G(x,y)G(y,x_{\epsilon})^{p} dy 
\\
&\qquad + \int_{|y-x_{\ep}| > \frac{2d_{\ep}}{\sqrt{N}_{\ep}}} G(x,y)\left\{ \left[ A_{U,\epsilon}G(y,x_{\epsilon}) + R_{\epsilon}(y)\right]^{p}-A_{U,\epsilon}^p G(y,x_{\epsilon})^{p}\right\} dy.
\end{split}
\end{equation}
Note that
\begin{equation}\label{eq-u-63}
\begin{split}
&\int_{|y-x_{\ep}| >\frac{2d_{\ep}}{\sqrt{N}_{\ep}}}G(x,y)G(y,x_{\ep})^p dy
\\
&= \int_{\Omega}G(x,y) G(y,x_{\ep})^p dy - \int_{|y-x_{\ep}| \leq\frac{2d_{\ep}}{\sqrt{N}_{\ep}}}G(x,y)G(y,x_{\ep})^p dy.
\end{split}
\end{equation}
For $|y-x_{\ep}| \leq \frac{2d_{\ep}}{\sqrt{N_{\ep}}}$ and $x \in \partial B(x_{\ep}, 2d_{\ep})$ we have $|y-x_{\ep}| \geq d_{\ep}$. Using this we find the estimate
\begin{equation}\label{eq-u-64}
\int_{|y-x_{\ep}| \leq\frac{2d_{\ep}}{\sqrt{N}_{\ep}}}G(x,y)G(y,x_{\ep})^p dy \leq C d_{\ep}^{-(n-2)} \left( \frac{2d_{\ep}}{\sqrt{N}_{\ep}}\right)^{-(n-2)p +n} = C d_{\ep}^{2- (n-2)p}\sqrt{N_{\ep}}^{(n-2)p-n}. 
\end{equation} 
By definition \eqref{eq-h1-6} we have 
\begin{equation}\label{eq-u-65}
A_{U,\epsilon}^p \int_{\Omega}G(x,y)G(y,x_{\epsilon})^{p} dy  = A_{U,\ep}^{p} \widetilde{G}(x,x_{\ep}).
\end{equation}
Combining \eqref{eq-u-63}-\eqref{eq-u-65} we see that
\begin{equation}\label{eq-u-66}
A_{U,\epsilon}^p\int_{|y-x_{\ep}| >\frac{2d_{\ep}}{\sqrt{N}_{\ep}}}G(x,y)G(y,x_{\ep})^p dy = A_{U,\epsilon}^p\widetilde{G}(x,x_{\ep}) + o(d_{\ep}^{2-(n-2)p}).
\end{equation}
Now it only remains to estimate the last integration of \eqref{eq-u-62} as an error in \eqref{eq-u-70}. We apply the basic inequality $|a^p -b^p | \leq C|a-b|(a^{p-1} +b^{p-1})$ for $a, b>0$ to get
\begin{equation}\label{eq-u-67}
\begin{split}
&\int_{|y-x_{\ep}|> \frac{2d_{\ep}}{\sqrt{N}_{\ep}}} G(x,y)\left\{ \left[ A_{U,\epsilon}G(y,x_{\epsilon}) + R_{\epsilon}(y)\right]^{p}-A_{U,\epsilon}^p G(y,x_{\epsilon})^{p}\right\} dy 
\\
&\quad \leq C\int_{|y-x_{\ep}|> \frac{2d_{\ep}}{\sqrt{N}_{\ep}}} \frac{1}{|x-y|^{n-2}}\frac{ R_{\epsilon}(y)}{|y-x_{\epsilon}|^{(n-2)(p-1)}} dy:= I.
\end{split}
\end{equation}
We apply the bound of $R_{\ep}$ in \eqref{eq-u-68} and split the integration as
\begin{equation*}
I \leq \frac{C}{N_{\ep}^{\delta}} \int_{|y-x_{\ep}|> \frac{2d_{\ep}}{\sqrt{N}_{\ep}}} \frac{1}{|x-y|^{n-2}}\frac{1}{|y-x_{\epsilon}|^{(n-2)p}} dy := J_1 + J_2 + J_3,
\end{equation*}
where
\begin{equation*}
\begin{split}
J_1 &= \frac{C}{N_{\ep}^{\delta}}\int_{A_1} \frac{1}{|x-y|^{n-2}}\frac{1}{|y-x_{\epsilon}|^{(n-2)p}} dy \quad \textrm{with}\quad  A_1 = B(x, d_{\epsilon}),
\\ 
J_2 &= \frac{C}{N_{\ep}^{\delta}}\int_{A_2} \frac{1}{|x-y|^{n-2}}\frac{1}{|y-x_{\epsilon}|^{(n-2)p}} dy \quad \textrm{with}\quad A_2 = \Omega \setminus ( B(x,d_{\epsilon}) \cup B(x_{\epsilon}, d_{\epsilon})),
\\
J_3 &=\frac{C}{N_{\ep}^{\delta}} \int_{A_3} \frac{1}{|x-y|^{n-2}}\frac{1}{|y-x_{\epsilon}|^{(n-2)p}} dy \quad \textrm{with}\quad A_3 = B(x_{\epsilon}, d_{\epsilon}) \setminus B(x_{\epsilon}, \frac{2d_{\ep}}{\sqrt{N_{\ep}}}).
\end{split}
\end{equation*}
For $y \in A_1$ we have $|y-x_{\ep}|\geq |x-x_{\ep}|-|y-x|\geq 2 d_{\ep} -d_{\ep}=d_{\ep}$. Using this we may estimate $J_1$ as
\begin{equation*}
\begin{split}
J_1 &\leq \frac{C}{N_{\ep}^{\delta}} \cdot d_{\ep}^{-(n-2)p}  \int_{B(x, d_{\ep})} \frac{1}{|y-x|^{n-2}} dy
\\
&\leq C d_{\epsilon}^{-(n-2)p+2} N_{\epsilon}^{-\delta}.
\end{split}
\end{equation*}
Next, to estimate $J_2$, we note that 
\[|y-x_{\ep}| \geq \frac{1}{2}(d_{\ep} + |y-x|)\quad \textrm{and}\quad |y-x|\geq \frac{1}{2}(d_{\ep}+|y-x|)\quad \textrm{for all}~ y \in A_2.\]
Using this and noting that $(n-2)p + (n-2) >n$, we get the estimate
\begin{equation*}
\begin{split}
J_2 &\leq \frac{C}{N_{\ep}^{\delta}}\int_{A_2} \frac{1}{(d_{\ep} + |y-x|)^{(n-2)p+(n-2)}} dy
\\
&\leq \frac{C}{N_{\ep}^{\delta}} \int_{\mathbb{R}^n} \frac{1}{(d_{\ep} + |y-x|)^{(n-2)p+(n-2)}} dy
\\
&\leq \frac{C}{N_{\ep}^{\delta}}\frac{1}{d_{\ep}^{(n-2)p-2}} = Cd_{\epsilon}^{-(n-2)p +2} N_{\epsilon}^{-\delta}.
\end{split}
\end{equation*}
To estimate $J_3$, we note that $|y-x|\geq |x-x_{\ep}| - |y-x_{\ep}| \geq d_{\ep}$ for any $y \in A_3$. Using this we estimate $J_3$ as
\begin{equation*}
\begin{split}
J_3&\leq \frac{C}{N_{\ep}^{\delta}}\int_{A_3} \frac{1}{|y-x|^{n-2}} \frac{1}{|y-x_{\epsilon}|^{(n-2)p}} dy
\\
&\leq \frac{C}{N_{\ep}^{\delta}} d_{\ep}^{-(n-2)} \int_{B(x_{\ep}, d_{\ep})} \frac{1}{|y-x_{\ep}|^{(n-2)p }} dy
\\
&\leq \frac{C}{N_{\ep}^{\delta}} d_{\epsilon}^{-(n-2)} d_{\epsilon}^{n-(n-2)p}
\\
&= Cd_{\epsilon}^{-(n-2)p+2} N_{\epsilon}^{-\delta}.
\end{split}
\end{equation*}
Combining the above estimates gives the estimate $I \leq Cd_{\ep}^{-(n-2)p+2} N_{\ep}^{-\delta}$. Putting this and \eqref{eq-u-66} into \eqref{eq-u-62} we obtain 
\begin{equation*}
\lambda_{\ep}^{\frac{np}{q_{\ep}+1}} u_{\ep} (x) = (A_{U,\ep})^{p} \widetilde{G}(x,x_{\ep}) + o (d_{\ep}^{-(n-2)p +2}).
\end{equation*}
Hence we have obtained the desired estimate for $u_{\ep}$. 
\

In the same way, we can prove the desired estimate for $\nabla u_{\ep}$. For completeness, let us explain it briefly. First, we find from \eqref{eq-u-1} that
\begin{equation}\label{eq-u-1b}
\lambda_{\epsilon}^{\frac{np}{q_{\epsilon}+1}}\nabla u_{\epsilon}(x) = \int_{|y-x_{\ep}| \leq \frac{2d_{\ep}}{\sqrt{N}_{\ep}}}\nabla_x G(x,y)\left( \lambda_{\epsilon}^{\frac{p}{q_{\epsilon}+1}} v_{\epsilon}(y)\right)^p dy + \int_{|y-x_{\ep}| >\frac{2d_{\ep}}{\sqrt{N}_{\ep}}}\nabla_x G(x,y)\left( \lambda_{\epsilon}^{\frac{p}{q_{\epsilon}+1}} v_{\epsilon}(y)\right)^p dy.
\end{equation}
Similarly to \eqref{eq-u-61}, we have
\begin{equation}\label{eq-u-61b}
\begin{split}
&\int_{|y-x_{\ep}| \leq \frac{2d_{\ep}}{\sqrt{N}_{\ep}}}\nabla_x G(x,y)\left( \lambda_{\epsilon}^{\frac{p}{q_{\epsilon}+1}} v_{\epsilon}(y)\right)^p dy 
\\
&\qquad \leq d_{\ep}^{-(n-1)}\int_{|y-x_{\ep}| \leq \frac{2d_{\ep}}{\sqrt{N}_{\ep}}}(\lambda_{\ep}^{(n-2)}V (\lambda_{\ep}(y -x_{\ep})))^p dy 
\\
&\qquad = C d_{\ep}^{-(n-2)p+1} (\sqrt{N_{\ep}})^{-n + p (n-2)}.
\end{split}
\end{equation}
Similarly to \eqref{eq-u-62} we write
\begin{equation}\label{eq-u-62b}
\begin{split}
&\int_{|y-x_{\ep}| > \frac{2d_{\ep}}{\sqrt{N}_{\ep}}}\nabla_x G(x,y)\left( \lambda_{\epsilon}^{\frac{p}{q_{\epsilon}+1}} v_{\epsilon}(y)\right)^p dy 
\\
&=  A_{U,\epsilon}^p \int_{|y-x_{\ep}| > \frac{2d_{\ep}}{\sqrt{N}_{\ep}}}\nabla_x G(x,y)G(y,x_{\epsilon})^{p} dy 
\\
&\qquad + \int_{|y-x_{\ep}| > \frac{2d_{\ep}}{\sqrt{N}_{\ep}}} \nabla_x G (x,y)\left\{ \left[ A_{U,\epsilon}G(y,x_{\epsilon}) + R_{\epsilon}(y)\right]^{p}-A_{U,\epsilon}^p G(y,x_{\epsilon})^{p}\right\} dy.
\end{split}
\end{equation}
We have
\begin{equation}\label{eq-u-63b}
\begin{split}
&\int_{|y-x_{\ep}| >\frac{2d_{\ep}}{\sqrt{N}_{\ep}}}\nabla_x G(x,y)G(y,x_{\ep})^p dy
\\
&= \int_{\Omega}\nabla_x G(x,y) G(y,x_{\ep})^p dy - \int_{|y-x_{\ep}| \leq\frac{2d_{\ep}}{\sqrt{N}_{\ep}}}\nabla_x G(x,y)G(y,x_{\ep})^p dy
\end{split}
\end{equation}
and
\begin{equation}\label{eq-u-64b}
\left|\int_{|y-x_{\ep}| \leq\frac{2d_{\ep}}{\sqrt{N}_{\ep}}}\nabla_x G(x,y)G(y,x_{\ep})^p dy \right|\leq C d_{\ep}^{1- (n-2)p}\sqrt{N_{\ep}}^{(n-2)p-n}. 
\end{equation} 
Note from \eqref{eq-h1-6} that 
\begin{equation}\label{eq-u-65b}
A_{U,\epsilon}^p \int_{\Omega}\nabla_x G(x,y)G(y,x_{\epsilon})^{p} dy  = A_{U,\ep}^{p}\nabla_x \widetilde{G}(x,x_{\ep}).
\end{equation}
Hence
\begin{equation}\label{eq-u-66b}
\int_{|y-x_{\ep}| >\frac{2d_{\ep}}{\sqrt{N}_{\ep}}}\nabla_x G(x,y)G(y,x_{\ep})^p dy = \nabla_x \widetilde{G}(x,x_{\ep}) + o(d_{\ep}^{1-(n-2)p}).
\end{equation}
Next, we estimate
\begin{equation}\label{eq-u-67}
\begin{split}
&\int_{|y-x_{\ep}|> \frac{2d_{\ep}}{\sqrt{N}_{\ep}}} \nabla_x G(x,y)\left\{ \left[ A_{U,\epsilon}G(y,x_{\epsilon}) + R_{\epsilon}(y)\right]^{p}-A_{U,\epsilon}^p G(y,x_{\epsilon})^{p}\right\} dy 
\\
&\quad \leq C\int_{|y-x_{\ep}|> \frac{2d_{\ep}}{\sqrt{N}_{\ep}}} \frac{1}{|x-y|^{n-1}}\frac{ R_{\epsilon}(y)}{|y-x_{\epsilon}|^{(n-2)(p-1)}} dy:= \widetilde{I}.
\end{split}
\end{equation}
We have
\begin{equation*}
|\nabla_x G(x,y) | \leq C\left( \frac{1}{d_{\ep}|x-y|^{n-2}} + \frac{1}{|x-y|^{n-1}}\right).
\end{equation*}
Using this and the bound of $R_{\ep}$ we estimate $\widetilde{I}$ as
\begin{equation*}
\begin{split}
\widetilde{I} &\leq \frac{C}{N_{\ep}^{\delta}} \int_{|y-x_{\ep}|> \frac{2d_{\ep}}{\sqrt{N}_{\ep}}} \frac{1}{|x-y|^{n-1}}\frac{1}{|y-x_{\epsilon}|^{(n-2)p}} dy 
\\
&\quad\quad +\frac{C}{N_{\ep}^{\delta} d_{\ep}} \int_{|y-x_{\ep}|> \frac{2d_{\ep}}{\sqrt{N}_{\ep}}} \frac{1}{|x-y|^{n-2}}\frac{1}{|y-x_{\epsilon}|^{(n-2)p}} dy.
\end{split}
\end{equation*}
The first integration is bounded by $Cd_{\ep}^{-(n-2)p+1}N_{\ep}^{-\delta}$ by the estimate of $I$. Also, the second integration can be estimated by the same way for $I$. More precisely, we split the integration again as
\begin{equation*}
\int_{|y-x_{\ep}|> \frac{2d_{\ep}}{\sqrt{N}_{\ep}}} \frac{1}{|x-y|^{n-1}}\frac{1}{|y-x_{\epsilon}|^{(n-2)p}} dy := \widetilde{J}_1 +\widetilde{J}_2 + \widetilde{J}_3,
\end{equation*}
where
\begin{equation*}
\begin{split}
\widetilde{J}_1 &= \frac{C}{N_{\ep}^{\delta}}\int_{A_1} \frac{1}{|x-y|^{n-1}}\frac{1}{|y-x_{\epsilon}|^{(n-2)p}} dy \quad \textrm{with}\quad  A_1 = B(x, d_{\epsilon}),
\\ 
\widetilde{J}_2 &= \frac{C}{N_{\ep}^{\delta}}\int_{A_2} \frac{1}{|x-y|^{n-1}}\frac{1}{|y-x_{\epsilon}|^{(n-2)p}} dy \quad \textrm{with}\quad A_2 = \Omega \setminus ( B(x,d_{\epsilon}) \cup B(x_{\epsilon}, d_{\epsilon})),
\\
\widetilde{J}_3 &=\frac{C}{N_{\ep}^{\delta}} \int_{A_3} \frac{1}{|x-y|^{n-1}}\frac{1}{|y-x_{\epsilon}|^{(n-2)p}} dy \quad \textrm{with}\quad A_3 = B(x_{\epsilon}, d_{\epsilon}) \setminus B(x_{\epsilon}, \frac{2d_{\ep}}{\sqrt{N_{\ep}}}).
\end{split}
\end{equation*}
For $y \in A_1$ we have $|y-x_{\ep}|\geq d_{\ep}$ and we estimate $\widetilde{J}_1$ as
\begin{equation*}
\begin{split}
\widetilde{J}_1 &\leq \frac{C}{N_{\ep}^{\delta}} \cdot d_{\ep}^{-(n-2)p}  \int_{B(x, d_{\ep})} \frac{1}{|y-x|^{n-2}} dy
\\
&\leq C d_{\epsilon}^{-(n-2)p+1} N_{\epsilon}^{-\delta}.
\end{split}
\end{equation*}
Also, similarly for the estimate $J_2$, we have
\begin{equation*}
\begin{split}
\widetilde{J}_2 &\leq \frac{C}{N_{\ep}^{\delta}}\int_{A_2} \frac{1}{(d_{\ep} + |y-x|)^{(n-2)p+(n-1)}} dy
\\
&\leq \frac{C}{N_{\ep}^{\delta}} \int_{\mathbb{R}^n} \frac{1}{(d_{\ep} + |y-x|)^{(n-2)p+(n-1)}} dy
\\
&\leq \frac{C}{N_{\ep}^{\delta}}\frac{1}{d_{\ep}^{(n-2)p-1}} = Cd_{\epsilon}^{-(n-2)p +1} N_{\epsilon}^{-\delta}.
\end{split}
\end{equation*}
To estimate $\widetilde{J}_3$, we note that for any $y \in A_3$ we have $|y-x|\geq |x-x_{\ep}| - |y-x_{\ep}| \geq 2 d_{\ep}-d_{\ep} =d_{\ep}$. Then
\begin{equation*}
\begin{split}
\widetilde{J}_3&\leq \frac{C}{N_{\ep}^{\delta}}\int_{A_3} \frac{1}{|y-x|^{n-1}} \frac{1}{|y-x_{\epsilon}|^{(n-2)p}} dy
\\
&\leq \frac{C}{N_{\ep}^{\delta}} d_{\ep}^{-(n-1)} \int_{B(x_{\ep}, d_{\ep})} \frac{1}{|y-x_{\ep}|^{(n-2)p }} dy
\\
&\leq Cd_{\epsilon}^{-(n-2)p+1} N_{\epsilon}^{-\delta}.
\end{split}
\end{equation*}
Combining the above estimates gives the desired estimate for $\nabla u_{\ep}$. The proof is finished.
\end{proof}
Now we are ready to prove Theorem \ref{thm-2}.
\begin{proof}[Proof of Theorem \ref{thm-2}]
Let $d_{\ep} = d (x_{\ep}, \partial \Omega)/4$. Then we have to show that $\int_{\ep >0} d_{\ep} >0$ for the proof. As before, we argue by contradiction. Suppose not. Then, we have $d_{\ep} \searrow 0$ in a subsequence. By Lemma \ref{lem-bb} we have
$d_{\ep} = N_{\ep}/\lambda_{\ep}$ with $\lim_{\ep \rightarrow \infty}N_{\ep} = \infty$. 
\

We set $B_{\epsilon} = B(x_{\epsilon}, 2d_{\epsilon})$ for each $\ep >0$. 
Applying Lemma \ref{lem-poho} with $D = B_{\epsilon}$ we get
the identity $L_j^{\epsilon}= R_j^{\epsilon}$ with
\begin{equation*}
\begin{split}
L_j^{\epsilon}&:= -\int_{\partial B_{\epsilon}} \left( \frac{\partial u_{\epsilon}}{\partial \nu} \frac{\partial v_{\epsilon}}{\partial x_j} + \frac{\partial v}{\partial \nu} \frac{\partial u}{\partial x_j} \right) dS_x + \int_{\partial B} (\nabla u_{\ep} \cdot \nabla v_{\ep}) \nu_j dS_x
\\
&\qquad - \frac{1}{p+1} \int_{\partial B} v_{\ep}^{p+1} \nu_j dS_x,
\\
R_j^{\epsilon}&:= \frac{1}{q_{\epsilon}+1} \int_{\partial B} u_{\ep}^{q_{\epsilon}+1} \nu_j dS_x.
\end{split}
\end{equation*}
We shall derive a contradiction by obtaining sharp estimates of $L_j^{\ep}$ and $R_j^{\ep}$. 
\

To evaluate the left hand side, we use Lemma \ref{lem-v} and Lemma \ref{lem-u} to yield
\begin{equation}\label{eq-h3-1}
\begin{split}
L_{j}^{\epsilon} & = \left( \lambda_{\epsilon}^{-\frac{np}{q_{\epsilon}+1}}\lambda_{\epsilon}^{-\frac{n}{q_{\epsilon}+1}}\right) \left[ \int_{\partial B_{\epsilon}} -\left( \frac{\partial \widetilde{G}}{\partial \nu} (x,x_{\epsilon})  \frac{\partial G}{\partial x_j}(x,x_{\epsilon}) + \frac{\partial G}{\partial \nu}(x,x_{\epsilon}) \frac{\partial \widetilde{G}}{\partial x_j}(x,x_{\epsilon})\right) dS_x\right.
\\
&\quad +\left. \int_{\partial B_{\epsilon}}\left( \nabla G (x,x_{\epsilon}) \cdot \nabla \widetilde{G}(x,x_{\epsilon}) \right) \nu_j dS_x 
-\frac{1}{p+1} \int_{\partial B_{\epsilon}} G^{p+1}(x,x_{\epsilon}) \nu_j dS_x\right]
\\
&\quad + o ( |\partial B_{\epsilon}| \lambda_{\epsilon}^{-\frac{np}{q_{\epsilon}+1}} \lambda_{\epsilon}^{-\frac{n}{q_{\epsilon}+1}} d_{\epsilon}^{-(n-1)} d_{\epsilon}^{-(n-2)p+1})
\\
&:= \left( \lambda_{\epsilon}^{-\frac{np}{q_{\epsilon}+1}}\lambda_{\epsilon}^{-\frac{n}{q_{\epsilon}+1}}\right) I (2d_{\epsilon})+ o ( |\partial B_{\epsilon}| \lambda_{\epsilon}^{-\frac{np}{q+1}} \lambda_{\epsilon}^{-\frac{n}{q+1}} d_{\epsilon}^{-(n-1)} d_{\epsilon}^{-(n-2)p+1}),
\end{split}
\end{equation}
where
\begin{equation}\label{eq-h3-4}
\begin{split}
I(r) &=  - \int_{\partial B_{\epsilon}} \left( \frac{\partial \widetilde{G}}{\partial \nu} (x,x_{\epsilon})  \frac{\partial G}{\partial x_j}(x,x_{\epsilon}) + \frac{\partial G}{\partial \nu}(x,x_{\epsilon}) \frac{\partial \widetilde{G}}{\partial x_j}(x,x_{\epsilon})\right) dS_x
\\
&\quad + \int_{\partial B_{\epsilon}}\left( \nabla_x G (x,x_{\epsilon}) \cdot \nabla_x \widetilde{G}(x,x_{\epsilon}) \right) \nu_j dS_x 
-\frac{1}{p+1} \int_{\partial B_{\epsilon}} G^{p+1}(x,x_{\epsilon}) \nu_j dS_x.
\end{split}
\end{equation}
To compute the value of $I(2d_{\epsilon})$, we first claim that the value of $I(r)$ is independent of $r>0$. To prove this claim, we let $A_{r} = B(x_{\epsilon},3d_{\epsilon})\setminus B(x_{\epsilon},r)$ for each $r \in (0,2d_{\ep}]$ and note that we have $-\Delta_x \widetilde{G} (x,x_{\ep})= G^{p}(x,x_{\ep})$ for $x \in A_r$ by definition \eqref{eq-h1-6}. Using this and an integration by parts, we get 
\begin{equation*} 
\begin{split}
\frac{1}{p+1} \int_{\partial A_r} G^{p+1}(x,x_{\ep}) \nu_j dS_x &=\int_{A_r} G^{p}(x,x_{\ep}) \frac{\partial G (x,x_{\ep})}{\partial x_j} dx
\\
& = \int_{A_r} -\Delta_x \widetilde{G} (x,x_{\ep})\frac{\partial G}{\partial x_j}(x,x_{\ep}) dx 
\\
&= - \int_{\partial A_r}\frac{\partial \widetilde{G}}{\partial \nu} \frac{\partial G}{\partial x_j} (x,x_{\ep})dS_x + \int_{A_r} \nabla_x \widetilde{G} \cdot \nabla_x \frac{\partial G}{\partial x_j}(x,x_{\ep}) dx.
\end{split}
\end{equation*}
Similarly, we use that $\Delta_x G (x,x_{\ep})=0$ for $x \in A_r$  to find
\begin{equation*}
0=\int_{A_r} \frac{\partial \widetilde{G}}{\partial x_j}  \cdot (-\Delta_x G ) (x,x_{\ep}) dx = -\int_{\partial A_r}\frac{\partial \widetilde{G}}{\partial x_j} \frac{\partial G}{\partial \nu} (x,x_{\ep}) dS_x + \int_{A_r} \nabla_x \frac{\partial \widetilde{G} }{\partial x_j} \nabla_x G (x,x_{\ep}) dx.
\end{equation*}
Summing up these two equalities and using an integration by parts further, we get
\begin{equation*}
\begin{split}
&\frac{1}{p+1} \int_{\partial A_r} G^{p+1} (x,x_{\ep}) \nu_j dS_x
\\
&\qquad= - \int_{\partial A_r}\frac{\partial \widetilde{G}}{\partial \nu} \frac{\partial G}{\partial x_j} (x,x_{\ep})dS_x  -\int_{\partial A_r}\frac{\partial\widetilde{G}}{\partial x_j} \frac{\partial G}{\partial \nu} (x,x_{\ep}) dS_x + \int_{\partial A_r} ( \nabla_x \widetilde{G} \cdot \nabla_x G)(x,x_{\ep}) \, \nu_j dS_x.
\end{split}
\end{equation*}
This equality implies that $I(r)$ is constant function on $r \in (0,2d_{\epsilon}]$, and so 
\begin{equation*}
I( 2d_{\epsilon}) = \lim_{r \rightarrow 0} I(r).
\end{equation*}
Now we are going to find this limit. First we are concerned with the last term of $I(r)$. Using \eqref{eq-h1-3} we have
\begin{equation}\label{eq-h3-6}
\begin{split}
& \int_{\partial B (x_\epsilon, r)} \frac{1}{p+1} G^{p+1} (x,x_{\ep})\,\nu_j dS_y 
\\
& = \int_{\partial B (x_\epsilon, r)} \frac{1}{p+1} \left( \frac{c_n}{|y-x_{\epsilon}|^{n-2}} - H(y,x_{\epsilon})\right)^{p} \nu_j dS_y
\\
& = \int_{\partial B (x_\epsilon, r)} \frac{1}{p+1} \left[ \left( \frac{c_n}{|y-x_{\epsilon}|^{n-2}} - H(y,x_{\epsilon})\right)^{p+1} - \left( \frac{c_n}{|y-x_{\epsilon}|^{n-2}}\right)^{p+1}+ \frac{(p+1)H(y,x_{\epsilon})}{|y-x_{\epsilon}|^{(n-2)p}} \right]\nu_j dS_y
\\
&\qquad - \int_{\partial B (x_\epsilon, r)} \frac{(p+1)H(y,x_{\epsilon})}{|y-x_{\epsilon}|^{(n-2)p}} \nu_j dS_y,
\end{split}
\end{equation}
where we used that 
\begin{equation*} 
\int_{\partial B(x_{\ep},r)} \left( \frac{c_n}{|y-x_{\ep}|^{n-2}}\right)^{p+1} \nu_j dS_y\quad \textrm{for}~ j =1,\cdots, n.
\end{equation*}
Using the Taylor theorem of second order, we derive
\begin{equation*}
\begin{split}
&\left|\int_{\partial B (x_\epsilon, r)} \frac{1}{p+1} \left[ \left( \frac{c_n}{|y-x_{\epsilon}|^{n-2}} - H(y,x_{\epsilon})\right)^{p+1} - \left( \frac{c_n}{|y-x_{\epsilon}|^{n-2}}\right)^{p+1}+ \frac{(p+1)H(y,x_{\epsilon})}{|y-x_{\epsilon}|^{(n-2)p}} \right]\nu_j dS_y\right|
\\
&\leq C\int_{\partial B (x_\epsilon, r)} \frac{1}{|y-x_{\epsilon}|^{(n-2)(p-1)}} dS_y \leq C r^{(n-1)-(n-2)(p-1)}.
\end{split}
\end{equation*}
This term converges to zero as $r \rightarrow 0$ since $(n-2)(p-1) \leq 2 \leq n-1$ for $p < \frac{n}{n-2}$. On the other hand, 
\begin{equation*}
\begin{split}
\left| \int_{\partial B (x_\epsilon, r)} \frac{(p+1)H(y,x_{\epsilon})}{|y-x_{\epsilon}|^{(n-2)p}} \nu_j dS_y\right| & = \left| \int_{\partial B (x_\epsilon, r)} \frac{(p+1)[H(y,x_{\epsilon})-H(x_{\epsilon}, x_{\epsilon})]}{|y-x_{\epsilon}|^{(n-2)p}} \nu_j dS_y\right|
\\
&\leq  C_{\epsilon} \int_{\partial B (x_\epsilon, r)} \frac{r}{r^{(n-2)p}} dS_r \leq Cr^{n-(n-2)p},
\end{split}
\end{equation*}
which goes to zero as $r \rightarrow 0$ since $p < \frac{n}{n-2}$. Combining the above two estimates with \eqref{eq-h3-6} we find that
\begin{equation}\label{eq-h3-5}
\lim_{r \rightarrow 0}  \int_{\partial B (x_\epsilon, r)} \frac{1}{p+1} G^{p+1} (x,x_{\ep}) \nu_j dS_y =0.
\end{equation}
Now we shall estimate the other three terms of $I(r)$ in \eqref{eq-h3-4}. For this aim, we recall from \eqref{eq-h1-2} that
\begin{equation*}
\widetilde{G}(x,y) = \frac{\alpha_1}{|x-y|^{(n-2)p-2}} -\frac{\alpha_2 H(x,y)}{|x-y|^{[(n-2)p-n]}}- \widetilde{H}(x,y)\quad \textrm{for}~ x \neq y.
\end{equation*}
Using this we deduce
\begin{equation*}
\begin{split}
& \lim_{r \rightarrow 0} \int_{\partial B(x_{\epsilon},r)} \frac{\partial \widetilde{G}}{\partial \nu} \frac{\partial G}{\partial x_j}(x,x_{\ep}) dS_x
\\
&= \lim_{r \rightarrow 0} \int_{\partial B(x_{\epsilon},r)} \left[ -\frac{\alpha_1 [(n-2)p-2]}{|x-x_{\epsilon}|^{p(n-2)-1}}+\frac{\alpha_2 [(n-2)p-n] H(x,x_{\epsilon})}{|x-x_{\epsilon}|^{[(n-2)p-n+1]}} + \frac{\alpha_2 \frac{\partial H}{\partial \nu}(x,x_{\epsilon})}{|x-x_{\epsilon}|^{[(n-2)p-n]}}- \frac{\partial \widetilde{H}}{\partial \nu}(x,x_{\epsilon})\right]
\\
&\qquad\left[ -(n-2)c_n\frac{(x-x_{\epsilon})_j}{|x-x_{\epsilon}|^{n}} - \frac{\partial H}{\partial x_j}(x,x_{\epsilon}) \right] dS_x
\\
&= Q_1 + Q_2 + Q_3 + Q_4,
\end{split}
\end{equation*}
where
\begin{equation*}
\begin{split}
Q_1&= \lim_{r \rightarrow 0} \int_{\partial B(x_{\epsilon},r)} \left[ -\frac{\alpha_1 [(n-2)p-2]}{|x-x_{\epsilon}|^{p(n-2)-1}}\right]
\left[ -(n-2)c_n\frac{(x-x_{\epsilon})_j}{|x-x_{\epsilon}|^{n}}  \right] dS_x,
\\
Q_2&= \lim_{r \rightarrow 0} \int_{\partial B(x_{\epsilon},r)} \left[ \frac{\alpha_2 [(n-2)p-n] H(x,x_{\epsilon})}{|x-x_{\epsilon}|^{[(n-2)p-n+1]}} + \frac{\alpha_2 \frac{\partial H}{\partial \nu}(x,x_{\epsilon})}{|x-x_{\epsilon}|^{[(n-2)p-n]}}- \frac{\partial \widetilde{H}}{\partial \nu}(x,x_{\epsilon})\right]
\left[ -(n-2)c_n\frac{(x-x_{\epsilon})_j}{|x-x_{\epsilon}|^{n}}\right] dS_x,
\\
Q_3&= \lim_{r \rightarrow 0} \int_{\partial B(x_{\epsilon},r)} \left[ -\frac{\alpha_1 [(n-2)p-2]}{|x-x_{\epsilon}|^{p(n-2)-1}}\right]
\left[ - \frac{\partial H}{\partial x_j}(x,x_{\epsilon}) \right] dS_x,
\\
Q_4 &= \lim_{r \rightarrow 0} \int_{\partial B(x_{\epsilon},r)} \left[ \frac{\alpha_2 [(n-2)p-n] H(x,x_{\epsilon})}{|x-x_{\epsilon}|^{[(n-2)p-n+1]}} + \frac{\alpha_2 \frac{\partial H}{\partial \nu}(x,x_{\epsilon})}{|x-x_{\epsilon}|^{[(n-2)p-n]}}- \frac{\partial \widetilde{H}}{\partial \nu}(x,x_{\epsilon})\right]
\left[- \frac{\partial H}{\partial x_j}(x,x_{\epsilon}) \right] dS_x.
\end{split}
\end{equation*}
First we see that $Q_1 =0$ by the oddness of the integrand. Also, it is easy to see that $Q_3 = Q_4 =0$ by counting the order of singularity of $|x-x_{\ep}|^{-1}$ with the fact that $p< \frac{n}{n-2}$. 
\

In order to compute $Q_2$, we note that
\begin{equation*}
\begin{split}
\int_{\partial B(x_{\epsilon},r)} \frac{H(x,x_{\epsilon}) (x-x_{\epsilon})_j}{|x-x_{\epsilon}|^{(n-2)p +1}} dx & =\int_{\partial B(x_{\epsilon},r)} \frac{[H(x,x_{\epsilon})- H(x_{\epsilon}, x_{\epsilon})] (x-x_{\epsilon})_j}{|x-x_{\epsilon}|^{(n-2)p+1}}dS_x
\\
&= O \left( \int_{\partial B(x_{\epsilon},r)} \frac{1}{|x-x_{\epsilon}|^{(n-2)p-1}} dS_x\right)
\\
&= O\left( \frac{r^{n-1}}{r^{(n-2)p-1}}\right),
\end{split}
\end{equation*}
which converges to zero as $r\rightarrow 0$ because $p < \frac{n}{n-2}$. Also, counting the singularity we have
\begin{equation*}
\lim_{r \rightarrow 0} \int_{\partial B(x_{\ep}, r)} \frac{\partial H}{\partial \nu} (x,x_{\ep}) \frac{1}{|x-x_{\ep}|^{(n-2s)p}} dS_x = 0,
\end{equation*}
and it is easy to get the following limit
\begin{equation*}
\lim_{r \rightarrow 0} \int_{\partial B(x_{\ep},r)} \frac{\partial \widetilde{H}}{\partial \nu}(x,x_{\ep}) (n-2) c_n \frac{(x-x_{\ep})_j}{|x-x_{\ep}|^n} dS_x = 
\frac{(n-2)c_n}{n} \frac{\partial \widetilde{H}}{\partial x_j}(x_{\epsilon}, x_{\epsilon})|S_{n-1}|.
\end{equation*}
Using these estimates, we get
\begin{equation*}
Q_2 =\frac{(n-2)c_n}{n} \frac{\partial \widetilde{H}}{\partial x_j}(x_{\epsilon}, x_{\epsilon})|S_{n-1}|.
\end{equation*}
Thus we have
\begin{equation*}
 \lim_{r \rightarrow 0} \int_{\partial B(x_{\epsilon},r)} \frac{\partial \widetilde{G}}{\partial \nu} \frac{\partial G}{\partial x_j}(x,x_{\ep}) dS_x =Q_2= \frac{(n-2)c_n}{n} \frac{\partial \widetilde{H}}{\partial x_j}(x_{\epsilon}, x_{\epsilon})|S_{n-1}|.
\end{equation*}
By computing similarly, we can obtain the following limits:
\begin{equation*}
\begin{split}
&\lim_{r \rightarrow 0} \int_{\partial B(x_{\epsilon},r)} \frac{\partial \widetilde{G}}{\partial x_j} \frac{\partial G}{\partial \nu} dS_x
 = (n-2) c_n \frac{\partial \widetilde{H}}{\partial x_j} (x_{\epsilon}, x_{\epsilon})|S_{n-1}|,
\end{split}
\end{equation*}
and 
\begin{equation*}
\lim_{r\rightarrow 0} \int_{\partial B_r} (\nabla \widetilde{G}\cdot \nabla G) \nu_j dS_x = - \frac{(n-2)c_n}{n} \frac{\partial \widetilde{H}}{\partial x_j}(x_{\epsilon}, x_{\epsilon})|S_{n-1}|.
\end{equation*}
Plugging the above computations and \eqref{eq-h3-5} into \eqref{eq-h3-4}, we get
\begin{equation}\label{eq-h3-2}
\begin{split}
I(2d_{\epsilon})&= \lim_{r \rightarrow 0} I(r)
\\
&=\left[\frac{(n-2)c_n}{n} \partial_{x_j} \widetilde{H}(x_0, x_0) + (n-2)c_n \partial_{x_j} \widetilde{H}(x_0, x_0) - \frac{(n-2)c_n}{n} \partial_{x_j} \widetilde{H}(x_0,x_0)\right]|S_{n-1}|
\\
& = (n-2)c_n |S_{n-1}|\partial_{x_j} \widetilde{H}(x_0, x_0).
\end{split}
\end{equation}
Now we find $(a_1, \cdots, a_n) \in S^{n-1}$ such that $n_x = (a_1,\cdots, a_n)$. Then, using \eqref{eq-h3-1} and \eqref{eq-h3-2} we obtain the estimate
\begin{equation}\label{eq-h3-3}
\begin{split}
\sum_{j=1}^{n} a_j L_j^{\epsilon} &=(n-2)c_n |S_{n-1}| \lambda_{\epsilon}^{-\frac{np}{q_{\epsilon}+1}} \lambda_{\epsilon}^{-\frac{n}{q_{\epsilon}+1}} \partial_{\nu_{x_{\epsilon}}} \widetilde{H}(x_{\epsilon}, x_{\epsilon}) + o \left(\lambda_{\epsilon}^{-\frac{np}{q+1}} \lambda_{\epsilon}^{-\frac{n}{q+1}} d_{\epsilon}^{1-(n-2)p }\right).
\end{split}
\end{equation}
By applying the estimate \eqref{eq-h2} of \textbf{(A1)} in the above inequality, we get
\begin{equation*}
\sum_{j=1}^{n} a_j L_j^{\epsilon}\geq C \lambda_{\epsilon}^{-\frac{n(p+1)}{q+1}} \lambda_{\epsilon}^{-1 + (n-2)p} N_{\epsilon}^{1-(n-2)p} =C \lambda_{\epsilon} N_{\epsilon}^{1-(n-2)p},
\end{equation*}
where we made use of the relation $(p,q)$ given by \eqref{eq-cr-hy} in the equality. \

On the other hand, we may estimate $R_j^{\ep}$ using Lemma \ref{lem-gu} and \eqref{eq-decay} to find
\begin{equation*}
\begin{split}
\left|R_j^{\ep}\right| = \frac{1}{q_{\ep}+1}\left|\int_{\partial B(x_{\epsilon},r)}u^{q_{\epsilon}+1} \nu_j dS_x \right|&\leq C \lambda_{\epsilon}^{\frac{n}{q_{\epsilon}+1} (q_{\epsilon}+1)} d_{\epsilon}^{(n-1)} N_{\epsilon}^{-(p(n-2)-2)(q+1)}
\\
&= \lambda_{\ep}^{n} \lambda_{\ep}^{-(n-1)} N_{\ep}^{(n-1)- (p(n-2)-2)(q+1)}
\\
& = \lambda_{\epsilon} N_{\epsilon}^{-np-1},
\end{split}
\end{equation*}
where we used the relation \eqref{eq-cr-hy} in the last equality. Combining this estimate and \eqref{eq-h3-3}, we get the following inequality
\begin{equation*}
\begin{split}
C\lambda_{\epsilon} N_{\epsilon}^{1-(n-2)p}& \leq \sum_{j=1}^{n} a_j L_j^{\ep}
\\
&= \sum_{j=1}^{n} a_j R_j^{\ep}
\\
&\leq C \lambda_{\epsilon} N_{\epsilon}^{-2p - p(n-2) -1}.
\end{split}
\end{equation*}
Because $N_{\ep} \rightarrow \infty$ as $\ep \rightarrow 0$, the above estimate implies that $2p +2 \leq 0$, which contradicts to the fact that $p >0$. Therefore the assumption $d_{\epsilon} \rightarrow 0$ cannot be true. Hence the maximum point $x_{\ep}$ is away from the boundary $\partial \Omega$ uniformly for $\ep >0$. The proof is finished.
\end{proof}

\section{Proof of Theorem \ref{thm-3}}
This section is devoted to prove Lemma \ref{lem-as} and Theorem \ref{thm-3} concerning the property of $\widetilde{H}$ defined in \eqref{eq-h1-2}. 
\

For given $y \in \Omega$, using the translation and rotation invariance property, we may assume that $0 \in \partial \Omega$ is the closet point in $\partial \Omega$ to the point $y$ and the point $y$ is given by $y= (0, \cdots, 0,\kappa) =\kappa e_n$ for some $\kappa>0$. Then we have $y^* = -\kappa e_n$ and $\partial \Omega$ is parametrized as $w= (z_1, \cdots, z_{n-1}, f(z))$ with a function $f : \mathbb{R}^{n-1} \rightarrow \overline{\mathbb{R}^{+}}$ such that $f(0)=0$ and $\nabla f(0) =0$. \

For the proof of Lemma \ref{lem-as}, we shall rescale and take a limit. Namely, we set $\Omega_{\kappa} := \frac{\Omega}{{\kappa}}$  and rescale the function $\widetilde{H}$ to define  the function $W_{\kappa} : \Omega_{\kappa} \rightarrow \mathbb{R}$ for each ${\kappa}>0$ by
\begin{equation}\label{eq-g-0}
W_{\kappa} (z) = \widetilde{H}({\kappa}z, {\kappa}e_n ) {\kappa}^{p(n-2)-2}.
\end{equation}
We remind the well known inequality:
\begin{equation}\label{eq-g-1}
G(x,y) \leq C\min \left\{ \frac{\mathbf{d}(x)\mathbf{d}(y)}{|x-y|^n}, \frac{1}{|x-y|^{n-2}}\right\}.
\end{equation}
Let $G_{\kappa}$ be Green's function of $-\Delta$ on $\Omega_{\kappa}$ with the Dirichlet boundary condition. Then, by \eqref{eq-g-1} we have
\begin{equation*}
\begin{split}
G_{\kappa} (x,y)&= {\kappa}^{n-2} G({\kappa}x, {\kappa}y)
\\
& \leq C d^{n-2} \min \left( \frac{\mathbf{d}({\kappa}x) \mathbf{d}({\kappa}y)}{{\kappa}^{n}{|x-y|^{n}}}, \frac{1}{{\kappa}^{n-2}|x-y|^{n-2}}\right)
\\
& = C\min\left( \frac{\mathbf{d}({\kappa}x)\mathbf{d}({\kappa}y)}{{\kappa}^2 |x-y|^{n}}, \frac{1}{|x-y|^{n-2}}\right).
\end{split}
\end{equation*}
For each $y \in \mathbb{R}^{n}_{+}$ we denote $y^* = (y_1, \cdots, y_{n-1}, -y_n)$ for $y = (y_1, \cdots, y_n) \in \mathbb{R}^{n}_{+}$. And we consider the function $\widetilde{H}_0:\overline{\mathbb{R}^n_{+}} \times \overline{\mathbb{R}^{n}_{+}} \rightarrow \mathbb{R}$ satisfying
\begin{equation}\label{eq-h0-1}
-\Delta_z \widetilde{H}_0 (z,y)=
\left\{\begin{array}{l}\begin{split} &\frac{c_n^p}{|z-y|^{(n-2)p}}-\left( \frac{c_n}{|z-y|^{n-2}}-\frac{c_n}{|z-y^*|^{n-2}} \right)^p - p \frac{1}{|z-y|^{(n-2)(p-1)}} \cdot \frac{c_n^p}{|z-y^*|^{n-2}}
\\
&- \frac{2pc_n^p (n-2)}{[(n-2)(p-2)-2]} \frac{(z-y^*)}{|z-y^*|^n} \frac{z-y}{|z-y|^{(n-2)(p-1)}}\quad \textrm{if}~p \in \left[ \frac{n-1}{n-2}, \frac{n}{n-2}\right]
\end{split}
\\
\\
\frac{c_n^p}{|z-y|^{(n-2)p}}-\left( \frac{c_n}{|z-y|^{n-2}}-\frac{c_n}{|z-y^*|^{n-2}} \right)^p \quad \textrm{if} ~p\in \left[1, \frac{n-1}{n-2}\right),
\end{array}
\right.
\end{equation}
for $z \in \mathbb{R}^n_{+}$ and $y \in \mathbb{R}^n_{+}$ with the boundary condition
\begin{equation}\label{eq-h0-2}
\widetilde{H}_0 (z,y) = \left\{\begin{array}{ll}\frac{\alpha_1-\alpha_2}{|(z-y)|^{(n-2)p-2}}&\quad \textrm{if}~p \in \left[ \frac{n-1}{n-2}, \frac{n}{n-2}\right]
\\
\frac{\alpha_1}{|(z-y)|^{(n-2)p-2}}&\quad\textrm{if}~p\in \left[1, \frac{n-1}{n-2}\right).
\end{array}
\right. z \in \partial \overline{\mathbb{R}^n_{+}}.
\end{equation}
Here $\alpha_1$ and $\alpha_2$ are the values defined in \eqref{eq-alpha}. 
Now we set $W_0 : \overline{\mathbb{R}^n_{+}} \rightarrow \mathbb{R}$ by $W_0 (z) := \widetilde{H}_0 (z,e_n)$. Then we have the following result.
\begin{lem}\label{lem-hd}
As ${\kappa} \rightarrow 0$, the function $W_{\kappa}$ converges to $W_0$ in $C^1 (B(e_n, 1/4))$.
\end{lem}
\begin{proof}
By definition \eqref{eq-g-0} and the property \eqref{eq-tilh} of $\widetilde{H}_0$, the function $W_{\kappa}$ satisfies
\begin{equation}\label{eq-hd-5}
\begin{split}
-\Delta_w W_{\kappa} (w)& = {\kappa}^{p(n-2)} (-\Delta \widetilde{H}) ( {\kappa}w, {\kappa}e_n)
\\
& = {\kappa}^{p(n-2)} \left[ G^p ({\kappa}w, {\kappa}e_n) - \frac{c_n^p}{|{\kappa}(w-e_n)|^{(n-2)p}} + p H( {\kappa}w, {\kappa}e_n) c_n^{p-1} \frac{1}{|{\kappa}(w-e_n)|^{(n-2)(p-1)}}\right.
\\
&\qquad \left. + \frac{2p c_n^{p-1} \nabla_1 H({\kappa}w, {\kappa}e_n)}{[(n-2)(p-2)-2]} \frac{(-1)({\kappa}(w-e_n))}{|{\kappa}(w-e_n)|^{(n-2)(p-1)}}\right].
\end{split}
\end{equation}
Set the difference $R_{\kappa} : \Omega_{\kappa} \rightarrow \mathbb{R}$ by  $R_{\kappa} (x)= W_0 (x)- W_{\kappa} (x)$ for $x \in \Omega_{\kappa}$. Then, it suffices to show that $R_{\kappa} \rightarrow 0$ in $C_{loc}^1 (\mathbb{R}^{n+1}_{+})$. By \eqref{eq-hd-5} and \eqref{eq-h0-1} we have 
\begin{equation}\label{eq-hd-6}
\begin{split}
&(-\Delta_w) R_{\kappa} (w) 
\\
&= {\kappa}^{p(n-2)} G^p ({\kappa}w, {\kappa}e_n) - \left( \frac{c_n}{|w-e_n|^{n-2}} - \frac{c_n}{|w+e_n|^{n-2}}\right)^p
\\
&\quad + p H({\kappa}w, {\kappa}e_n ) c_n^{p-1} \frac{{\kappa}^{n-2}}{|w-e_n|^{(n-2)(p-1)}} 
\\
&\quad + \frac{2p \nabla_1 H({\kappa}w, {\kappa}e_n) c_n^{p-1}}{[(n-2)(p-2)-2]} (-1) ({\kappa}w-{\kappa}e_n) |{\kappa}w -{\kappa}e_n|^{-(n-2)(p-1)}{\kappa}^{(n-2)p}
\\
&\quad - p \frac{c_n^p}{|w+e_n|^{n-2}}\cdot \frac{1}{|w-e_n|^{(n-2)(p-1)}} - \frac{2p(n-2)c_n^p}{[(n-2)(p-2)-2]} \frac{w+e_n}{|w+e_n|^n} \frac{w-e_n}{|w-e_n|^{(n-2)(p-1)}}.
\end{split}
\end{equation}
By Lemma \ref{lem-h-asym} we have
\begin{equation*}
\left\{\begin{aligned}
{\kappa}^{n-2}H({\kappa}w,{\kappa}e_n) &= \frac{c_n}{|w+e_n|^{n-2}} + T_{\kappa} (w),
\\
{\kappa}^{n-1}\nabla_1 H({\kappa}w,{\kappa}e_n)&= -\frac{c_n (n-2)(w+e_n)}{|w+e_n|^n} + \widetilde{T}_{\kappa} (w),
\end{aligned}\right.
\end{equation*}
where 
\begin{equation}\label{eq-hd-11} 
T_{\kappa} (w) = O \left( \frac{\textbf{d}({\kappa}e_n)}{|w+e_n|^{n-2}}\right)
\quad \textrm{and}\quad \widetilde{T}_{\kappa} (w) =  O \left( \frac{{\kappa}\cdot \mathbf{d}({\kappa} e_n)}{\mathbf{d}({\kappa}w)|w+e_n|^{n-2}}\right).
\end{equation}
Also we have
\begin{equation*}
{\kappa}^{n-2}G({\kappa}w,{\kappa}e_n) = \frac{c_n}{|w-e_n|^{n-2}} - \frac{c_n}{|w+e_n|^{n-2}} - T_{\kappa} (w).
\end{equation*}
Inserting these formulas into \eqref{eq-hd-6} and arranging them, we get
\begin{equation*}
\begin{split}
&(-\Delta_w) R_{\kappa} (w) 
:= I_1 (w) + I_2 (w),
\end{split}
\end{equation*} 
where
\begin{equation}\label{eq-hd-3}
\begin{split}
I_1 (w)&= \left( \frac{c_n}{|w-e_n|^{n-2}} - \frac{c_n}{|w+e_n|^{n-2}} - T_{\kappa} (w)\right)^p - \left( \frac{c_n}{|w-e_n|^{n-2}}-\frac{c_n}{|w+e_n|^{n-2}}\right)^p 
\\
&\quad+ pT_{\kappa} (w) \frac{c_n^{p-1}}{|w-e_n|^{(n-2)(p-1)}}
\end{split}
\end{equation}
and
\begin{equation}\label{eq-hd-14}
I_2 (w)= {2p(n-2)}{c_n}^{p-1} \widetilde{T}_{\kappa} (w) \frac{w-e_n}{|w-e_n|^{(n-2)(p-1)}}.
\end{equation}
We split the function $R_{\kappa}$ into $R_{\kappa} = \mathcal{R}_{\kappa}^1 + \mathcal{R}_{\kappa}^2$, where 
\begin{equation}\label{eq-hd-10}
\left\{ \begin{array}{ll} -\Delta \mathcal{R}_{\kappa}^1 (x) = -\Delta{R}_{\kappa} (x)&\quad \textrm{in}~\Omega_{\kappa},
\\
\mathcal{R}_{\kappa}^1 (x) =0&\quad \textrm{on}~\partial \Omega_{\kappa},
\end{array}
\right.
\quad \textrm{and}\quad 
\left\{ \begin{array}{ll} -\Delta \mathcal{R}_{\kappa}^2 (x) =0&\quad \textrm{in}~\Omega_{\kappa},
\\
\mathcal{R}_{\kappa}^2 (x) = \mathcal{R}_{\kappa} (x)&\quad \textrm{on}~\partial \Omega_{\kappa}.
\end{array}
\right.
\end{equation}
For the proof, first we shall show that $\mathcal{R}_{\kappa}^1 (x)$ and $\mathcal{R}_{\kappa}^2 (x)$ converge to zero in $C_0 (B(e_n, 1/3))$.

\

\noindent \textbf{$\bullet$ $C^0$ convergence of $\mathcal{R}_{\kappa}^1$.} We shall show that $R_{\kappa}^1 (x) \rightarrow 0$ in $C^1 (B(e_n,1/3))$. By \eqref{eq-hd-10} we have
\begin{equation}\label{eq-hd-12}
\begin{split}
\mathcal{R}_{\kappa}^1 (x) &=\int_{\Omega_{\kappa}} G_{\kappa} (x,w) (-\Delta)R_{\kappa} (w) dw 
\\
&= \int_{|w-e_n| \leq \frac{1}{2}}G_{\kappa} (x,w) (I_1 + I_2)(w) dw + \int_{|w-e_n| > \frac{1}{2}} G_{\kappa} (x,w) (I_1 + I_2)(w) dw.
\end{split}
\end{equation}
We aim to show that the above value goes to zero as ${\kappa} \rightarrow 0$ uniformly for $x$ in any given compact set.
\

Let us consider first the integration on the region $|w-e_n| \leq \frac{1}{2}$. For $r \in (0,1)$ and $|a| \leq \frac{1}{2}$ we have $(1+ar)^p = 1 + p ar + O((ar)^2)$, which leads to
\begin{equation}\label{eq-hd-1}
\left( \frac{1}{r} + a\right)^p = \frac{1}{r^p} + \frac{ap}{r^{p-1}} + O (a^2 r^{2-p})\quad \forall r \in (0,1)\quad a \in (-\frac{1}{2}, \frac{1}{2}).
\end{equation}
For $|w-e_n| \leq \frac{1}{2}$ we have
\begin{equation*}
\frac{1}{|w-e_n|^{n-2}} - \frac{1}{|w+e_n|^{n-2}} \geq \frac{1}{2|w-e_n|^{n-2}}\geq 2^{n-3}.\end{equation*}
Hence we may apply \eqref{eq-hd-1} to get
\begin{equation*}
\begin{split}
& \left[ \frac{c_n}{|w-e_n|^{n-2}} - \frac{c_n}{|w+e_n|^{n-2}} - T_{\kappa} (w)\right]^p
\\
&=\left[ \frac{c_n}{|w-e_n|^{n-2}} - \frac{c_n}{|w+e_n|^{n-2}} \right]^p - p T_{\kappa} (w)\left[ \frac{c_n}{|w-e_n|^{n-2}} - \frac{c_n}{|w+e_n|^{n-2}}\right]^{p-1}
\\
&\quad + O\left(T_{\kappa} (w)^2 |w-e_n|^{(n-2)(2-p)}\right).
\end{split}
\end{equation*}
Putting this into \eqref{eq-hd-3} we get 
\begin{equation}\label{eq-hd-4}
\begin{split}
I_1 (w) & = p T_{\kappa} (w)\left[ \frac{c_n^{p-1}}{|w-e_n|^{(n-2)(p-1)}} - \left(  \frac{c_n}{|w-e_n|^{n-2}} - \frac{c_n}{|w+e_n|^{n-2}}\right)^{p-1}\right]
\\
&\quad + O\left(T_{\kappa} (w)^2 |w-e_n|^{(n-2)(2-p)}\right).
\end{split}
\end{equation}
We have $(1+ar)^{p-1} = 1 + O(r)$ for $r \in (0,\frac{1}{4}$ and $a \in (-2,2)$. Hence
\begin{equation*}
\left( \frac{1}{r}+a\right)^{p-1} = \frac{1}{r^{p-1}} + O(r^{p-2})\quad \forall a \in(-2,2)\quad \textrm{and}\quad r \in (0, \frac{1}{4}).
\end{equation*}
Using this and \eqref{eq-hd-11} we can estimate \eqref{eq-hd-4} as
\begin{equation}\label{eq-hd-7}
\begin{split}
I_1 (w)&\leq C T_{\kappa} (w)  |w-e_n|^{(n-2)(2-p)} + O(T_{\kappa} (w)^2 |w-e_n|^{(n-2)(2-p)})
\\
&\leq \frac{{\kappa}}{|w+e_n|^{n-2}} |w-e_n|^{(n-2)(2-p)} = O(1).
\end{split}
\end{equation}
Therefore
\begin{equation*}
\begin{split}
& \left|\int_{B(e_n, \frac{1}{2})} G(x,w) I_1 (w) dw\right|
\\
&\quad \leq {\kappa} \int_{B(e_n, \frac{1}{2})} \frac{1}{|x-w|^{n-2}} \frac{1}{|w+e_n|^{n-2}} |w-e_n|^{(n-2)(2-p)} dw,
\end{split}
\end{equation*}
which converges to zero as ${\kappa} \rightarrow 0$. 
For $I_2$, we have the bound
\begin{equation*}
\left| \int_{|w-e_n| \leq \frac{1}{2}} G_k (x,w) I_2 (w) dw \right| \leq C \int_{B(e_n, \frac{1}{2})} \frac{1}{|x-w|^{n-2}} \frac{{\kappa}}{|w+e_n|^{n-2}} \frac{|w-e_n|}{|w-e_n|^{(n-2)(p-1)}} dw,
\end{equation*}
which goes to zero as ${\kappa} \rightarrow 0$ since $(n-2)+(n-2)(p-1) = p(n-2) <n$. Thus we have 
\begin{equation*}
\lim_{{\kappa} \rightarrow 0} \int_{|w-e_n| \leq \frac{1}{2}} G_{\kappa} (x,w) (-\Delta) R_{\kappa} (w) dw = 0.
\end{equation*}
Next we turn to estimate the second integration of \eqref{eq-hd-12}. First we note that there is a constant $C>1$ such that 
\begin{equation*}
\frac{1}{C}|w-e_n| \leq |w +e_n| \leq C|w-e_n|\quad \textrm{for}~ w \in B\left(e_n, \frac{1}{2}\right)^{c}.
\end{equation*}
Hence we easily see from \eqref{eq-hd-3} that
\begin{equation*}
I_1 (w) \leq CT_{\kappa} (w) \frac{1}{|w+e_n|^{(n-2)(p-1)}}.
\end{equation*}
Using this and \eqref{eq-hd-11} we get 
\begin{equation*}
\int_{|w-e_n| \geq \frac{1}{2}} G_{\kappa} (x,w)I_1 (w) dw \leq C {\kappa}\int_{B(e_n, \frac{1}{2})^{c}}  \frac{1}{|w+e_n|^{(n-2)p}} \frac{1}{|x-w|^{n-2}} dw,
\end{equation*}
which goes to zero as ${\kappa} \rightarrow 0$ because $p  > \frac{2}{n-2}$. Next we use \eqref{eq-hd-11} to estimate
\begin{equation*}
\begin{split}
& \int_{|w-e_n| \geq \frac{1}{2}} G_{\kappa} (x,w) I_2 (w) dw 
\\
&\leq C \int_{B(e_n, 1/2)^{c}} \frac{{\kappa}\cdot {\kappa}}{|w+e_n|^{n-2}\mathbf{d}({\kappa}w)} \frac{|w+e_n|}{|w+e_n|^{(n-2)(p-1)}} \frac{\mathbf{d}({\kappa}x)\mathbf{d}({\kappa}w)}{{\kappa}^2 |x-w|^{n}} dw
\\
&= C \mathbf{d}({\kappa}x)\int_{B(e_n, 1/2)^{c}} \frac{1}{|w+e_n|^{n-2}} \frac{|w+e_n|}{|w+e_n|^{(n-2)(p-1)}} \frac{1}{ |x-w|^{n}} dw,
\end{split}
\end{equation*}
which goes to zero as ${\kappa} \rightarrow 0$. Thus we have
\begin{equation*}
\lim_{{\kappa} \rightarrow 0} \int_{\Omega_{\kappa}} G_{\kappa} (x,w) (-\Delta)R_{\kappa} (w) dw = 0.
\end{equation*}
\textbf{$\bullet$ $C^0$ convergence of $\mathcal{R}_{\kappa}^2$.}
Since $\mathcal{R}_{\kappa}^2$ is harmonic in $\Omega_{\kappa}$, we only need to show that 
\begin{equation*}
\lim_{{\kappa} \rightarrow 0} \sup_{x \in \Omega_{\kappa}} |\mathcal{R}_{\kappa}^2 (x)| = 0.
\end{equation*}
Recall that the boundary $\partial \Omega \cap B(0,r)$ is assumed to be parametrized by
\begin{equation*} 
(\bar{z}, f(\bar{z}))\quad \textrm{for}~\bar{z}=(z_1,\cdots, z_{n-1}),
\end{equation*}
where $f :\mathbb{R}^{n-1}\rightarrow \{ y \geq 0\}$ satisfies $f(0)=0$ and $\nabla f(0)  =0$. Consequently, we can parametrize  boundary point $z_{\kappa}$ on $\partial \Omega_{\kappa}$ by 
\begin{equation*}z_{\kappa}=(z_1, \cdots, z_{n-1}, f({\kappa} z)/{\kappa}) = (\bar{z}, f({\kappa}\bar{z})/{\kappa})
\end{equation*}
Using this and \eqref{eq-h1-2} we have
\begin{equation}\label{eq-hd-13}
\begin{split}
W_{\kappa} (z_{\kappa}) &= \widetilde{H} (({\kappa}z_1, \cdots, {\kappa}z_{n-1}, f({\kappa} z)), {\kappa}e_n ) {\kappa}^{p(n-2)-2}
\\
& = {\kappa}^{p(n-2)-2} \left[ \alpha_1 \frac{{\kappa}^{-[(n-2)p-2]}}{|(z-e_n)|^{(n-2)p-2}}- \alpha_2 p H(({\kappa}z, f({\kappa}z)), {\kappa}_n) c_n^{-1}  \frac{{\kappa}^{-[(n-2)(p-1)-2]}}{|(z-e_n)|^{(n-2)(p-1)-2}}.\right]
\\
& = \left[ \frac{\alpha_1}{|(z-e_n)|^{(n-2)p-2}}- \alpha_2 p \cdot {\kappa}^{n-2}H(({\kappa}z, f({\kappa}z)), {\kappa}_n) c_n^{-1}  \frac{1}{|(z-e_n)|^{(n-2)(p-1)-2}}.\right]
\end{split}
\end{equation}
Since $f({\kappa}z)= O(({\kappa}z)^2)$, we have $\lim_{{\kappa} \rightarrow 0} z_{\kappa} =(\bar{z},0)$. Combining this with \eqref{eq-h1-4} we find
\begin{equation*}
\lim_{{\kappa} \rightarrow 0} - {\kappa}^{(n-2)} H(({\kappa}z, f({\kappa}z)), {\kappa}_n)= \frac{c_n}{|((\bar{z},0)-e_n)|^{n-2}}.
\end{equation*}
Injecting this into \eqref{eq-hd-13} we get
\begin{equation*}
\lim_{{\kappa} \rightarrow 0} W_{\kappa} (z_{\kappa})=\left\{\begin{array}{ll} \left( \alpha_1 -\alpha_2\right) \frac{1}{|((\bar{z},0)-e_n)|^{(n-2)p-2}}& \textrm{if}~ p \in \left[ \frac{n-1}{n-2}, \frac{n}{n-2}\right]
\\
\frac{\alpha_1}{|((\bar{z},0)-e_n)|^{(n-2)p-2}}&\quad \textrm{if}~ p \in \left[ 1, \frac{n-1}{n-2}\right),
\end{array}
\right.
\end{equation*}
which implies
\begin{equation*}
\begin{split}
\lim_{{\kappa} \rightarrow 0} \sup_{z_{\kappa} \in \partial \Omega_{\kappa} \cap B(0, r/{\kappa})} |\mathcal{R}_{\kappa}^2 (z_{\kappa}) |&= \lim_{{\kappa} \rightarrow 0}\sup_{z_{\kappa} \in \partial \Omega_{\kappa} \cap B(0, r/{\kappa})} |R_{\kappa} (z_{\kappa})|
\\
&= \lim_{{\kappa} \rightarrow 0}\sup_{z_{\kappa} \in \partial \Omega_{\kappa} \cap B(0, r/{\kappa})} |W_0 (z_{\kappa}) - W_{\kappa} (z_{\kappa})| = 0.
\end{split}
\end{equation*}
This implies that $\lim_{{\kappa} \rightarrow 0}\sup_{z \in \Omega_{\kappa}} |\mathcal{R}_{\kappa}^2(z)| =0$ since $\mathcal{R}_{\kappa}^2$ is harmonic in $\Omega_{\kappa}$.
\

Combining the above two convergence results, we can deduce that $R_{\kappa} (x) \rightarrow 0$ uniformly for $|x-e_n| < \frac{1}{3}$. 

\

\noindent \textbf{$\bullet$ The $C^1$ convergence of $R_{\kappa}$.} 
From \eqref{eq-hd-7} and \eqref{eq-hd-14} we know that
\begin{equation*}
(-\Delta) R_{\kappa} (x) = (I_1 + I_2) (x) = O \left( \frac{1}{|x-e_n|^{(n-2)(p-1)-1}}\right) \quad \textrm{for}~ |x-e_n| \leq \frac{1}{2}.
\end{equation*}
This estimate implies $(-\Delta) R_{\kappa} (x) \in L^{n+\alpha}(B(e_n,1/3))$ for some $\alpha >0$ since $(n-2)(p-1)-1 <1$ for $p < \frac{n}{n-2}$. Therefore  $R_{\kappa}$ is contained in $C^{1,\beta} (B(e_n,1/4))$ uniformly in ${\kappa}>0$ for some $\beta >0$. Thus $R_{\kappa}$ converges to some function $f$ in $C^1 (B(e_n, 1/4))$. Actually we have $f\equiv 0$ since $R_{\kappa}$ converges to $0$ in $C^0 (B(e_n,1/2))$.
The lemma is proved.
\end{proof}
\begin{lem}\label{lem-h0}~
\begin{enumerate}
\item Assume $p \in \left[ 1, \frac{n-1}{n-2}\right)$. Then we have $\frac{\partial}{\partial x_n} W_0 (e_1) \neq 0$ if and only if \eqref{eq-as-0} holds
\item Assume $p \in \left[ \frac{n-1}{n-2}, \frac{n}{n-2}\right).$ Then we have $\frac{\partial}{\partial x_n} W_0 (e_1) \neq 0$ if and only if \eqref{eq-as-1} holds.
\end{enumerate} 
\end{lem}
\begin{proof}
We shall only prove the second statement since the same argument applies for the first statement. For the proof, we recall that the explicit formula of Green's function on the half space $\mathbb{R}^{n}_{+}$ is given by
\begin{equation*}
 \left( \frac{c_n}{|x-y|^{n-2}} - \frac{c_n}{|x-y^*|^{n-2}}\right)
\end{equation*}
and the Poisson kernel is equal to
\begin{equation*}
(n-2)c_n \frac{2x_n}{|x-y|^{n}}
\end{equation*}
for $y \in \partial \mathbb{R}^n_{+}$. Therefore, reminding \eqref{eq-h0-1} and \eqref{eq-h0-2}, we find
\begin{equation*}
\begin{split}
W_0 (x)&= \int_{\mathbb{R}^{n}_{+}} \left( \frac{c_n}{|z-x|^{n-2}} - \frac{c_n}{|z-x^*|^{n-2}}\right) (-\Delta_z W_0 )(z) dz 
\\
&\quad + \int_{\partial \mathbb{R}^n_{+}} (n-2)c_n\frac{2x_n}{|x-y|^n} \frac{\alpha_1 -\alpha_2}{|(y-e_n)|^{(n-2)p-2}} dy
\\
&:= I_1 (x) + I_2 (x).
\end{split}
\end{equation*}
First we compute
\begin{equation*}
\frac{\partial I_2}{\partial x_n}(x) = (n-2)c_n \int_{\partial \mathbb{R}^n_{+}} \left( \frac{2}{|x-y|^{n}} - 2n \frac{x_n^2}{|x-y|^{n+2}} \right) \frac{\alpha_1 -\alpha_2}{|(y-e_n)|^{(n-2)p-2}} dy.
\end{equation*}
Taking $x=e_n$ here, we find
\begin{equation*}
\frac{\partial I_2}{\partial x_n}(e_n) = (n-2)c_n (\alpha_1 - \alpha_2) \int_{\mathbb{R}^n_{+}} \left( \frac{2}{|(y-e_n)|^{(n-2)(p+1)}} - \frac{2n}{|(y-e_n)|^{(n-2)p+n}}\right) dy.
\end{equation*}
Next we use \eqref{eq-h0-1} to get
\begin{equation*}
\begin{split}
\frac{\partial I_1}{\partial x_n}(x) &= -\int_{\mathbb{R}^{n}_{+}} (n-2)c_n \left( \frac{(x_n - z_n)}{|x - z|^{n}} - \frac{(x_n +  z_n)}{|z-x^*|^{n-2}}\right)
\\
&\qquad \times \left[\left( \frac{1}{|z-e_n|^{n-2}} - \frac{1}{|z+e_n|^{n-2}}\right)^p - \frac{1}{|z-e_n|^{(n-2)p}} + p \frac{1}{|z-e_n|^{(n-2)(p-1)}} \cdot \frac{1}{|z+e_n|^{n-2}}\right.
\\
&\quad \quad\quad+\left. \frac{2p(n-2)}{[(n-2)(p-2)-2]} \frac{(z+e_n)}{|z+e_n|^n} \frac{z-e_n}{|z-e_n|^{(n-2)(p-1)}}\right] dz.
\end{split}
\end{equation*}
Letting $x=e_n$ we obtain 
\begin{equation*}
\begin{split}
\frac{\partial I_1}{\partial x_n}(e_n) &= -\int_{\mathbb{R}^{n}_{+}}(n-2)c_n \left( \frac{(1- z_n)}{|e_n - z|^{n}} - \frac{(1 +  z_n)}{|z+e_n|^{n-2}}\right)
\\
&\qquad \times \left[\left( \frac{1}{|z-e_n|^{n-2}} - \frac{1}{|z+e_n|^{n-2}}\right)^p - \frac{1}{|z-e_n|^{(n-2)p}} + p \frac{1}{|z-e_n|^{(n-2)(p-1)}} \cdot \frac{1}{|z+e_n|^{n-2}}\right.
\\
&\quad \quad\quad+\left. \frac{2p(n-2)}{[(n-2)(p-2)-2]} \frac{(z+e_n)}{|z+e_n|^n} \frac{z-e_n}{|z-e_n|^{(n-2)(p-1)}}\right] dz.
\end{split}
\end{equation*}
The above estimates shows that $\frac{\partial}{\partial x_n} W_0 (e_1) \neq 0$ holds provided by this
the integration \eqref{eq-as-1} is not zero. The lemma is proved.
\end{proof}
\begin{proof}[Proof of Lemma \ref{lem-as}]
The proof follows immediately from combining Lemma \ref{lem-hd} and Lemma \ref{lem-h0}.
\end{proof}
Now we shall finish the proof of Theorem \ref{thm-3}. 
\begin{proof}[Proof of Theorem \ref{thm-3}]
For fixed dimension $n > 4$ we denote by $W_{0,p}$ the function $W_0$ defined in \eqref{eq-h0-1} for given $p$. In Lemma \ref{lem-h0} we confirmed that \eqref{eq-as-0} holds if and only if $\frac{\partial}{\partial x_n} W_{0,p}(e_n) \neq 0$. Having this in mind, to prove the theorem, we shall prove that $\frac{\partial}{\partial x_n}W_{0,p} (e_n) < 0$ holds for $p =1$ and show the continuity property of the integrations in \eqref{eq-as-0} with respect to $p \in \left[1, \frac{n-1}{n-2}\right)$. This implies  an existence of a value $\delta >0$ such that $\frac{\partial}{\partial x_n} W_{0,p} (e_n) < 0$ for $p \in [1, 1+\delta]$. 
\

\noindent \textbf{Continuity of the integrations of \eqref{eq-as-0} with respect to $p$.} We aim to check that the derivation of the integration values of \eqref{eq-as-0} with respect to $p$ is bounded uniformly for $p \in [1,1+\delta)$ for small $\delta >0$. For this, we let
\begin{equation*}
F(p) =\int_{\mathbb{R}^{n}_{+}} \left[ \frac{(1-z_n)}{|z-e_n|^{n}} - \frac{(1+z_n)}{|z+e_n|^n}\right]\left\{ \left( \frac{1}{|z-e_n|^{n-2}} - \frac{1}{|z+e_n|^{n-2}}\right)^p - \frac{1}{|z-e_n|^{(n-2)p}}\right\}dz
\end{equation*}
and
\begin{equation*}
 G(p)= (n-2)\alpha_1\left[\int_{\partial \mathbb{R}^n_{+}} \frac{2}{|(y-e_n)|^{(n-2)(p+1)}} dy - \int_{\mathbb{R}^{n}_{+}} \frac{2n}{|(y-e_n)|^{(n-2)p+n}} dy\right].
\end{equation*}
Let $f(a) = a^p \log a$. Then, 
\begin{equation*}
\begin{split}
F' (p)&=\int_{\mathbb{R}^{n}_{+}} \left[ \frac{(1-z_n)}{|z-e_n|^{n}} - \frac{(1+z_n)}{|z+e_n|^n}\right]\left\{ f \left(\frac{1}{|z-e_n|^{n-2}} - \frac{1}{|z+e_n|^{n-2}}\right) - f\left(\frac{1}{|z-e_n|^{(n-2)p}}\right)\right\}dz
\\
& = \int_{|z-e_n| \leq \frac{1}{2}} dz + \int_{|z-e_n| \geq \frac{1}{2}} dz.
\end{split}
\end{equation*}
Since there is no singularity in the region $|z-e_n| \geq \frac{1}{2}$ and decaying is good enough, we have $\int_{|z-e_n| \geq \frac{1}{2}} dz = O(1)$. For $|z-e_n| \leq \frac{1}{2}$ we use the mean value theorem to see that
\begin{equation*}
\begin{split}
&f \left(\frac{1}{|z-e_n|^{n-2}} - \frac{1}{|z+e_n|^{n-2}}\right) - f\left(\frac{1}{|z-e_n|^{(n-2)p}}\right)
\\
&\qquad\qquad\leq \frac{C}{|z+e_n|^{n-2}} \left[ \frac{1}{|z-e_n|^{(n-2)(p-1)}} \log \left( \frac{1}{|z-e_n|}\right)\right].
\end{split}
\end{equation*}
Using this we can estimate
\begin{equation*}
\int_{|z-e_n| \leq \frac{1}{2}} \frac{C}{|z-e_n|^{(n-1)+(n-2)(p-1)}} dz,
\end{equation*}
which is bounded uniformly for $p \in \left[1, \frac{n-1}{n-2}-\delta\right)$ with any fixed $\delta >0$. As for $G' (p)$, it is much easier to check the integration is uniformly bounded since there is no singularity in the integrand. Therefore $F'(p)$ and $G'(p)$ are uniformly bounded for $p \in \left[1,\frac{n-1}{n-2}-\delta\right)$ and so $F(p)$ and $G(p)$ are continuous in the interval.

\

\noindent \textbf{The condition \eqref{eq-as-0} holds for $p=1$.} Note that for $p=1$ the function $\widetilde{H}_0 (x,y)$ is given by
\begin{equation}\label{eq-eh-1}
\widetilde{H}_0 (x,y) = \frac{\alpha_1}{|x-y|^{(n-4)}}- \widetilde{G}_0 (x,y)\quad x \in \mathbb{R}^{n}_{+}.
\end{equation}
Here $\widetilde{G}_0 (x,y)$ is given by
\begin{equation*}
\left\{ \begin{array}{ll}-\Delta_x \widetilde{G}_0 (x,y) = G_0 (x,y),&\quad x \in \mathbb{R}^{n}_{+},
\\
\widetilde{G}_0 (x,y) =0&\quad x \in \mathbb{R}^{n-1} \times \{0\},
\end{array}
\right.
\end{equation*}
where 
\begin{equation*}
G_0 (x,y) = \frac{c_n}{|x-y|^{n-2}} - \frac{c_n}{|x-y^*|^{n-2}}.
\end{equation*}
It has the expression 
\begin{equation*}
\widetilde{G}_0 (x,y) = \int_{\mathbb{R}^{n}_{+}} G_0 (x,z) G_0 (z,y) dz.
\end{equation*}
Since $G_0 (x,y)$ is symmetric, i.e., $G_0 (x,y) = G_0 (y,x)$ for all $x, y \in \mathbb{R}^{n}_{+}$, so is $\widetilde{G}_0 (x,y)$. Thus, we see from \eqref{eq-eh-1} that $\widetilde{H}_0 (x,y) = \widetilde{H}_0 (y,x)$. 
\

We note that
\begin{equation}\label{eq-eh-2}
\begin{split}
\widetilde{H}_0 (t e_n, te_n) & =\int_{\mathbb{R}^n_{+}} \left( \frac{c_n}{|z-te_n|^{n-2}} - \frac{c_n}{|z+te_n|^{n-2}} \right) \frac{c_n}{|z+te_n|^{n-2}} dz
\\
&\quad +(n-2)c_n \int_{\partial \mathbb{R}^{n}_{+}}\frac{2t}{|te_n-y|^n} \frac{1}{2(n-4)}\frac{1}{|y-te_n|^{n-4}} dy
\\
& = t^{4-n} \int_{\mathbb{R}^n_{+}}\left( \frac{c_n}{|z-e_n|^{n-2}} - \frac{c_n}{|z+e_n|^{n-2}} \right) \frac{\alpha_1}{|z+e_n|^{n-2}} dz
\\
&\quad + t^{4-n}(n-2)c_n \int_{\partial \mathbb{R}^{n}_{+}} \frac{2}{n\alpha (n)} \frac{1}{|e_n-y|^n} \frac{1}{2(n-4)}\frac{1}{|y-e_n|^{n-4}} dy
\\
&= t^{4-n} \widetilde{H}_0 (e_n, e_n).
\end{split}
\end{equation}
Combining the symmetric property with \eqref{eq-eh-2}, we get
\begin{equation}\label{eq-eh-3}
\begin{split}
\left.\frac{\partial}{\partial x_n}W_0 (x) \right|_{x=e_n}&=
\left(\frac{\partial}{\partial x_n} \widetilde{H}_0 (x, e_n) \right)_{x=e_n}
\\
& =\frac{1}{2}\left( \frac{\partial}{\partial x_n} \widetilde{H}_0 (x,x)\right)_{x=e_n}
\\
& = \frac{1}{2}\left.\frac{\partial}{\partial t} \widetilde{H}_0 (te_n, te_n)\right|_{t=1} = \frac{(4-n)}{2} \widetilde{H}(e_n, e_n).
\end{split}
\end{equation}
It is easy to check from \eqref{eq-eh-1} that $(-\Delta)\widetilde{H}_0 >0$ in $\mathbb{R}^n_{+}$ and $\widetilde{H}_0 >0$ on $\partial \mathbb{R}^{n}_{+}$. Thus we have $\widetilde{H}_0 (e_n, e_n) \neq 0$. Combining this fact with \eqref{eq-eh-3} we deduce that $\frac{\partial}{\partial x_n} W_{0,1}(e_n) < 0$, which implies that \eqref{eq-as-0} holds for $p=1$. The proof is finished.
\end{proof}

\appendix

\section{The proof of Lemma \ref{lem-h-asym}}
In this appendix, we give the proof of Lemma \ref{lem-h-asym} stated as follows.
\begin{lem}
For $(x,y) \in \Omega \times \Omega$ we have
\begin{equation}\label{a-eq-h1-4}
H(x,y) = \frac{c_n}{|x-y^*|^{n-2}} + O \left( \frac{\mathbf{d}(y)}{|x-y^*|^{n-2}}\right),
\end{equation}
and
\begin{equation}\label{a-eq-h1-5}
\nabla_x H(x,y) = - \frac{(n-2)c_n (x-y^*)}{|x-y^*|^{n}} + O \left( \frac{\mathbf{d}(y)}{\mathbf{d}(x)|x-y^*|^{n-2}}\right).
\end{equation}
\end{lem}
\begin{proof}
Fix a point $y \in \Omega$. For simplicity we assume that $0 \in \partial \Omega$, $\mathbf{d}(y) = \textrm{dist}(y,0)$, and $y = (0,\cdots, 0, \mathbf{d}(y))$. Then, for a small value $r>0$ the boundary $\partial \Omega \cap B(0,r)$ is parametrized as $w= (z_1, \cdots, z_{n-1}, f(z))$ with a function $f : \mathbb{R}^{n-1} \rightarrow \overline{\mathbb{R}^{+}}$ such that $f(0)=0$ and $\nabla f(0) =0$. We note that  $y^* = (0,\cdots, 0, -\mathbf{d}(y))$. It will be clear in the proof that this setting does not lose any generality.

Now we set $D: \Omega \times \Omega \rightarrow \mathbb{R}$ by
\begin{equation*}
D(x,y) = H(x,y) - \frac{c_n}{|x-y^*|^{n-2}},
\end{equation*}
and it then suffices to estimate this function for our purpose. The function $D(x,y)$ satisfies 
\begin{equation*}
\left\{ \begin{array}{ll}
-\Delta_x D(x,y) = 0&\quad x \in \Omega,
\\
D(x,y) = \frac{c_n}{|x-y|^{n-2}} - \frac{c_n}{|x-y^*|^{n-2}}&\quad x \in \partial \Omega.
\end{array}
\right.
\end{equation*}
Denote by $K(x,y): \Omega \times \Omega \rightarrow \mathbb{R}$ the Neumann kernel of $\Omega$ such that we have 
\[g(x) = \int_{\partial \Omega} K(x,y) f(y) dS_y\]
whenever $(-\Delta) g (x) =0$ in $\Omega$ and $g(x) = f(x)$ on $\partial \Omega$ under suitable regularity assumptions on $f$ and $g$.
 Then it is well-known that $|K(x,w)| \leq \frac {C \mathbf{d}(x)}{|x-w|^{n}}$ for some constant $C= C(\Omega) >0$. Using this we can estimate $D(x,y)$ as 
\begin{equation}\label{eq-h-5}
\begin{split}
\left|D(x,y)\right| &= \left|\int_{\partial \Omega} K(x,w) \left[ \frac{c_n}{|w-y|^{n-2}} - \frac{c_n}{|w-y^*|^{n-2}}\right] dS_w\right|
\\
&\leq C \int_{\partial \Omega} \frac{\mathbf{d}(x)}{|x-w|^{n}} \left| 
\frac{1}{|w-y|^{n-2}} - \frac{1}{|w-y^*|^{n-2}}\right| dS_w
\\
&= C \int_{\partial \Omega \cap B^n(0,r)} \frac{\mathbf{d}(x)}{|x-w|^{n}} \left| 
\frac{1}{|w-y|^{n-2}} - \frac{1}{|w-y^*|^{n-2}}\right|  dS_{w} 
\\
&\quad\quad+ C \int_{\partial \Omega \cap B^n(0,r)^{c}}\frac{\mathbf{d}(x)}{|x-w|^{n}} \left| 
\frac{1}{|w-y|^{n-2}} - \frac{1}{|w-y^*|^{n-2}}\right|  dS_{w}
\\
&=: I_1 + I_2.
\end{split}
\end{equation}
It is easy to see that for a universal constant $C = C(\Omega)>0$ we have
\begin{equation}\label{eq-h-6}
\frac{1}{C} |w-y| \leq |w-y^{*}|\leq C|w-y|\quad \forall w \in \partial \Omega.
\end{equation}
We then can  deduce that
\begin{equation}\label{eq-h-2}
\begin{split}
\left| \frac{1}{|w-y|^{n-2}} - \frac{1}{|w-y^*|^{n-2}}\right| &= \frac{||w-y|^{n-2}- |w-y^*|^{n-2}|}{|w-y|^{n-2}|w-y^*|^{n-2}}
\\
& \leq \frac{C|w-y|^{n-3} ||w-y|-|w-y^*||}{|w-y|^{n-2}|w-y^*|^{n-2}}.
\end{split}
\end{equation}
For $w \in \partial \Omega$ we may write $w-y = (w_1, \cdots, w_{n-1}, w_n -\mathbf{d}(y))$ and $w-y^* = (w_1, \cdots, w_{n-1}, w_n + \mathbf{d}(y))$. Hence 
\begin{equation}\label{eq-h-4}
|w-y|-|w-y^*|= \frac{|w-y|^2 - |w-y^*|^2}{|w-y|+ |w-y^*|} = \frac{ -4w_n \mathbf{d}(y)}{|w-y|+|w-y^*|}.
\end{equation}
First we shall estimate $I_2$. Note that $|w-y| \geq \frac{r}{2}$ for all $w \in \partial \Omega \cap B(0,r)^{c}$. Hence we may deduce from \eqref{eq-h-4} that
\begin{equation*}
|w-y|-|w-y^*| \leq C \mathbf{d}(y),
\end{equation*}
and also, from \eqref{eq-h-2}, 
\begin{equation*}
\left| \frac{1}{|w-y|^{n-2}} - \frac{1}{|w-y^*|^{n-2}}\right|  \leq C \mathbf{d}(y).
\end{equation*}
Using this we get
\begin{equation}\label{eq-h-10}
\begin{split}
I_2 &=C\int_{\partial \Omega \cap B^n (0,r)^{c}}\frac{\mathbf{d}(x)}{|x-w|^{n}} \left| 
\frac{1}{|w-y|^{n-2}} - \frac{1}{|w-y^*|^{n-2}}\right|  dS_{w}
\\
& =C \int_{\partial \Omega \cap B^n (0,r)^{c}} \frac{\mathbf{d}(x)}{|x-w|^{n}}  O(\mathbf{d}(y)) dS_{w} = O(\mathbf{d}(y)).
\end{split}
\end{equation}
Next we shall estimate $I_1$. Remind that $\partial \Omega \cap B^n (0,r)$ is parametrized by the map $w: B^{n-1} (0,r)\rightarrow \partial \Omega \cap B^n (0,r)$ defined by
\begin{equation*}
w(z) = (z_1,\cdots, z_{n-1}, f(z)) \quad \textrm{for}\quad z=(z_1, \cdots, z_{n-1}) \in \mathbb{R}^{n-1}.
\end{equation*}
Using this we write
\begin{equation}\label{eq-h-1}
w(z)- y = (z_1, \cdots, z_{n-1}, f(z) -\mathbf{d}(y))\quad\textrm{and}\quad w(z) -y^* = (z_1, \cdots, z_{n-1}, f(z) + \mathbf{d}(y)).
\end{equation}
This enables us to get the estimate
\begin{equation}\label{eq-h-3}
\begin{split}
|w-y|-|w-y^*|&= \frac{|w-y|^2 - |w-y^*|^2}{|w-y|+ |w-y^*|} = \frac{ 2 f(z) \mathbf{d}(y)}{|w-y|+|w-y^*|}
\\
&\leq \frac{f(z) \mathbf{d}(y)}{|w-y|}.
\end{split}
\end{equation}
Using the estimates \eqref{eq-h-6}, \eqref{eq-h-2} and \eqref{eq-h-3}, we find 
\begin{equation*}
\begin{split}
\left| \frac{c_n}{|w-y|^{n-2}} - \frac{c_n}{|w-y^*|^{n-2}}\right|  \leq \frac{C f(z) \mathbf{d}(y)}{|w-y|^{n}}
\end{split}
\end{equation*}
Thus, 
\begin{equation*}
\begin{split}
I_1 = \int_{\partial \Omega \cap B^n (0,r)}\frac{\mathbf{d}(x)}{|x-w|^{n}} \left| 
\frac{c_n}{|w-y|^{n-2}} - \frac{c_n}{|w-y^*|^{n-2}}\right|   dS_{w}  &\leq C \int_{B^{n-1} (0,r)} \frac{\mathbf{d}(x)}{|x-w(z)|^{n}} \frac{f(z) \mathbf{d}(y)}{|w(z) -y^*|^{n}} dz.
\end{split}
\end{equation*}
Since $f(0)=0$ and $\nabla f(0)=0$, we have $f(z) = O(|z|^2)$, and we note that 
\begin{equation*}
|z| \leq |w(z) -y^*|\qquad \forall z \in B^{n-1} (0,r).
\end{equation*}
Using this we find
\begin{equation*}
\begin{split}
I_1 &\leq  C\int_{B^{n-1} (0,r)} \frac{\mathbf{d}(x)}{|x-w(z)|^{n}} \frac{|z|^2 \mathbf{d}(y)}{|w(z) -y^*|^{n}} dz
\\
&\leq C \int_{B(0,r)} \frac{\mathbf{d}(x)}{|x-w(z)|^{n}} \frac{\mathbf{d}(y)}{|w(z) -y^*|^{n-2}} dz
\end{split}
\end{equation*}
We have to bound this by $O \left(\frac{\mathbf{d}(y)}{|x-y^*|^{n-2}}\right)$. For this aim, we split the region $B^{n-1}(0,r)$ as 
\begin{equation*}
B^{n-1} (0,r) = A_1 + A_2 + A_3,
\end{equation*}
where
\begin{equation*}
\begin{aligned}
A_1 &= \left\{ z\in B^{n-1} (0,r): |x-w(z)| \leq \frac{|x-y^*|}{2}\right\}, 
\\
A_2 &= \left\{z\in B^{n-1} (0,r): |w(z)- y^*| \leq \frac{|x-y^*|}{2}\right\},
\\
A_2&= \left\{ z\in  B^{n-1} (0,r):|x-w(z)| \geq \frac{|x-y^*|}{2}, |w(z)-y^*| \geq \frac{|x-y^*|}{2}\right\}.
\end{aligned}
\end{equation*}
For any $z \in B^{n-1} (0,r)$, we have the triangle inequality 
\begin{equation}\label{eq-h-7}
|x-w(z)| + |y^* - w(z)| \geq |x-y^*|.
\end{equation}
For $z \in A_1$ we see from \eqref{eq-h-7} that $|y^* - w(z) | \geq \frac{1}{2}|x-y^*|$. Thus,
\begin{equation*}
\int_{A_1}\frac{\mathbf{d}(x)}{|x-w(z)|^{n}} \frac{\mathbf{d}(y)}{|w(z) -y^*|^{n-2}} dS_z \leq \frac{\mathbf{d}(y)}{|x-y^*|^{n-2}} \int_{A_1} \frac{\mathbf{d}(x)}{|x-w(z)|^{n}} dz \leq \frac{C\mathbf{d}(y)}{|x-y^*|^{n-2}}.
\end{equation*}
For $z \in A_2$ we see from \eqref{eq-h-7} that $|x-w(z)| \geq \frac{1}{2} |x-y^*|$. Using this we find
\begin{equation*}
\begin{split}
\int_{A_2} \frac{\mathbf{d}(x)}{|x-w(z)|^{n}} \frac{\mathbf{d}(y)}{|w(z) -y^*|^{n-2}} dS_z &\leq \frac{C \mathbf{d}(z) \mathbf{d}(y)}{|x-y^*|^{n}} \int_{A_2} \frac{1}{|w(z)-y^*|^{n-2}} dz 
\\
&\leq \frac{C \mathbf{d}(x) \mathbf{d}(y)}{|x-y^*|^{n-1}}.
\end{split}
\end{equation*}

For $z \in A_3$ we have
\begin{equation*}
|y^* - w(z)| = \frac{|y^* - w(z)|}{2}  + \frac{|y^* - w(z)|}{2} \geq \frac{|x-y^*| }{4} + \frac{|y^* - w(z)|}{2}  \geq \frac{|x-w(z)|}{4}. 
\end{equation*}
Using this we get
\begin{equation*}
\begin{split}
\int_{A_3} \frac{\mathbf{d}(x)}{|x-w(z)|^{n}} \frac{\mathbf{d}(y)}{|w(z) -y^*|^{n-2}} dS_z& \leq C \mathbf{d}(x) \mathbf{d}(y) \int_{|x-w(z)|\geq \frac{|x-y^*|}{2}} \frac{1}{|x-w(z)|^{2n-2}} dz
\\
&\leq  C\frac{\mathbf{d}(x) \mathbf{d}(y)}{|x-y^*|^{n-1}}.
\end{split}
\end{equation*}
The above estimates along with that $\mathbf{d}(x) \leq |x-y^*|$ yields that 
\begin{equation*}
I_1 \leq  \frac{C\mathbf{d}(y)}{|x-y^*|^{n-2}}.
\end{equation*}
Combining this and \eqref{eq-h-10} we get the desired estimate \eqref{a-eq-h1-4}. To get the second estimate \eqref{a-eq-h1-5}, we note that the function $a \rightarrow H(a,y) - \frac{c_n}{|a-y^*|^{n-2}}$ is a harmonic on $B(x, \mathbf{d}(x))$. Hence we may use the standard regularity theory and deduce from \eqref{eq-h1-4} that
\begin{equation*}
\frac{\partial H}{\partial x_j}(x,y) + c_n (n-2) \frac{(x-y^*)_j}{|x-y^*|^{n}}= \frac{1}{\mathbf{d}(x)}\cdot O \left( \frac{\mathbf{d}(y)}{|x-y^*|^{n-2}}\right).
\end{equation*}
The lemma is proved.
\end{proof}

\section*{Acknowledgements}\thispagestyle{empty}
The author is grateful to the financial support
from POSCO TJ Park Foundation. 

\end{document}